\documentclass[10pt, a4paper,reqno]{amsart}
\usepackage{latexsym}
\usepackage{amsmath}
\usepackage[dvips]{graphicx}%
\usepackage{amsfonts}%
\usepackage{amssymb}
\usepackage{color}

\newcommand{\tH}{\tilde H}
\newcommand{\beq}{\begin{equation}}
\newcommand{\eeq}{\end{equation}}

\newcommand{\dloc}{d_{\rm loc}}

\def\R{\mathbb R}
\def\N{\mathbb N}
\def\C{\mathbb C}
\def\Z{\mathbb Z}
\def\cal{\mathcal}

\def\H{{\cal H}}

\def\de{\delta}
\def\e{\varepsilon}

\def\vphi{\varphi}

\newcommand{\medint}{-\kern -,375cm\int}
\newcommand{\medintinrigo}{-\kern -,315cm\int}

\def\Om{\Omega}
\newcommand{\wto}{\rightharpoonup}
\def\pa{\partial}

\def\loc{{\rm loc}}
\def\Div{{\rm div}}

\numberwithin{equation}{section}
\newtheorem{theorem}{Theorem}[section]

\newtheorem{corollary}[theorem]{Corollary}

\newtheorem{lemma}[theorem]{Lemma}

\newtheorem{proposition}[theorem]{Proposition}

\theoremstyle{definition}
\begingroup
\newtheorem{definition}[theorem]{Definition}
\newtheorem{remark}[theorem]{Remark}

\endgroup

\begin{document}

\title[Motion of films]{Motion of three-dimensional elastic films by anisotropic surface diffusion with curvature regularization}
\author{I. Fonseca, N. Fusco, G. Leoni, M. Morini}
\address[I.\ Fonseca]{Department of Mathematical Sciences, Carnegie Mellon University, Pittsburgh, PA, U.S.A.}
\email[I.\ Fonseca]{fonseca@andrew.cmu.edu}
\address[N.\ Fusco]{Dipartimento di Matematica e Applicazioni "R. Caccioppoli",
Universit\`{a} degli Studi di Napoli "Federico II" , Napoli, Italy}
\email[N.\ Fusco]{n.fusco@unina.it}
\address[G. \ Leoni]{Department of Mathematical Sciences, Carnegie Mellon University, Pittsburgh, PA, U.S.A.}
\email[G. \ Leoni]{giovanni@andrew.cmu.edu}
\address[M.\ Morini]{Dipartimento di Matematica,
Universit\`{a} degli Studi di Parma , Parma, Italy}
\email[M.\ Morini]{massimiliano.morini@unipr.it}

\begin{abstract}
Short time existence for a surface diffusion evolution equation with curvature regularization is proved  in the context of epitaxially strained three-dimensional films.  This is achieved by implementing a  minimizing movement scheme, which is hinged on the   $H^{-1}$-gradient  flow structure 
underpinning the evolution law. Long-time behavior and Liapunov stability  in the case of initial data close to a flat configuration are also addressed.
\end{abstract}

\maketitle

{\small

}

\tableofcontents

\section{Introduction}

In this paper we study  the morphologic evolution  of anisotropic epitaxially strained films, driven by stress and surface mass transport in three dimensions.  This can be viewed as the evolutionary counterpart of the static theory developed in \cite{BC, FFLM, FM09, FFLV, Bo0, CJP} in the two-dimensional case and in \cite{Bo} in three dimensions. The two dimensional formulation of the same evolution problem has been addressed in \cite{FFLM2} (see also \cite{Pi} for the case of motion by evaporation-condensation).

The physical setting behind the evolution equation is the following. The free interface is allowed to evolve via {\em surface  diffusion} under the influence of a chemical potential $\mu$. Assuming that mass transport in the bulk occurs at a much faster time scale, and thus can be neglected (see \cite{Mu63}), we have, according to the Einstein-Nernst relation,  that the evolution is governed by the {\em volume preserving} equation
\beq\label{i1}
V=C\Delta_{_\Gamma}\mu\,, 
\eeq
where $C>0$,  $V$ denotes the normal velocity of the evolving interface $\Gamma$, $\Delta_{_\Gamma}$ stands for the tangential laplacian, and the chemical potential $\mu$ is given by the first variation of the underlying free-energy functional. 

In our case, the free energy functional associated with the physical system is given by
\beq\label{i0}
\int_{\Om_h} W(E(u))\, dz+\int_{\Gamma_h}\psi(\nu)\, d\H^2\,,
\eeq
where $h$ is the function whose graph $\Gamma_h$ describes the evolving profile of the film,  $\Om_h$ is the region occupied by the film, $u$ is displacement of the material, which is assumed to be in (quasistatic) elastic equilibrium at each time, $E(u)$ is the symmetric part of $Du$, $W$ is a positive definite quadratic form, and $\H^2$ denotes the two-dimensional Hausdorff measure. Finally, 
$\psi$ is an anisotropic surface energy density, evaluated at  the unit normal $\nu$ to $\Gamma_h$. 
  The first variation of \eqref{i0} can be written as the sum of three contributions: A constant Lagrange multiplier related to mass conservation, the (anisotropic) curvature of the surface, and the elastic energy density evaluated at the displacement of the solid on the profile of the film. Hence, \eqref{i1} takes the form (assuming $C=1$)
  \beq\label{i2}
  V=\Delta_\Gamma\bigl[\Div_{\Gamma}(D\psi(\nu))+W(E(u))\bigr]\,, 
  \eeq
where $\Div_\Gamma$ stands for the tangential divergence along $\Gamma_{h(\cdot, t)}$, and 
$u(\cdot, t)$ is the elastic equilibrium in $\Om_{h(\cdot, t)}$, i.e., the minimizer of the elastic energy under the prescribed periodicity and boundary conditions (see \eqref{leiintro} below).

In the physically relevant case of a highly anisotropic non-convex interfacial energy there may exist certain directions $\nu$ at which the ellipticity condition
$$
D^2\psi(\nu)[\tau, \tau]>0\qquad\text{for all  $\tau\perp\nu$, $\tau\neq 0$}
$$
fails, see for instance \cite{dicarlo-gurtin-guidugli92, siegel-miksis-voorhees04}. Correspondingly, the above evolution equation  becomes {\em backward parabolic} and thus ill-posed.
To overcome this ill-posedness, and following the work of Herring (\cite{herring51}), an additive curvature regularization to surface energy has been proposed, see \cite{dicarlo-gurtin-guidugli92, GurJab}.
Here we consider the following regularized surface energy:
$$
\int_{\Gamma_h}\Bigl(\psi(\nu)+\frac\e{p}|H|^p\Bigr)\, d\H^{2}\,,
$$
where $p>2$, $H$ stands for the sum $\kappa_1+\kappa_2$  of the principal curvatures of $\Gamma_h$, and $\e$ is a (small) positive constant. 
The restriction on the range of exponents $p>2$ is of technical nature and it is motivated by the fact that in two-dimensions the Sobolev space $W^{2,p}$ embeds into $C^{1, \frac{p-2}{p}}$ if $p>2$. The extension of our analysis to the case $p=2$ seems to require different ideas. 

The regularized free-energy functional then reads 
\beq\label{i0rep}
\int_{\Om_h} W(E(u))\, dz+\int_{\Gamma_h}\Bigl(\psi(\nu)+\frac\e{p}|H|^p\Bigr)\, d\H^{2}\,,
\eeq
and \eqref{i1} becomes 
\begin{equation}
V=\Delta_\Gamma\left[\Div_{\Gamma}(D\psi(\nu))+W(E(u))-\e\Bigl(\Delta_{\Gamma}(|H|^{p-2}H)
- |H|^{p-2}H\Bigl(\kappa_1^2+\kappa_2^2-\frac1pH^2\Bigr)\Bigr)\right]\,. \label{sixth order evolution equation}
\end{equation}
Sixth-order evolution equations of this type have  already been  considered in \cite{GurJab} for the case without elasticity. Its two-dimensional version  was studied numerically in \cite{siegel-miksis-voorhees04} for the evolution of  voids in elastically stressed materials, and analytically in \cite{FFLM2} in the context of evolving one-dimensional graphs. We also refer to \cite{RRV, BHSV} and references therein for some numerical results in the three-dimensional case. However, to the best of our knowledge no analytical results were available in the literature prior to ours. 

As in \cite{FFLM2}, in this paper we focus on evolving graphs, and to be precise on the case where \eqref{sixth order evolution equation}  models the evolution toward equilibrium of epitaxially strained elastic films  deposited over a rigid substrate. Given $Q:=(0, b)^2$, $b>0$, we look for a spatially $Q$-periodic solution to the following Cauchy problem:
\beq\label{leiintro}
\begin{cases}
\displaystyle\frac 1J\frac{\partial h}{\partial t}=\Delta_\Gamma\left[\Div_{\Gamma}(D\psi(\nu))+W(E(u))
\vphantom{\frac1p}\right.\\
\displaystyle\qquad\qquad\qquad \left.-\e\Bigl(\Delta_{\Gamma}(|H|^{p-2}H)
- |H|^{p-2}H\Bigl(\kappa_1^2+\kappa_2^2-\frac1pH^2\Bigr)\Bigr)\right]\,,
& \text{in $\R^2\times (0, T_0)$,}\\
\Div\,\C E(u)=0 \quad \text{in $\Om_{h}$},\\
\C E(u)[\nu]=0 \quad \text{on $\Gamma_h$,} \qquad u(x ,0, t)=(e^1_0 x_1, e_0^2 x_2,0)\,,\\
h(\cdot, t) \text{ and }D u(\cdot,t)  \quad \text{ are $Q$-periodic,}\\
h(\cdot, 0)=h_0\,,
\end{cases}
\eeq
where, we recall, $h:\R^2\times [0,T_0]\to (0, +\infty)$ denotes the  function describing the two-dimensional profile $\Gamma_h$ of the film, $$
J:=\sqrt{1+|D_x h|^2}\,,\qquad\qquad 
$$  
 $W(A):=\frac12 \C A:A$  for all $A\in \mathbb{M}^{2\times2}_{\rm sym}$ with $\C$ a positive definite fourth
 order tensor,   $e_0:=(e_0^1, e_0^2)$, with $e_0^1$, $e_0^2>0$, is a vector  that embodies the mismatch between the crystalline lattices of the film and the substrate,  and $h_0\in H^2_{loc}(\R^2)$ is a $Q$-periodic function. 
 Note that in  \eqref{leiintro} the sixth-order (geometric) parabolic equation for the film profile is coupled with the elliptic system of elastic equilibrium equations in the bulk.  

It was  observed by  Cahn and Taylor in \cite{cahn-taylor94} that the surface diffusion equation can be regarded as a gradient flow of the  free-energy functional with respect to a suitable $H^{-1}$-Riemannian structure. To formally illustrate  this point, consider the manifold of subsets of $Q\times (0, +\infty)$ of fixed volume $d$, which are subgraphs of a $Q$-periodic function, that is, 
$$
\mathcal{M}:=\Bigl\{\Om_h:\, h\text{ $Q$-periodic, } \int_Q h\, dx=d\Bigr\}\,,
$$
where  $\Om_h:=\{(x,y):\, x\in Q\,, 0<y<h(x)\}$. 
The tangent space $T_{\Om_h}\mathcal M$ at an element $\Om_h$ is described by the kinematically admissible normal velocities
$$
T_{\Om_h}\mathcal M:=\Bigl\{V: \Gamma_h\to \R:\, V\text{is $Q$-periodic, }\int_{\Gamma_h}V\, d\H^2=0\Bigr\}\,,
$$
where $\Gamma_h$ is the graph of $h$ over the periodicity cell $Q$, and it is endowed with the $H^{-1}$ metric tensor
$$
g_{\Om_h}(V_1, V_2):=\int_{\Gamma_h}\nabla_{\Gamma_h} w_1\nabla_{\Gamma_h} w_2\, d\H^2 \qquad\text{for all }V_1,\, V_2\in T_{\Om_h}\mathcal M\,,
$$
where $w_i$, $i=1,2$, is the  solution to
$$
\begin{cases}
-\Delta_{\Gamma_h}w_i=V_i & \text{on $\Gamma_h$,}\\
w_i \text{ is $Q$-periodic,}\\
\displaystyle\int_{\Gamma_h}w_i\, d\H^2=0\,.
\end{cases}
$$
Consider now the following {\em reduced free-energy functional}
$$
G(\Om_h):=\int_{\Om_h} W(E(u_h))\, dz+\int_{\Gamma_h}\Bigl(\psi(\nu)+\frac\e{p}|H|^p\Bigr)\, d\H^{2}\,,
$$
 where  $u_h$ is the minimizer of the elastic energy in $\Om_h$ under the boundary and periodicity conditions described above. Then, the evolution described by  \eqref{leiintro} is such that at each time the normal velocity $V$ of the evolving profile $h(t)$ is the element of the tangent space 
 $T_{\Om_{h(t)}}\mathcal M$ corresponding to the steepest descent of $G$, i.e., \eqref{leiintro} may be formally rewritten as
 $$
 g_{\Om_{h(t)}}(V, \tilde V)=-\partial G(\Om_{h(t)})[\tilde V]\qquad \text{for all $\tilde V\in T_{\Om_{h(t)}}\mathcal M$,}
$$
 where $\partial G(h(t))[\tilde V]$ stands for the first variation of $G$ at $\Om_{h(t)}$ in the direction $\tilde V$.

In order to solve \eqref{leiintro}, we take advantage of  this gradient flow structure and we implement a  {\em minimizing movements scheme} (see \cite{ambrosio95}), which consists in constructing  discrete time evolutions by solving iteratively suitable minimum incremental problems.  

It is interesting to observe that the gradient flow of the free-energy functional $G$ with respect to an $L^2$-Riemannian structure, (instead of $H^{-1}$)  leads to a fourth order evolution equation, which describes motion  by evaporation-condensation (see \cite{cahn-taylor94,GurJab} and \cite{Pi}, where the  two-dimensional case was studied analytically).

 
This paper is organized as follows.  
 In Section~\ref{sec:incremental}  we set up  the problem  and introduce the discrete  time  evolutions. In Section~\ref{sec:existence} we prove our main local-in-time existence result for \eqref{leiintro}, by showing that (up to subsequences) the discrete time evolutions converge to a weak solution of \eqref{leiintro}  in $[0, T_0]$ for some 
 $T_0>0$ (see Theorem~\ref{th:existence}). By a {\em  $Q$-periodic weak solution}  we mean a function  $h\in H^1(0,T_0; H^{-1}_{\#}(Q))\cap L^{\infty}(0,T_0; H^2_{\#}(Q))$,  such that  $(h, u_h)$ satisfies the system \eqref{leiintro} in the distributional sense (see Definition~\ref{def:weaksol}).
 To the best of our knowledge,  Theorem~\ref{th:existence}  is the first (short time) existence result for a surface diffusion type geometric evolution equation in the presence of elasticity in three-dimensions.
 Moreover, also the use of minimizing movements  appears to be new in the context of higher order geometric flows (the only other papers we are aware of in which a similar approach is adopted, but in two-dimensions,  are \cite{FFLM2} and \cite{Pi}). 
 
 Compared to mean curvature flows, where the minimizing movements algorithm is nowadays classical after the pioneering work of \cite{ATW} (see also \cite{C, BCCN, CC}), a major technical difference lies in the fact that no comparison principle is available in this higher order framework. The convergence analysis is instead based on subtle interpolation and regularity estimates.
It is worth mentioning that for geometric surface diffusion equation without elasticity and without curvature regularization 
$$
V=\Delta_{\Gamma}H
$$
(corresponding to the case $W=0$, $\psi=1$, and $\e=0$) short time existence of a smooth solution was proved in \cite{EMS},  using semigroup
 techniques. See also \cite{BMN, mantegazza}. It is still an open question whether the solution constructed via the minimizing movement scheme is unique, and thus independent of the subsequence.
 
 
In Section~\ref{sec:liapunov} we address the Liapunov stability of the flat configuration, corresponding to an horizontal (flat) profile. Roughly speaking, we show that if the surface energy density is strictly convex and the second variation of  the functional \eqref{i0}
 at a given flat configuration is positive definite, then such a  configuration is asymptotically stable, that is, for all initial data $h_0$  sufficiently close to it the corresponding evolutions constructed via minimizing movements exist for all times, and  converge asymptotically to the flat configuration as $t\to+\infty$ (see Theorem~\ref{th:boli}).
We remark that Theorem~\ref{th:boli} may be regarded as an evolutionary counterpart of the static stability analysis of the flat configuration performed 
in \cite{FM09, Bo0, Bo}. In Theorem~\ref{th:bonacinievol} we address also the case of a non-convex anisotropy and we show that if the corresponding Wulff shape contains an horizontal facet, then the Asaro-Grinfeld-Tiller instability does not occur and the flat configuration is {\em always} Liapunov stable (see \cite{Bo0, Bo} for the corresponding result in the static case). Both results are completely new even in the two-dimensional case, to which they obviously apply 
(see Subsection~\ref{subsec:2d}). We remark that our treatment is purely variational and it is hinged on the fact that \eqref{i0rep} is a Liapunov functional for the evolution. 

Finally, in the  Appendix, we collect several auxiliary results that are used troughout the paper.

\section{Setting of the problem}\label{sec:incremental}

Let $Q:=(0,b)^{2}\subset\R^{2}$, $b>0$, $p>2$, and let $h_0\in W^{2,p}_{\#}(Q)$  be a positive function, describing the initial profile of the film. We recall that $W^{2,p}_{\#}(Q)$ stands for the subspace of $W^{2,p}(Q)$ of all functions  whose $Q$-periodic extension belong to $W^{2,p}_{loc}(\R^2)$. 
Given $h\in W^{2,p}_{\#}(Q)$, with $h\geq 0$, we set 
$$
\Om_h:=\{(x, y)\in Q\times\R:\, 0<y<h(x)\} 
$$
and we denote by $\Gamma_h$ the graph of $h$ over $Q$. We will identify a function
$h\in  W^{2,p}_{\#}(Q)$ with its periodic extension to $\R^{2}$, and denote by $\Om^{\#}_h$ and $\Gamma_h^\#$ the open subgraph and the graph of such extension, respectively. Note that $\Om^{\#}_h$ is  the periodic extension of $\Om_h$.  Set
\begin{multline*}
LD_\#(\Om_h;\R^3){:=} \bigl\{u\in L^2_{\rm loc}(\Om^\#_h;\R^3):\, u(x,y)=u(x{+}bk,y)\text{ for }(x,y)\in \Om_h^\# \text{ and }k\in\Z^{2}\,, \\E(u)|_{{\Om_h}}\in L^2(\Om_h;\R^3)\bigr\}\,,
\end{multline*}
where $E(u):=\frac12(Du+D^T u)$, 
with $Du$ the distributional gradient of $u$ and $D^T u$ its transpose, is the strain of the displacement  $u$.  We  prescribe the  Dirichlet boundary condition 
$u(x,0)=w_0(x,0)$ for $x\in  Q$, with $w_0\in H^1(U\times (0,+\infty))$ for every bounded open subset $U\subset \R^2$ and such that $Dw_0(\cdot, y)$ is $Q$-periodic for a.e. $y>0$.  A typical choice is given by $w_0(x,y):=(e^1_0 x_1, e_0^2 x_2,0)$, where the vector $e_0:=(e_0^1, e_0^2)$, with $e_0^1$, $e_0^2>0$, embodies the mismatch between the crystalline lattices of film and substrate.  
Define 
\begin{align*}
X:=\Bigl\{(h,u):\,h\in W^{2,p}_{\#}(Q),\, h\geq 0,\, u:\Om_h^\#\to \R^3\,\, \text{\rm s.t. } u-w_0 & \in LD_\#(\Om_{h};\R^3)\,,\\
&\text{ and }u(x,0)=w_0\,\, \text{for all $x\in\R^{2}$}\Bigr\}\,.
\end{align*}
The elastic energy density 
$W:\mathbb{M}^{3\times 3}_{\rm sym}\to [0, +\infty)$ takes the form
 $$
W(A):=\frac{1}{2}\C A:A\,,
$$
with $\C$ a positive definite fourth-order tensor, so that  $W(A)>0$ for all $A\in \mathbb{M}^{3\times 3}_{\rm sym}\setminus\{0\}$.
Given $ h\in W^{2,p}_{\#}(Q),\, h\geq 0$, we denote by $u_h$ the corresponding elastic equilibrium in $\Om_h$, i.e., 
$$
u_h:=\operatorname*{argmin} \biggl\{\int_{\Om_h}W(E(u))\, dz:\, u\in w_0+LD_\#(\Om_{h};\R^3),\, u(x,0)=w_0(x,0)\biggr\}\,.
$$
Let $\psi:\R^3\to [0, +\infty)$ be a positively one-homogeneous function of class $C^2$ away from the origin.
Note that, in particular,
\beq\label{sotto}
\frac1c|\xi|\leq \psi(\xi)\leq c|\xi| \qquad\text{for all $\xi\in \R^3$}\,,
\eeq
for some  $c>0$.

 We now introduce the energy functional
\beq\label{EF}
F(h,u):=\int_{\Omega_h}W(E(u))\,dz+\int_{\Gamma_h}\Bigl(\psi(\nu)+\frac\e{p}|H|^p\Bigr)\, d\H^{2}\,,
\eeq
defined for all $(h,u)\in X$,
where $\nu$ is the outer unit normal to $\Om_h$,  $H=\mathrm{div}_{\Gamma_h}\nu$ denotes the sum of the principal curvatures of $\Gamma_h$, and $\e$ is a positive constant. In the sequel we will often use the fact that 
\begin{equation}
-\Div\biggl(\frac{Dh}{\sqrt{1+|Dh |^2}}\biggr)=H\qquad\text{in $Q$}\,,
\label{div eq}
\end{equation}
which, in turn, implies
\begin{equation}
\int_Q H\,dx=0\,.
\label{average 0}
\end{equation}

\begin{remark}  \label{notation} {\bf Notation:}  In the sequel we denote by $z$ a generic point in $Q\times\R$ and we write $z=(x,y)$ with $x\in Q$ and $y\in \R$. Moreover,  given  $g:\Gamma_h\to \R$, where $\Gamma_h$ is the graph of some function $h$ defined in $Q$, we denote by the same symbol $g$  the function  from $Q$ to $\R$ given by $x\mapsto g(x, h(x))$.  Consistently, $Dg$ will stand for the gradient of the function from $Q$ to $\R$ just defined.
\end{remark}  
%

\subsection{The incremental minimum problem}

In this subsection we introduce the incremental minimum problems that will be used to define the discrete time evolutions. 
As  standing assumption throughout this paper, we start  from an initial configuration $(h_0, u_0)\in X$, such that 
\beq\label{sa}
h_0\in W^{2,p}_\#(Q)\,,\qquad h_0>0\,, 
\eeq
and $u_0$ minimizes the elastic energy in $\Om_{h_0}$ among all $u$ with $(h_0, u)\in X$.

Fix  a sequence $\tau_n\searrow 0$ representing the discrete time increments.  For $i\in \N$ we define inductively $(h_{i,n}, u_{i,n})$ as the 
solution of the minimum problem
\begin{multline}\label{pin}
\min\Biggl\{F(h,u)+\frac{1}{2\tau_n}\int_{\Gamma_{i-1,n}}|D_{\Gamma_{i-1,n}}v_{h}|^2\,d\H^{2}:\, (h,u)\in X\,, \\ \|Dh\|_{L^\infty(Q)}\leq \Lambda_0\,, \int_Qh\, dx= \int_Qh_{0}\, dx\Biggr\}\,,
\end{multline}
where $\Gamma_{i-1,n}$ stands for $\Gamma_{h_{i-1,n}}$, $\Lambda_0$ is a positive constant such  that
{
 \begin{equation}\Lambda_0>\|h_0\|_{C^1_\#(Q)}\,,\label{Lambda0}
\end{equation} 
}
and  $v_{h}$ is the unique solution in $H^{1}_\#(\Gamma_{h_{i-1,n}})$ to the following problem:
\beq\label{vacca}
\begin{cases}
\displaystyle\Delta_{\Gamma_{i-1,n}}v_{h}=\frac{h-h_{i-1,n}}{\sqrt{1+|Dh_{i-1,n} |^2}}\circ\pi\,,\vspace{7pt}\\
\displaystyle\int_{\Gamma_{h_{i-1,n}}} v_h\, d\H^{2}=0\,,
\end{cases}
\eeq
where $\pi$ is the canonical projection $\pi(x,y)=x$.
For $x\in Q$ and  $(i-1)\tau_n\leq t\leq i\tau_n$, $i\in \N$,  we define the linear interpolation
\beq\label{hn}
h_n(x,t):=h_{i-1,n}(x)+\frac{1}{\tau_n}\bigl(t-(i-1)\tau_n\bigr)\bigl(h_{i,n}(x)-h_{i-1,n}(x)\bigr)\,,
\eeq
and we let $u_n(\cdot, t)$ be the {\em elastic equilibrium corresponding to $h_n(\cdot, t)$}, i.e.,
{
\begin{equation}
F(h_n(\cdot, t),u_n(\cdot, t))=\min_{(h_n(\cdot, t), u)\in X}F(h_n(\cdot, t),u)\,.\label{min Fn}
\end{equation} 
}

The remaining of this subsection is devoted to the proof of the existence of a minimizer for the minimum incremental problem \eqref{pin}.

\begin{theorem}\label{th:incremental}
The minimum problem \eqref{pin} admits a solution $(h_{i,n}, u_{i,n})\in X$.
\end{theorem}
\begin{proof}
Let $\{(h_k,u_k)\}\subset X$ be a minimizing sequence for \eqref{pin}. Let $H_k$ denote the sum of principal curvatures of $\Gamma_{h_k}$. Since the sequence $\{H_k\}$ is bounded in $L^p(Q)$ and $\|Dh_k\|_{L^\infty_\#(Q)}\leq \Lambda_0$, it follows from \eqref{div eq} and
 Lemma~\ref{lm:morini} that $\|h_k\|_{W^{2,p}_{\#}(Q)}\leq C$. Then, up to a subsequence (not relabelled), we may assume that $h_k\wto h$ weakly in $W^{2,p}_{\#}(Q)$, and thus strongly in $C_\#^{1,\alpha}(Q)$ for some $\alpha>0$. As a consequence, $H_k\wto H$ in $L^p(Q)$, where $H$ is the sum of the principal curvatures of $\Gamma_h$. In turn, the $L^p$-weak convergence of $\{H_k\}$ and the $C^1$-convergence of $\{h_k\}$ imply   by lower semicontinuity  that 
\beq\label{ire1}
\int_{\Gamma_h}\Bigl(\psi(\nu)+\frac\e{p}|H|^p\Bigr)\, d\H^{2}\leq \liminf_{k}
\int_{\Gamma_{h_k}}\Bigl(\psi(\nu)+\frac\e{p}|H_k|^p\Bigr)\, d\H^{2}\,.
\eeq
Moreover, we also have that $v_{h_k}\to v_h$ strongly in $H^1(\Gamma_{i-1, n})$, and thus  
\beq\label{ire2}
\lim_k\frac{1}{2\tau_n}\int_{\Gamma_{i-1,n}}|D_{\Gamma_{i-1,n}}v_{h_k}|^2\,d\H^{2}=\frac{1}{2\tau_n}\int_{\Gamma_{i-1,n}}|D_{\Gamma_{i-1,n}}v_{h}|^2\,d\H^{2}\,.
\eeq
Finally, since $\sup_k\int_{\Om_{h_k}}|Eu_k|^2\, dz<+\infty$, reasoning as in \cite[Proposition 2.2]{FFLM}, from the uniform convergence of $\{h_k\}$ to $h$ and Korn's inequality we conclude that there exists $u\in H^1_{loc}(\Om^\#_h;\R^3)$ such that $(h,u)\in X$ and, up to  a subsequence, $u_k\wto u$ weakly in 
$H^1_{loc}(\Om^\#_h;\R^3)$. Therefore, we have that 
$$
\int_{\Omega_h}W(E(u))\,dz\le \liminf_k\int_{\Omega_{h_k}}W(E(u_k))\,dz\,,
$$ 
which, together with \eqref{ire1} and \eqref{ire2}, allows us to conclude that $(h,u)$ is a minimizer.
\end{proof}

\section{Existence of the evolution}\label{sec:existence}
In this section we prove short time existence of a solution of the geometric evolution equation
\beq\label{geq}
V=\Delta_\Gamma\Bigl[\Div_{\Gamma}(D\psi(\nu))+W(E(u))-\e\Bigl(\Delta_{\Gamma}(|H|^{p-2}H)-
\frac1p|H|^pH+|H|^{p-2}H|B|^2\Bigr)\Bigr]\,,
\eeq
where $V$ denotes the outer normal velocity of $\Gamma_{h(\cdot, t)}$, $|B|^2$ is the sum of the squares of the principal curvatures of $\Gamma_{h(\cdot, t)}$,  $u(\cdot, t)$ is the elastic equilbrium in $\Om_{h(\cdot, t)}$, and $W(E(u))$ is the trace of $W(E(u(\cdot, t)))$ on 
$\Gamma_{h(\cdot, t)}$. In the sequel we denote by $H^{-1}_\#(Q)$ the dual space of $H^{1}_\#(Q)$. Note that if $f\in H^{1}_\#(Q)$, then
$\Delta f$ can be identified with the element of $H^{-1}_\#(Q)$ defined by
$$
\langle \Delta f, g\rangle:=-\int_Q  DfDg\, dx\qquad \text{for all }g\in H^{1}_\#(Q)\,.
$$
Moreover, a function $f\in L^2(Q)$ can be identified with the element of $H^{-1}_\#(Q)$ defined by
$$
\langle  f, g\rangle:=\int_Q  fg\, dx\qquad \text{for all }g\in H^{1}_\#(Q)\,.
$$
\begin{definition}\label{def:weaksol}
Let $T_0>0$. We say that $h\in L^{\infty}(0,T_0; W^{2,p}_\#(Q))\cap H^1(0,T_0; H^{-1}_\#(Q))$ is a  {\em solution of} \eqref{geq} in $[0,T_0]$ if
\begin{itemize}
\item[(i)] $\Div_{\Gamma}(D\psi(\nu))+W(E(u))-\e\Bigl(\Delta_{\Gamma}(|H|^{p-2}H)-
\frac1p|H|^pH+|H|^{p-2}H|B|^2\Bigr)\in L^2(0,T_0; H^1_\#(Q))$,
\item[(ii)] for a.e. $t\in (0,T_0)$
$$
\frac{1}{J}\frac{\partial h}{\partial t}=\Delta_\Gamma\Bigl[\Div_{\Gamma}(D\psi(\nu))+W(E(u))-\e\Bigl(\Delta_{\Gamma}(|H|^{p-2}H)-
\frac1p|H|^pH+|H|^{p-2}H|B|^2\Bigr)\Bigr]\quad\text{in $H^{-1}_\#(Q)$,}
$$
\end{itemize}
where $J:=\sqrt{1+|Dh|^2}$,  $u(\cdot, t)$ is the elastic equilbrium in $\Om_{h(\cdot, t)}$, and where we wrote  $\Gamma$ in place of $\Gamma_{h(\cdot, t)}$.
\end{definition}
\begin{remark}
An immediate consequence of the above definition is that the evolution is {\em volume preserving}, that is, 
$\int_Q h(x,t)\, dx=\int_Q h_0(x)\, dx$ for all $t\in[0,T_0]$. Indeed, for all $t_1, t_2\in [0, T_0]$ and for 
$\vphi\in H^1_\#(Q)$ we have
\begin{align*}
\int_Q[h(x, t_2)-h(x,t_1)]\vphi\, dx&= \int_{t_1}^{t_2}\Bigl\langle\frac{\partial h}{\partial t}(\cdot, t), \vphi   \Big\rangle\, dt\\
&=\int_{t_1}^{t_2}\Bigl\langle J \Delta_\Gamma\Bigl[\Div_{\Gamma}(D\psi(\nu))+W(E(u))\\
&\quad-\e\Bigl(\Delta_{\Gamma}(|H|^{p-2}H)-
\frac1p|H|^pH+|H|^{p-2}H|B|^2\Bigr)\Bigr], \vphi  \Big\rangle\, dt\\
&=-\int_{t_1}^{t_2}\int_\Gamma D_\Gamma\Bigl[\Div_{\Gamma}(D\psi(\nu))+W(E(u))\\
&\quad-\e\Bigl(\Delta_{\Gamma}(|H|^{p-2}H)-
\frac1p|H|^pH+|H|^{p-2}H|B|^2\Bigr)\Bigr] D_{\Gamma}(\vphi\circ\pi)\,d\H^2  dt\,.
\end{align*}
Choosing $\vphi=1$, we conclude that 
$$
\int_Q[h(x, t_2)-h(x,t_1)] \,dx=0\,.
$$
\end{remark}
\begin{remark}\label{rm:normah-1}
In the sequel, we consider the following equivalent norm on $H^{-1}_\#(Q)$. Given $\mu\in H^{-1}_\#(Q)$, we set
$$
\|\mu\|_{H^{-1}_\#(Q)}:=\sup\bigl\{\langle\mu, g\rangle:\, g\in H^{1}_\#(Q)\text{ s.t. $\bigl|\textstyle{\int_Q g\, dx}\bigr|+\|Dg\|_{L^2(Q)}\leq 1$}\bigr\}\,.
$$
Note that if $f\in L^2(Q)$, with $\int_Qf\, dx=0$, we have
$$
\|f\|_{H^{-1}_\#(Q)}=\|Dw\|_{L^2(Q)}\,,
$$
where $w\in H^{1}_\#(Q)$ is the unique periodic solution to the problem
\beq\label{vudoppio}
\begin{cases}
\Delta w=f  & \text{in $Q$,}\vspace{5pt}\\
\displaystyle\int_Q w\, dx=0\,.
\end{cases}
\eeq
To see this, first observe that since $\int_Qf\,dx=0$ we have
$$
\|f\|_{H^{-1}_\#(Q)}=\sup\biggl\{\int_Q f g\, dx:\, g\in H^{1}_\#(Q)\text{ s.t. $\textstyle{\int_Q g\, dx=0}$ and $\|Dg\|_{L^2(Q)}\leq 1$}\biggr\}\,.
$$
Thus, since by \eqref{vudoppio}
$$
\int_Q fg\, dx=-\int_QDwDg\, dx\leq\|Dw\|_{L^2(Q)}\,,
$$
we have $\|f\|_{H^{-1}_\#(Q)}\leq \|Dw\|_{L^2(Q)}$. The opposite inequality follows by taking $g=-w/\|Dw\|_{L^2(Q)}$.
\end{remark}
\begin{theorem}\label{th:pre-pippa}
For all $n$, $i\in\N$ we have
\beq\label{31}
\int_0^{+\infty}\Bigl\|\frac{\partial h_n}{\partial t}\Bigr\|^2_{H^{-1}_\#(Q)}dt\leq CF(h_0, u_0)\,,
\eeq
\beq\label{31.5}
F(h_{i,n}, u_{i,n})\leq F(h_{i-1,n}, u_{i-1,n}) \leq F(h_0, u_0)\,,
\eeq
and 
\beq\label{ghnboundbis}
\sup_{t\in[0,+\infty)}\|h_{n}(\cdot, t)\|_{W^{2,p}_\#(Q)}<+\infty
\eeq
for some $C=C(\Lambda_0)>0$.
 Moreover, up to a subsequence, 
\beq\label{32}
h_n\to h \text{ in $C^{0,\alpha}([0,T]; L^2(Q))$ for all $\alpha\in (0, \tfrac{1}{4})$,}
\quad h_n\wto h \text{ weakly in $H^1(0,T; H^{-1}_{\#}(Q))$}
\eeq
for all $T>0$ and for some function $h$ such that $h(\cdot, t)\in W^{2,p}_{\#}(Q)$ for every $t\in [0, +\infty)$ and 
\beq\label{32.5}
F(h(\cdot, t), u_{h(\cdot, t)})\leq F(h_0, u_0) \qquad \text{for all $t\in [0,+\infty)$.}
\eeq
\end{theorem}

\begin{proof}
By the minimality of $(h_{i,n}, u_{i,n})$ (see \eqref{pin}) we have that 
\beq\label{33}
F(h_{i,n},u_{i,n})+\frac{1}{2\tau_n}\int_{\Gamma_{i-1,n}}|D_{\Gamma_{i-1,n}}v_{h_{i,n}}|^2\,d\H^{2}\leq F(h_{i-1,n},u_{i-1,n})
\eeq
for all $i\in \N$, which yields in particular \eqref{31.5}. Hence, 
$$
\frac{1}{2\tau_n}\int_{\Gamma_{i-1,n}}|D_{\Gamma_{i-1,n}}v_{h_{i,n}}|^2\,d\H^{2}\leq 
F(h_{i-1,n},u_{i-1,n})-F(h_{i,n},u_{i,n})\,,
$$
and summing over $i$, we obtain 
\beq\label{pre-pippa0}
\sum_{i=1}^\infty\frac{1}{2\tau_n}\int_{\Gamma_{i-1,n}}|D_{\Gamma_{i-1,n}}v_{h_{i,n}}|^2\,d\H^{2}\leq 
F(h_0, u_0)\,.
\eeq
Let  $w_{h_{i,n}}\in H^1_{\#}(Q)$ denote the unique periodic solution to the  problem
$$
\begin{cases}
\Delta w_{h_{i,n}}=h_{i,n}-h_{i-1,n}  & \text{in $Q$,}\vspace{5pt}\\
\displaystyle\int_Q w_{h_{i,n}}\, dx=0\,.
\end{cases}
$$
Note that 
\begin{align*}
\int_Q|Dw_{h_{i,n}}|^2\, dx&=\int_Q\Delta w_{h_{i,n}}w_{h_{i,n}}\, dx=
\int_{\Gamma_{i-1, n}}\frac{h_{i,n}-h_{i-1,n}}{\sqrt{1+|Dh_{i-1,n} |^2}}\circ\pi w_{h_{i,n}}\, d\H^2\\
&=
\int_{\Gamma_{i-1, n}}\Delta_{\Gamma_{i-1,n}}v_{h_{i,n}} w_{h_{i,n}}\, d\H^2 
=-\int_{\Gamma_{i-1, n}}D_{\Gamma_{i-1,n}}v_{h_{i,n}} D_{\Gamma_{i-1,n}}w_{h_{i,n}}\, d\H^2\\
& \leq \|D_{\Gamma_{i-1,n}}v_{h_{i,n}} \|_{L^2(\Gamma_{i-1,n})}\|D_{\Gamma_{i-1,n}}w_{h_{i,n}} \|_{L^2(\Gamma_{i-1,n})}\\
&\leq C(\Lambda_0) \|D_{\Gamma_{i-1,n}}v_{h_{i,n}} \|_{L^2(\Gamma_{i-1,n})}\|Dw_{h_{i,n}} \|_{L^2(Q)}\,.
\end{align*}
Combining this inequality with \eqref{pre-pippa0} and recalling \eqref{hn} and Remark~\ref{rm:normah-1}, we get \eqref{31}. 

Note from \eqref{31.5} it follows that 
$$
\sup_{i,n}\int_{\Gamma_{i,n}}|H|^p\, d\H^{2}<+\infty\,.
$$
 Hence, \eqref{ghnboundbis} follows immediately by Lemma~\ref{lm:morini}, taking into account that $\|Dh_{i,n}\|_{L^\infty(Q)}\leq \Lambda_0$. Using a diagonalizing argument, it can be shown that there exist $h$ such that $h_n\wto h$ weakly in $H^1(0,T; H^{-1}_\#(Q))$ for all $T>0$. 
Note also that, by \eqref{31} and using H\"older Inequality, we have for $t_2>t_1$,
\beq\label{pre-pippa2}
\|h_n(\cdot, t_2)-h_n(\cdot, t_1)\|_{H^{-1}(Q)}  \leq  \int_{t_1}^{t_2}\Bigl\|\frac{\partial h_n(\cdot, t)}{\partial t}\Bigr\|_{H^{-1}(Q)}\, dt \leq C (t_2-t_1)^{\frac12}\,.
\eeq
Therefore, applying Theorem~\ref{th:A} to the solution  $w\in H^1_\#(Q)$ of the problem
$$
\begin{cases}
\Delta w=h_n(\cdot, t_2)-h_n(\cdot, t_1)  & \text{in $Q$,}\vspace{5pt}\\
\displaystyle\int_Q w\, dx=0\,,
\end{cases}
$$
we get 
\begin{align}
\|h_n(\cdot, t_2)-h_n(\cdot, t_1)\|_{L^2(Q)}& =\|\Delta w\|_{L^2(Q)} \leq C\|D^3 w\|^{\frac12}_{L^2(Q)}\|Dw\|^{\frac12}_{L^2(Q)}\nonumber\\
&\leq C\|Dh(\cdot, t_2)-Dh(\cdot, t_1)\|^{\frac12}_{L^2(Q)}\|h(\cdot, t_2)-h(\cdot, t_1)\|^{\frac12}_{H^{-1}(Q)}\nonumber\\
&\leq C(\Lambda_0)(t_2-t_1)^{\frac14}\,, \label{pre-pippa-final}
\end{align}
where the last inequality follows from \eqref{pre-pippa2}. By the Ascoli-Arzel\`a theorem (see e.g. \cite[Proposition 3.3.1]{AGS}), we get \eqref{32}.
Finally, inequality \eqref{32.5} follows from \eqref{31.5} by lower semicontinuity, using \eqref{32} and \eqref{ghnboundbis}.
\end{proof}

In what follows, $\{h_n\}$ and $h$ are   the subsequence and the function found in Theorem~\ref{th:pre-pippa}, respectively. The next result shows that the convergence of $\{h_n\}$ to $h$ can be significantly improved for short time. 
\begin{theorem}\label{th:pippa}
There exist $T_0>0$ and $C>0$ depending only $(h_0, u_0$) such that
\begin{itemize}
\item[(i)] $h_n\to h$ in $C^{0,\beta}([0, T_0]; C^{1,\alpha}_\#(Q))$ for every $\alpha\in (0, \tfrac{p-2}p)$ and
$\beta\in (0, \frac{(p-2-\alpha p)(p+2)}{16p^2})$,\vspace{4pt}
\item[(ii)] $\displaystyle \sup_{t\in[0, T_0]}\|Du_n(\cdot, t)\|_{C^{0,\frac{p-2}{p}}(\overline \Om_{h_n(\cdot, t)})}\leq C$,
\item[(iii)]  $E(u_n(\cdot ,h_n))\to  E(u(\cdot,h))$ in $C^{0,\beta}([0, T_0]; C^{0,\alpha}_\#(Q))$ for every $\alpha\in (0, \tfrac{p-2}p)$ and
$0\leq \beta<\frac{(p-2-\alpha p)(p+2)}{16p^2}$, where
$u(\cdot, t)$ is the elastic equilibrium in $\Om_{h(\cdot, t)}$.
\end{itemize}
Moreover,  $h(\cdot, t)\to h_0$ in $C^{1, \alpha}_\#(Q)$ as $t\to 0^+$,   
$h_n$, $h\geq C_0>0$ for some positive constant $C_0$, and
\beq\label{noconstraint}
 \sup_{t\in[0,T_0]}\|Dh_n(\cdot,t)\|_{L^{\infty}(Q)}<\Lambda_0 
 \eeq
 for all $n$. 
\end{theorem}
\begin{proof}
To prove assertion (i), we start by observing that by Theorem~\ref{th:D}, 
\eqref{ghnboundbis},  Theorem~\ref{th:D} again, and  \eqref{pre-pippa-final} we have
\begin{align}
\Bigl\|{D h_n}(\cdot, t_2) -{D h_n}(\cdot, t_1) \Bigr\|_{L^\infty} & \leq 
C\Bigl\|{D^2 h_n} (\cdot, t_2)-{D^2 h_n} (\cdot, t_1)\Bigr\|_{L^p}^{\frac{p+2}{2p}}
\| h_n(\cdot, t_2)-h_n(\cdot, t_1)\Bigr\|_{L^p}^{\frac{p-2}{2p}}\nonumber\\
&\leq C \| h_n(\cdot, t_2)-h_n(\cdot, t_1)\Bigr\|_{L^p}^{\frac{p-2}{2p}}\nonumber\\
&\leq C \biggl(\|{D^2 h_n} (\cdot, t_2)-{D^2 h_n} (\cdot, t_1)\Bigr\|_{L^2}^{\frac{p-2}{2p}}
\| h_n(\cdot, t_2)-h_n(\cdot, t_1)\Bigr\|_{L^2}^{\frac{p+2}{2p}}\biggr)^{\frac{p-2}{2p}}\nonumber\\
&\leq C|t_2-t_1|^{\frac{p^2-4}{16p^2}} 
\label{44}
\end{align}
for all $t_1$, $t_2\in [0, T_0]$. Notice that from \eqref{ghnboundbis}  we have 
\beq\label{45}
\sup_{n, t\in [0, T_0]}\|h_n(\cdot, t)\|_{C^{1, \frac{p-2}p}_\#(Q)}<+\infty\,.
\eeq
Take $\alpha\in (0,\frac{p-2}p)$ and observe that
$$
\Bigl[Dh_n(\cdot, t_2)- Dh_n(\cdot, t_1)\Bigr]_\alpha\leq \Bigl[Dh_n(\cdot, t_2)-Dh_n(\cdot, t_1)\Bigr]_{\frac{p-2}{p}}^{\frac{\alpha p}{p-2}}
\biggl[\,\operatorname*{osc}_{[0,b]}\Bigl(D h_n(\cdot, t_2)-D h_n(\cdot, t_1)\Bigr)\biggr]^{\frac{p-2-\alpha p}{p-2}}\,,
$$
where $[\cdot]_\beta$ denotes the $\beta$-H\"older seminorm.
From this inequality,  \eqref{44},  \eqref{45}, and the Ascoli-Arzel\`a theorem  \cite[Proposition 3.3.1]{AGS},  assertion (i) follows. 

Standard elliptic estimates  ensure that if $h_n(\cdot, t)\in C^{1,\alpha}_\#(Q)$ for some $\alpha\in (0,1)$, then $Du_n(\cdot, t)$ can be estimated in $C^{0, \alpha}(\overline \Om_{h_n(\cdot, t)})$ with a constant depending only on the $C^{1, \alpha}$-norm of $h_n(\cdot, t)$, see, for instance, \cite[Proposition 8.9]{FM09}, where this property is proved in two dimensions but an entirely similar argument works in all dimensions. Hence, assertion (ii) follows from \eqref{45}.
Assertion (iii) is an immediate consequence of (i) and Lemma~\ref{lm:chepallequadrate}. Finally, \eqref{noconstraint} follows from \eqref{Lambda0} and (i).
\end{proof}
\begin{remark}\label{rm:elle0}
Note that in the previous theorem we can take 
$$
T_0:=\sup\{t>0:\, \|Dh_n(\cdot,s)\|_{L^{\infty}(Q)}<\Lambda_0 \quad\text{for all }s\in [0,t)\}\,.
$$
 In Theorem~\ref{th:existence} we will show that $h$ is a solution to \eqref{geq} in $[0,T_0)$, in the sense of Definition~\ref{def:weaksol}.
\end{remark}
We begin with some auxiliary results.
\begin{proposition}\label{prop:eulero}
Let $h\in W^{3,q}_{\#}(Q)$ for some $q>2$ and let $\Gamma$ be its graph. Let $\Phi:Q\times \R\times(-1,1)\to Q\times\R$ be the flow 
$$
\frac{\partial \Phi}{\partial t}=X(\Phi), \qquad \Phi(\cdot, 0)=Id\,, 
$$
where $X$ is a smooth vector field $Q$-periodic in the first two variables.
Set $\Gamma_t:=\Phi(\cdot, t)(\Gamma)$,  denote by $\nu_t$ the normal to $\Gamma_t$,   let $H_t$ be the sum of principal curvatures of $\Gamma_t$, and let $|B_t|^2$ be the sum of squares of the principal curvatures of $\Gamma_t$. Then
\begin{multline}\label{eq:eulero}
\frac{d}{dt}\frac 1p\int_{\Gamma_t}|H_t|^p\, d\H^2=
\int_{\Gamma_t}D_{\Gamma_t}(|H_t|^{p-2}H_t)D_{\Gamma_t}(X\cdot\nu_t)\, d\H^2\\
-\int_{\Gamma_t}|H_t|^{p-2}H_t\Bigl(|B_t|^2-\frac1pH_t^2\Bigr)(X\cdot \nu_t)\, d\H^2\,.
\end{multline}
\end{proposition}
\begin{proof}
Set $\Phi_t(\cdot):=\Phi(\cdot, t)$. 
We can extend $\nu_t$ to a tubular neighborhood of $\Gamma_t$ as the gradient of the signed distance from $\Gamma_t$.
We have
$$
\frac{d}{dt}\frac 1p\int_{\Gamma_t}|H_t|^p\, d\H^2 = \frac{d}{ds}\biggl(\frac 1p\int_{\Gamma_{t+s}}|H_{t+s}|^p\, d\H^2\biggr)_{\bigr|_{s=0}}=\frac{d}{ds}\biggl(\frac 1p
\int_{\Gamma_{t}}|H_{t+s}\circ\Phi_{s}|^p J_2\Phi_{s}\, d\H^2\biggr)_{\bigr|_{s=0}}\,,
$$
where $J_2$ denotes the two-dimensional Jacobian of $\Phi_s$ on $\Gamma_t$. Then we have
$$
\frac{d}{dt}\frac 1p\int_{\Gamma_t}|H_t|^p\, d\H^2  =
\frac 1p\int_{\Gamma_t}|H_t|^p \Div_{\Gamma_t}X\, d\H^2+ 
\int_{\Gamma_t}|H_t|^{p-2}H_t\frac{d}{ds}\bigl(H_{t+s}\circ\Phi_s\bigr)_{\bigr|_{s=0}}\, d\H^2\,.
$$
Concerning the last integral, we observe that 
$$
\frac{d}{ds}\bigl(H_{t+s}\circ\Phi_s\bigr)_{\bigr|_{s=0}}=
\frac{d}{ds}\Bigl(\Div_{\Gamma_{t+s}}\nu_{t+s}\Bigr)_{\bigr|_{s=0}}+DH_{t}\cdot X\,.
$$
Set
$$
\dot{\nu_t}:=\frac{d}{ds}{\nu_{t+s}}_{\bigr|_{s=0}}\,.
$$
By differentiating with respect to $s$ the identity $D\nu_{t+s}[\nu_{t+s}]=0$, we get $D\dot{\nu_t}[\nu_t]+D\nu_t[\dot{\nu_t}]=0$. Multiplying this identity by $\nu_t$ and recalling that $D\nu$ is symmetric matrix we get 
$$
D\dot{\nu_t}[\nu_t]\cdot \nu_t=-D\nu_t[\nu_t]\cdot \dot{\nu_t}=0\,.
$$
In turn, this implies that $\Div_{\Gamma_t}\dot{\nu_t}=\Div \dot{\nu_t}$, and so
 $$
\frac{d}{ds}\Bigl(\Div_{\Gamma_{t+s}}\nu_{t+s}\Bigr)_{\bigr|_{s=0}}=\Div_{\Gamma_t}\dot{\nu_t}\,.
$$
In turn, see \cite[Lemma 3.8-(f)]{CMM},
$$
\dot{\nu_t}=-(D_{\Gamma_t}X)^T[\nu_t]-D_{\Gamma_t}\nu_t[X]=-D_{\Gamma_t}(X\cdot\nu_t)\,.
$$
Collecting the above identities, integrating by parts, and using the identity $\partial_{\nu_t}H_t=-\mathrm{trace}\bigl((D\nu_t)^2\bigr)=-|B_t|^2$ proved in \cite[Lemma 3.8-(d)]{CMM}, we have
\begin{align}
\frac{d}{dt}\frac 1p\int_{\Gamma_t}|H_t|^p\, d\H^2  &=\frac 1p\int_{\Gamma_t}|H_t|^p \Div_{\Gamma_t}X\, d\H^2+
\int_{\Gamma_t}|H_t|^{p-2}H_t\bigl(-\Delta_{\Gamma_t}(X\cdot\nu_t)+DH_{t}\cdot X\bigr)\, d\H^2\nonumber\\
& = -\int_{\Gamma_t}|H_t|^{p-2}H_tD_{\Gamma_t}H_t\cdot X\, d\H^2+
\frac 1p\int_{\Gamma_t}|H_t|^pH_t(X\cdot\nu_t)\, d\H^2\nonumber\\
&\quad+ \int_{\Gamma_t}|H_t|^{p-2}H_t\bigl(-\Delta_{\Gamma_t}(X\cdot\nu_t)+DH_{t}\cdot X\bigr)\, d\H^2\nonumber\\
&=  \int_{\Gamma_t}|H_t|^{p-2}H_t\Bigl(-\Delta_{\Gamma_t}(X\cdot\nu_t)
+\partial_{\nu_t}H_t(X\cdot\nu_t)+\frac{1}pH_t^2(X\cdot\nu_t)\Bigr)\, d\H^2\nonumber\\
&=\int_{\Gamma_t}D_{\Gamma_t}(|H_t|^{p-2}H_t)D_{\Gamma_t}(X\cdot\nu_t)\, d\H^2\nonumber\\
&\quad -\int_{\Gamma_t}|H_t|^{p-2}H_t\Bigl\{\Bigl(|B_t|^2-\frac1pH_t^2\Bigr)(X\cdot \nu_t)\Bigr\}\, d\H^2\,.
\end{align}
Thus \eqref{eq:eulero} follows. 
\end{proof}
Proposition~\ref{prop:eulero} motivates the following definition.
\begin{definition}\label{def:critical}
We say that $(h,u_h)\in X$ is a \emph{critical pair} for the functional $F$ defined in \eqref{EF} if $|H|^{p-2}H\in H^1(\Gamma_h)$ and 
\begin{multline*}
\e\int_{\Gamma_h}D_{\Gamma_h}(|H|^{p-2}H)D_{\Gamma_{h}}\phi\, d\H^2
+\e\int_{\Gamma_{h}}\Bigl(\frac1p|H|^pH-|H|^{p-2}H|B|^2\Bigr)\phi\, d\H^2\\
+\int_{\Gamma_{h}}\bigl[\Div_{\Gamma_{h}}(D\psi(\nu))
+W(E(u_{h}))\bigr]\phi\, d\H^2=0
\end{multline*}
for all $\phi\in H^1_\#(\Gamma_h)$ with $\int_{\Gamma_h}\phi\,d\mathcal{H}^2=0$.
We will also say that $h$ is a \emph{critical profile} if $(h, u_h)$ is a critical pair.
\end{definition}

\begin{lemma}\label{lm:tosto}
Let $h\in W^{2,p}_\#(Q)$ such that $|H|^{p-2}H\in W^{1,q}_\#(Q)$, for some $q>2$. Then, there exist a sequence 
$\{h_j\}\subset W^{3,q}_\#(Q)$ such that $h_j\to h$ in $W^{2,p}_\#(Q)$ and $|H_j|^{p-2}H_j\to |H|^{p-2}H$ in $W^{1,q}_\#(Q)$, where 
$H_j$ stands for the sum of the principal curvatures of $\Gamma_{h_j}$.
\end{lemma}
\begin{proof} We may assume without loss of generality that $H\neq 0$, otherwise $h$ would have already the required regularity (see \eqref{div eq}).
By the Sobolev embedding theorem it follows that $|H|^{p-2}H\in C^{0, 1-\frac2q}_\#(Q)$ and, in turn, using the $\frac1{p-1}$ H\"older's continuity of the function  $t\mapsto t^{\frac1{p-1}}$, $H\in C^{0, \alpha }_\#(Q)$ for $\alpha:=\frac{q-2}{q(p-1)}$. Standard Schauder's estimates yield $h\in C^{2, \alpha}_\#(Q)$.

For $\de>0$ set 
$$
H_\de:=
\begin{cases}
H-\de & \text{if $H\geq \de$,}\\
H+\de' & \text{if $H\leq-\de'$,}\\
0 & \text{otherwise ,}
\end{cases}
$$
where $\de'$ is chosen in such a way that $\int_Q H_\de\, dx=0$. Observe that this choice of $\de'$ is always possible, although not necessarily unique. Indeed, 
by \eqref{average 0} and the fact that $H\neq 0$, if $\de$ is sufficiently small
$$
\int_{\{H>\de\}}(H-\de)\, dx+\int_{\{H<0\}}H\, dx<0 \qquad\text{and}\qquad \int_{\{H>\de\}}(H-\de)\, dx>0\,.
$$
By continuity it is then clear that we may find $\de'>0$ such that 
\beq\label{zero}
\int_{\{H>\de\}}(H-\de)\, dx+\int_{\{H<-\de'\}}(H+\de')\, dx=0\,.
\eeq
We now show that, independently of the choice of $\de'$ satisfying \eqref{zero}, $\de'\to 0$ as $\de\to 0$. Indeed, if not, there would exist a sequence $\de_n\to 0$ and a corresponding sequence $\de'_n\to \de'>0$, such that \eqref{zero} holds with $\de$ and $\de'$ replaced by 
$\de_n$ and $\de'_n$, respectively. But then,  passing to the limit as $n\to\infty$, we would get
$$
\int_{\{H>0\}}H\, dx+\int_{\{H<-\de'\}}(H+\de')\, dx=0\,,
$$
which contradicts \eqref{average 0}.

Note that $H_\de\to H$ in $C^{0, \alpha }_\#(Q)$ as $\de\to 0$. Moreover,  we claim that $|H_\de|^{p-2}H_\de \to |H|^{p-2}H$ in $W^{1,q}_\#(Q)$. Indeed, observe that $H\in W^{1,q}(A_\de)$ where $A_\de:=\{H> \de\}\cup\{H<-\de'\}$  for all $\de>0$. Hence, 
$$
D(|H|^{p-2}H)=
\begin{cases}
(p-1)|H|^{p-2}DH & \text{if $H\neq 0$,}\\
0 & \text{elsewhere,}
\end{cases}
$$
and
$$
D(|H_\de|^{p-2}H_\de)=
\begin{cases}
(p-1)|H_\de|^{p-2}DH & \text{in  $A_\de$,}\\
0 & \text{elsewhere.}
\end{cases}
$$
The claim follows by observing that $D(|H_\de|^{p-2}H_\de)\to D(|H|^{p-2}H)$ a.e. and that 
$|D(|H_\de|^{p-2}H_\de)|\leq |D(|H|^{p-2}H)|$. Observe now that $H\in W^{1,q}(A_\de)$ implies $H_\de\in W^{1,q}_\#(Q)$. In order to conclude the proof it is enough to show that for $\de$ sufficiently small there exist a unique periodic solution $h_\de$ to the problem
\beq\label{tosto1}
\begin{cases}
\displaystyle-\Div\biggl(\frac{Dh_\de}{\sqrt{1+|Dh_\de|^2}}\biggr)=H_\de \vspace{10pt} \\
\vspace{10pt}
\displaystyle\int_Q h_\de\, dx =\int_Q h\, dx\,.
\end{cases}
\eeq
This follows from Lemma~\ref{lm:IFT} below. 
 \end{proof}
\begin{lemma}\label{lm:IFT}
Let $h\in C^{2,\alpha}_\#(Q)$ and let $H$ denote the sum of the principal curvatures of $\Gamma_{h}$. Then there exist $\sigma$, $C>0$ with the following property:   for all $K\in C^{0, \alpha}_\#(Q)$, with $\int_Q K\, dx=0$ and $\|K-H\|_{ C^{0, \alpha}_\#(Q)}\leq \sigma$, there exists a unique periodic solution $k\in C^{2,\alpha}_\#(Q)$ to 
$$
\begin{cases}
\displaystyle-\Div\biggl(\frac{Dk}{\sqrt{1+|Dk|^2}}\biggr)=K \vspace{10pt} \\
\vspace{10pt}
\displaystyle\int_Q k\, dx =\int_Q h\, dx\,,
\end{cases}
$$
and 
\beq\label{stimozza}
\|k-h\|_{C^{2,\alpha}_\#(Q)}\leq C\|K-H\|_{ C^{0, \alpha}_\#(Q)}\,.
\eeq
\end{lemma}
\begin{proof}
Without loss of generality we may assume that $\int_{Q}h\, dx=0$.

 Set $X:=\{k\in C^{2,\alpha}_\#(Q):\, \int_Q k\, dx=0 \}$ and
$Y:=\{K\in C^{0,\alpha}_\#(Q):\, \int_Q K\, dx=0\}$, and consider the operator $T\colon X\to Y$ defined by
$$
T(k):=-\Div\biggl(\frac{Dk}{\sqrt{1+|Dk|^2}}\biggr)\,.
$$ 
By assumption we have that $T(h)=H$. We now use the inverse function theorem (see e.g. \cite[Theorem 1.2, Chap. 2]{AP}) to prove that 
$T$ is invertible in a $C^{2,\alpha}$-neighborhood of $h$ with a $C^1$-inverse. To see this, note that for any $k\in X$ we have that
$T'(k): X\to Y$ is the continuous linear operator defined by
$$
T'(h)[\vphi]:=-\Div\biggl[\frac{1}{\sqrt{1+|Dh|^2}}\bigg(I-\frac{Dh\otimes Dh}{1+|Dh|^2}\biggr)D\vphi \biggr]\,.
$$
It is easily checked that $T'$ is continuous map from $X$ to the space $\mathcal L(X,Y)$ of linear bounded operators from $X$ to $Y$, so that $T\in C^1(X,Y)$.
Finally, standard existence arguments for elliptic equations imply that for any $k\in X$ the operator $T'(k)$ is invertible. Thus we may apply the inverse function theorem to conclude that there exist $\sigma>0$ such that  for all $K\in C^{0, \alpha}_\#(Q)$, with $\int_Q K\, dx=0$ and $\|K-H\|_{ C^{0, \alpha}_\#(Q)}\leq \sigma$, there exists a unique periodic function $k=T^{-1}K\in C^{2,\alpha}_\#(Q)$. Moreover, the continuity of $T^{-1}$, together with standard Schauder's estimates, implies that \eqref{stimozza} holds for $\sigma$ sufficiently small.
\end{proof}
In what follows $J_{i,n}$ stands for
$$
 J_{i,n}:=\sqrt{1+|Dh_{i,n}|^2}\,,
$$
 $H_{i,n}$ is the sum of the principal curvatures of $\Gamma_{i,n}$, $|B_{i,n}|^2$ denotes the sum of the squares of the principal curvatures of $\Gamma_{i,n}$, and  
 $\tilde H_n\colon Q\times [0,T_0]\to\R$ is the function  defined as
{
 \begin{equation}\tilde H_n(x,t):=H_{i,n}(x, h_{i,n}(x), t)\qquad \text{ if $t\in [(i-1)\tau_n, i\tau_n)$}\,.
\label{H tilde}
\end{equation} 
}
 \begin{theorem}\label{th:5}
Let $T_0$ be as in Theorem~\ref{th:pippa} and let $\tilde H_n$ be given in \eqref{H tilde}. Then there exists $C>0$ such that 
\beq\label{51}
\int_0^{T_0}\int_Q|D^2(|\tilde H_n|^{p-2}\tilde H_n)|^2\, dxdt\leq C
\eeq
for $n\in \N$.
\end{theorem}
\begin{proof}
{\bf Step 1.} We claim that $|H_{i,n}|^{p-2}H_{i,n}\in W^{1, q}_\#(\Gamma_{i,n})$ for all $q\geq 1$ and that $h_{i,n}\in C^{2, \sigma}_\#(Q)$ for all $\sigma\in (0, \frac1{p-1})$.

We recall that $h_{i,n}$ is the solution to the incremental minimum problem \eqref{pin}.
We are going to show that $h_{i,n}\in W^{2,q}_{\#}(Q)$ for all $q\geq 2$. Fix a function $\vphi\in C^2_\#(Q)$ such that $\int_Q\vphi\, dx=0$. Then by minimality  and by \eqref{noconstraint} we have
$$
\frac{d}{ds}\biggl(F(h_{i,n}+s\vphi, u_{i,n})+\frac{1}{2\tau_n}\int_{\Gamma_{i-1,n}}|D_{\Gamma_{i-1,n}}v_{h_{i,n}+s\vphi}|^2\,d\H^{2}\biggr)_{\bigr|_{s=0}}=0\,,
$$
where, we recall, $v_{h_{i,n}+s\vphi}$ solves \eqref{vacca} with $h$ replaced by 
$h_{i,n}+s\vphi$. 
It can be shown that
\begin{align}
&\int_QW(E(u_{i,n}(x, h_{i,n}(x))))\vphi\, dx+\int_QD\psi(-Dh_{i,n},1)\cdot(-D\vphi, 0)\, dx+\frac{\e}p
\int_Q|H_{i,n}|^p\frac{Dh_{i,n}\cdot D\vphi}{J_{i,n}}\nonumber\\
&\qquad-\e \int_{Q}|H_{i,n}|^{p-2}H_{i,n}\biggl[\Delta\vphi-
\frac{D^2\vphi[Dh_{i,n},Dh_{i,n}]}{J^2_{i,n}}\nonumber\\
&\qquad\qquad-\frac{\Delta h_{i,n}Dh_{i,n}\cdot D\vphi}{J^2_{i,n}}-2\frac{D^2h_{i,n}[Dh_{i,n}, D\vphi]}{J^2_{i,n}}+3
\frac{D^2h_{i,n}[Dh_{i,n}, Dh_{i,n}]Dh_{i,n}\cdot D\vphi}{J^4_{i,n}}\biggr]\, dx\nonumber\\
&\qquad -\frac{1}{\tau_n}\int_Qv_{h_{i, n}}\vphi\, dx=0\,,\label{mostro} 
\end{align}
where the last integral is obtained by observing that $v_{h_{i,n}+s\vphi}=v_{h_{i,n}}+s v_{\vphi}$, with $v_{\vphi}$ solving
$$
\begin{cases}
\displaystyle\Delta_{\Gamma_{i-1,n}}v_{\vphi}=\frac{\vphi}{\sqrt{1+|Dh_{i-1,n} |^2}}\circ\pi\,,\vspace{7pt}\\
\displaystyle\int_{\Gamma_{h_{i-1,n}}} v_\vphi\, d\H^{2}=0\,.
\end{cases}
$$
Setting $w:=|H_{i,n}|^{p-2}H_{i,n}$, 
\begin{align}
A & :=\e\biggl(I-\frac{Dh_{i,n}\otimes Dh_{i,n}}{J^2_{i,n}}\biggr)\,, \label{eq:A}\\
b & :=\pi( D\psi(-Dh_{i,n},1))-\frac{\e}p
|H_{i,n}|^p\frac{Dh_{i,n}}{J_{i,n}}\nonumber\\
&\qquad+\e{w}\biggl[
-\frac{\Delta h_{i,n}Dh_{i,n}}{J^2_{i,n}}-2\frac{D^2h_{i,n}[Dh_{i,n}]}{J^2_{i,n}}+3
\frac{D^2h_{i,n}[Dh_{i,n}, Dh_{i,n}]Dh_{i,n}}{J^4_{i,n}}\biggr]\,,\nonumber\\
& c:=-W(E(u(x, h_{i,n}(x))))+\frac{1}{\tau_n}v_{h_{i, n}}\,, \nonumber
\end{align}
we have by \eqref{ghnboundbis}  and Theorem~\ref{th:pippa} that $A\in W^{1,p}_\#(Q; \mathbb{M}^{2\times 2}_{\rm sym})$, 
$b\in L^1(Q; \R^2)$, $c\in C_\#^{0, \alpha}(Q)$ for some $\alpha$, and we may rewrite \eqref{mostro} as
\beq\label{babymonster}
\int_Q w\,AD^2\vphi\, dx+\int_Qb\cdot D\vphi+\int_Q c\vphi\, dx=0\qquad\text{for all $\vphi\in C^\infty_{\#}(Q)$ with $\int_Q\vphi\, dx=0$.}
\eeq
By Lemma~\ref{lm:weyl-fusco} we get that $w\in L^q(Q)$ for $q\in (\frac{p}{p-1},2)$. Therefore, for any such $q$ we have 
$H_{i,n}\in L^{q(p-1)}(Q)$ and thus, by Lemma~\ref{lm:morini}, $h_{i,n}\in W^{2, q(p-1)}_\#(Q)$. In turn, using H\"older's inequality, this implies that $b$, $w\, \Div A\in L^{r_0}(Q;\R^2)$ where $r_0:=\frac{q(p-1)}{p}$. Observe that $r_0\in (1, 2)$. By applying Lemma~\ref{lm:weyl-fusco} again, we deduce that
$w\in W^{1, r_0}_{\#}(Q)$ and thus $w\in L^{\frac{2r_0}{2-r_0}}(Q)$. In turn, arguing as before, this implies that 
 $b$, $w\, \Div A\in L^{r_1}(Q;\R^2)$, where $r_1:=\frac{2r_0(p-1)}{(2-r_0)p}>r_0$.
If $r_1\geq 2$, then using again Lemma~\ref{lm:weyl-fusco}  we conclude that 
$w\in W^{1, r_1}_{\#}(Q)$, which  implies the claim, since $D^2h_{i,n}\in L^q(Q; \mathbb{M}^{2\times2}_{sym})$ and, in turn, $b$, $w\,\Div A\in L^q(Q;\R^2)$ for all $q$. Then the conclusion follows by Lemma~\ref{lm:weyl-fusco}. Otherwise, we proceed by induction. Assume that $w\in W^{1, r_{i-1}}_{\#}(Q)$. If $r_{i-1}\geq 2$ then the claim follows. Otherwise, a further application of Lemma~\ref{lm:weyl-fusco} implies that  $w\in W^{1, r_i}_{\#}(Q)$ with
$r_i:=\frac{2r_{i-1}(p-1)}{(2-r_{i-1})p}$. Since $r_{i-1}<2$, we have $r_{i}>r_{i-1}$. We claim that there exists $j$ such that $r_j>2$. Indeed, if not, the increasing sequence $\{r_i\}$ would converge to some $\ell\in (1, 2]$ satisfying 
$$
\ell=\frac{2\ell(p-1)}{(2-\ell)p}\,.
$$
However, this is impossible since the above identity is equivalent to $\ell=\frac2p<1$.

Finally, observe that since $|H_{i,n}|^{p-2}H_{i,n}\in W^{1,q}_\#(Q)$ for all $q\geq 1$, then $|H_{i,n}|^{p-1}\in C^{0,\alpha}_\#(Q)$ for every $\alpha\in (0,1)$. Hence $H_{i,n}\in C^{0,\sigma}_\#(Q)$ for all $\sigma \in (0, \frac{1}{p-1})$ and so, by  standard Schauder's estimates, 
$h_{i,n}\in C^{2,\sigma}_\#(Q)$ for all $\sigma \in (0, \frac{1}{p-1})$.

\noindent{\bf Step 2.} By Step 1 we may now write the Euler-Lagrange equation for $h_{i,n}$ in intrinsic form. To be precise, we claim that for all $\vphi\in C^2_\#(Q)$, with $\int_Q\vphi\, dx=0$, we have 
\begin{align}\label{51bis}
&\e\int_{\Gamma_{i,n}}D_{\Gamma_{i,n}}(|H_{i,n}|^{p-2}H_{i,n})D_{\Gamma_{i,n}}\phi\, d\H^2
-\e\int_{\Gamma_{i,n}}|H_{i,n}|^{p-2}H_{i,n}\Bigl(|B_{i,n}|^2-\frac1pH_{i,n}^2\Bigr)\phi\, d\H^2\nonumber\\
&+\int_{\Gamma_{i,n}}\bigl[\Div_{\Gamma_{i,n}}(D\psi(\nu_{i,n}))
+W(E(u_{i,n}))\bigr]\phi\, d\H^2
-\frac{1}{\tau_n}\int_{\Gamma_{i,n}}v_{h_{i,n}}\phi\, d\H^2=0\,,
\end{align}
where $\phi:=\frac{\vphi}{J_{i,n}}\circ\pi$. To see this, 
fix $h\in W^{3,q}_\#(Q)$ for some $q>2$, denote by $\Gamma$ and  $\Gamma_t$ the graphs of $h$ and $h+t\vphi$, respectively,   and by $H$ and $H_t$ the corresponding sums of the principal curvatures. Then by Proposition~\ref{prop:eulero} and arguing as in the proof of \eqref{mostro}, we have
\begin{align*}
\int_{\Gamma}D_{\Gamma}(|H|^{p-2}H)D_{\Gamma}\phi\, d\H^2
&-\int_{\Gamma}|H|^{p-2}H\Bigl(|B|^2-\frac1pH^2\Bigr)\phi\, d\H^2\\
&=\frac{1}p
\int_Q|H|^p\frac{Dh\cdot D\vphi}{J}- \int_{Q}|H|^{p-2}H\biggl[\Delta\vphi-
\frac{D^2\vphi[Dh,Dh]}{J^2}\nonumber\\
&\qquad-\frac{\Delta hDh\cdot D\vphi}{J^2}-2\frac{D^2h[Dh, D\vphi]}{J^2}+3
\frac{D^2h[Dh, Dh]Dh\cdot D\vphi}{J^4}\biggr]\, dx\,,
\end{align*}
where $\phi$ stands for $\frac{\vphi}{J}\circ\pi$ and $J:=\sqrt{1+|Dh|^2}$. By the approximation Lemma~\ref{lm:tosto}, this identity still holds if $h\in C^{2,\alpha}_\#(Q)$ and thus \eqref{51bis} follows from \eqref{mostro}, recalling that by  Step 1,  $h_{i,n}\in C^{2,\sigma}_\#(Q)$ for some $\sigma>0$.

 In order to show \eqref{51}, observe that Lemma~\ref{lm:morini}, together with the  bound $\|Dh_{i,n}\|_{L^\infty}\leq \Lambda_0$, implies  that 
\beq\label{jy}
\|D^2h_{i,n}\|_{L^q(Q)}\leq C(q,\Lambda_0)\|H_{i,n}\|_{L^q(Q)}\,.
\eeq
Moreover, since $\Gamma_{i,n}$ is of class $C^{2, \sigma}$, equation \eqref{51bis} yields that $|H_{i,n}|^{p-2}H_{i,n}\in H^2(\Gamma_{i,n})$, and in turn  $|H_{i,n}|^{p-2}H_{i,n}\in H^2(Q)$ (see Remark \ref{notation}).

As before, setting $w:=|H_{i,n}|^{p-2}H_{i,n}$, by approximation we may rewrite \eqref{51bis} as
\begin{align}
&\int_QA(x)DwD\Bigl(\frac{\vphi}{J_{i,n}}\Bigr) J_{i,n}\, dx-\e\int_Qw\vphi\Bigl(|B_{i,n}|^2-\frac1pH_{i,n}^2\Bigr)\, dx\nonumber\\
&+\int_{Q}\bigl[\Div_{\Gamma_{i,n}}(D\psi(\nu_{i,n}))
+W(E(u_{i,n}))\bigr]\vphi\, dx
-\frac{1}{\tau_n}\int_{Q}v_{h_{i,n}}\vphi\, dx=0\,,\label{schifoimmane}
\end{align}
for all $\vphi\in H^1_\#(Q)$, with $\int_Q\vphi\, dx=0$,
where $A$, defined as in \eqref{eq:A}, is an elliptic matrix with ellipticity constants depending only on $\Lambda_0$. Recall that $w\in H^2(Q)$. We now choose $\vphi=D_s\eta$, with $\eta\in H^2_\#(Q)$, and observe that  integrating by parts twice yields 
\begin{align*}
\int_QADwD\Bigl(\frac{D_s\eta}{J_{i,n}}\Bigr)J_{i,n}\, dx & = 
-\int_QAD(D_sw)D\Bigl(\frac{\eta}{J_{i,n}}\Bigr)J_{i,n}\, dx-\int_QD_s(AJ_{i,n})DwD\Bigl(\frac{\eta}{J_{i,n}}\Bigr)\, dx\\
&\quad+\int_QADwD\Bigl(\frac{\eta D_sJ_{i,n}}{J^2_{i,n}}\Bigr)J_{i,n}\, dx\\
&=-\int_QAD(D_sw)D\Bigl(\frac{\eta}{J_{i,n}}\Bigr)J_{i,n}\, dx-\int_QD_s(AJ_{i,n})DwD\Bigl(\frac{\eta}{J_{i,n}}\Bigr)\, dx\\
&\quad-\int_QAD^2w\frac{\eta D_sJ_{i,n}}{J_{i,n}}\, dx-\int_QD(AJ_{i,n})Dw\frac{\eta D_sJ_{i,n}}{J^2_{i,n}}\, dx\,.
\end{align*}

Therefore, recalling \eqref{schifoimmane}, and by density we may conclude that for every $\eta\in H^1_\#(Q)$
\begin{align*}
\int_QAD(D_sw)D\Bigl(\frac{\eta}{J_{i,n}}\Bigr)J_{i,n}\, dx& = -\int_QD_s(AJ_{i,n})DwD\Bigl(\frac{\eta}{J_{i,n}}\Bigr)\, dx\\
&\quad-\int_QAD^2w\frac{\eta D_sJ_{i,n}}{J_{i,n}}\, dx-\int_QD(AJ_{i,n})Dw\frac{\eta D_sJ_{i,n}}{J^2_{i,n}}\, dx \\
&\quad -\e\int_QwD_s\eta\Bigl(|B_{i,n}|^2-\frac1pH_{i,n}^2\Bigr)\, dx\\
&\quad+\int_{Q}\bigl[\Div_{\Gamma_{i,n}}(D\psi(\nu_{i,n}))
+W(E(u_{i,n}))\bigr]D_s\eta \, dx
-\frac{1}{\tau_n}\int_{Q}v_{h_{i,n}}D_s\eta\, dx\,.
\end{align*}
Finally, choosing $\eta=D_sw J_{i,n}$, we obtain
\begin{align*}
\int_QAD(D_sw)D(D_sw)J_{i,n}\, dx& = -\int_QD_s(AJ_{i,n})DwD(D_sw)\, dx\\
&\quad-\int_QAD^2wD_sw D_sJ_{i,n}\, dx-\int_QD(AJ_{i,n})Dw\frac{D_sw D_sJ_{i,n}}{J_{i,n}}\, dx \\
&\quad -\e\int_QwD_s(D_sw J_{i,n})\Bigl(|B_{i,n}|^2-\frac1pH_{i,n}^2\Bigr)\, dx\\
&\quad+\int_{Q}\bigl[\Div_{\Gamma_{i,n}}(D\psi(\nu_{i,n}))
+W(E(u_{i,n}))\bigr]D_s(D_sw J_{i,n}) \, dx\\
&\quad-\frac{1}{\tau_n}\int_{Q}v_{h_{i,n}}D_s(D_sw J_{i,n}), dx\,.
\end{align*}

Summing the resulting equations for $s=1,2$, estimating $D(AJ_{i,n})$ by  $D^2h_{i,n}$, and using several times  Young's Inequality, 
we deduce
\begin{multline}\label{52}
\int_Q |D^2w|^2\, dx\leq C\int_Q\Bigl(|Dw|^2|D^2h_{i,n}|^2\, dx+|H_{i,n}|^{2p+2}\\
+|H_{i,n}|^{2p-2}|D^2h_{i,n}|^4+\frac{v^2_{i,n}}{(\tau_n)^2}+1\Bigr)\, dx
\end{multline}
for some constant $C$  depending only on $\Lambda_0$, $D^2\psi$, and on the $C^{1,\alpha}$ bounds on $u_{i,n}$ provided by Theorem~\ref{th:pippa}. Note that by Young's Inequality and \eqref{jy}, we have
$$
\int_Q|H_{i,n}|^{2p-2}|D^2h_{i,n}|^4\, dx\leq 
C\int_Q\bigl(|H_{i,n}|^{2p+2}+|D^2h_{i,n}|^{2p+2}\bigr)\, dx
\leq C\int_Q|H_{i,n}|^{2p+2}\, dx\,.
$$
Combining the last estimate with \eqref{52}, we therefore have
 \beq\label{53}
\int_Q|D^2w|^2\,dx\leq C_0\int_Q\Bigl(|D^2h_{i,n}|^{2}|Dw|^2+|w|^{\frac{2(p+1)}{p-1}}
+\frac{v^2_{i,n}}{(\tau_n)^2}+1\Bigr)\, dx\,.
\eeq
To deal with the first term on the right-hand side,  we use H\"older's inequality,  \eqref{jy}  and Theorem~\ref{th:D} twice to get
\begin{align*}
C_0\int_Q|D^2h_{i,n}|^{2}|Dw|^2\, dx &\leq C_0 \biggl(\int_Q|D^2h|^{2(p-1)}\, dx\biggr)^{\frac{1}{p-1}} \biggl(\int_Q|Dw|^{\frac{2(p-1)}{p-2}}\, dx\biggr)^{\frac{p-2}{p-1}}\\
&\leq C\|w\|_{2}^{\frac{2}{p-1}}\|Dw\|_{\frac{2(p-1)}{p-2}}^{2}
\leq 
C\|w\|_{2}^{\frac{2}{p-1}}\Bigl(\|D^2w\|_2^{\frac{p}{2(p-1)}}\|w\|_2^{\frac{p-2}{2(p-1)}}\Bigr)^{2}\\
&=C\|D^2w\|_{2}^{\frac{p}{p-1}}\|w\|_2^{\frac{p}{p-1}}\leq 
C\|D^2w\|_{2}^{\frac{p}{p-1}}\Bigl(\|D^2w\|_{\frac{p}{p-1}}^{\frac{p-2}{2p}}\|w\|^{\frac{p+2}{2p}}_{\frac{p}{p-1}}\Bigr)^{\frac{p}{p-1}}\\
&\leq C\|D^2w\|_{2}^{\frac{3p-2}{2(p-1)}}\|w\|_{\frac{p}{p-1}}^{\frac{p+2}{2(p-1)}}\leq
\frac{1}{4}\|D^2w\|_{2}^2+C\,,
\end{align*}
where in the last inequality we used the fact that $\frac{3p-2}{2(p-1)}<2$ and that 
$\|w\|_{\frac{p}{p-1}}=\|H_{i,n}\|_{p}^{p-1}$ is uniformly bounded with respect to $i$, $n$.
Using again Theorem~\ref{th:D}, we also have
$$
C_0\int_{Q}|w|^{\frac{2(p+1)}{p-1}}\, dx\leq C\|D^2w\|_{\frac{p}{p-1}}^{\frac{p+2}{p}}\|w\|_{\frac{p}{p-1}}^{\frac{p^2+p+2}{p(p-1)}}\leq \frac{1}{4}\|D^2w\|_{2}^2+C\,,
$$
where as before we used the fact that $\frac{p+2}{p}<2$ and $\|w\|_{\frac{p}{p-1}}$ is uniformly bounded. Inserting the two estimates above in \eqref{53}, we then get
\begin{equation}\label{1000}
\int_Q|D^2w|^2\,dx\leq C\int_Q\Bigl(1
+\frac{v^2_{i,n}}{(\tau_n)^2}\Bigr)\, dx\,.
\end{equation}
Integrating the last inequality with respect to time and using \eqref{pre-pippa0} we conclude the proof of the theorem.
\end{proof}
\begin{remark}\label{rm:5}
The same argument used in Step 1 of the proof of Theorem~\ref{th:5} and in the proof of \eqref{51bis} shows that if $(h,u_h)\in X$ satisfies 
\begin{align*}
&\int_QW(E(u_{h}(x, h(x))))\vphi\, dx+\int_QD\psi(-Dh,1)\cdot(-D\vphi, 0)\, dx+\frac{\e}p
\int_Q|H|^p\frac{Dh\cdot D\vphi}{J}\nonumber\\
&\qquad-\e \int_{Q}|H|^{p-2}H\biggl[\Delta\vphi-
\frac{D^2\vphi[Dh,Dh]}{J^2}\nonumber\\
&\qquad\qquad-\frac{\Delta hDh\cdot D\vphi}{J^2}-2\frac{D^2h[Dh, D\vphi]}{J^2}+3
\frac{D^2h[Dh, Dh]Dh\cdot D\vphi}{J^4}\biggr]\, dx=0
\end{align*}
for all $\vphi\in C^2_\#(Q)$ such that $\int_Q\vphi\, dx=0$, then $(h,u_h)$ is a critical pair for the functional $F$.
\end{remark}
\begin{lemma}\label{th:contazzi}
Let $T_0$ and $\tilde H_n$  be as in Theorem~\ref{th:pippa}. Then $|\tH_n|^p$ is a Cauchy sequence in $L^1(0, T_0; L^{1}(Q))$.
Moreover, $|\tH_n|^{p-2}\tH_n$ is a Cauchy sequence in $L^1(0, T_0; L^{2}(Q))$.
\end{lemma}
For the proof of the lemma we need the following inequality.
\begin{lemma}\label{facile}
Let $p>1$. There exists $c_p>0$ such that
$$
\frac{1}{c_p}(x^{p-1}+y^{p-1})\leq \frac{|x^p-y^p|}{|x-y|}\leq c_p(x^{p-1}+y^{p-1})\,.
$$

\end{lemma}
\begin{proof}
By homogeneity it is enough to assume $y=1$ and $x>1$ and to observe that  
{
$$
\lim_{x\to+\infty}\frac{x^{p}-1}{(x-1)(x^{p-1}+1)}=1\qquad\lim_{x\to1}\frac{x^{p}-1}{(x-1)(x^{p-1}+1)}=\frac{p}2\,.
$$
}
\end{proof}
\begin{proof}[Proof of Lemma~\ref{th:contazzi}] We split the proof into two steps.

\noindent{\bf Step 1.} We start by showing that $|\tH_n|^p$ is a Cauchy sequence in $L^1(0,T_0;L^1(Q))$. Set $k:=[p]$, where $[\cdot]$ denotes the integer part. Note that $k\ge 2$ since $p>2$. From Lemma~\ref{facile} we get 
\begin{align}
\int_0^{T_0}\int_Q\bigl||\tH_n|^p&-|\tH_m|^p\bigr|\,dxdt= \int_0^{T_0}\int_Q\bigl||\tH^k_n|^{\frac{p}k}-|\tH^k_m|^{\frac{p}k}\bigr|\,dxdt\nonumber \\
&\leq c \int_0^{T_0}\int_Q\bigl||\tH^k_n|-|\tH^k_m|\bigr|\bigl(|\tH_n|^k+|\tH_m|^k\bigr)^{\frac{p}k-1}\,dxdt\nonumber\\
&\leq c\int_0^{T_0}\biggl(\int_Q|\tH_n^k-\tH_m^k|^2\, dx\biggr)^{\frac12}(\|\tH_n\|_{\infty}+\|\tH_m\|_{\infty})^{p-k}dt\nonumber\\
&\leq c\int_0^{T_0}\big(\|\tH_n^k-\tH_m^k-M_{m,n}\|_{2}+|M_{m,n}|\bigr)(\|\tH_n\|_{\infty}+\|\tH_m\|_{\infty})^{p-k}dt
\label{cz1}\,,
\end{align}
where $M_{m,n}:=\int_Q (\tH_n^k-\tH_m^k)\, dx$.
Set
\begin{equation}
w_n:=|\tH_n|^{p-2}\tH_n
\label{wn}
\end{equation}
and observe that $\tH_n^k=(w^+_n)^{\frac{k}{p-1}}+(-1)^k(w^-_n)^{\frac{k}{p-1}}$. Thus, 
\beq\label{cz2}
|D\tH_n^k|\leq \bigl|D(w^+_n)^{\frac{k}{p-1}}\bigr|+ \bigl|D(w^-_n)^{\frac{k}{p-1}}\bigr|\leq
c|Dw_n||w_n|^{\frac{k}{p-1}-1}=c|Dw_n||\tH_n|^{k-p+1}\,.
\eeq
From Lemma~\ref{lm:H-1} and inequalities \eqref{cz1}, \eqref{cz2} we get 
\begin{align}
\int_0^{T_0}&\int_Q\bigl||\tH_n|^p-|\tH_m|^p\bigr|\,dxdt\nonumber\\
&\leq c
\int_0^{T_0}\big(\|\tH_n^k-\tH_m^k-M_{m,n}\|^{\frac12}_{H^{-1}}\|D\tH_n^k-D\tH_m^k\|^{\frac12}_{2}+|M_{m,n}|\big)(\|\tH_n\|_{\infty}+\|\tH_m\|_{\infty})^{p-k}\,dt\nonumber\\
&\leq c
\int_0^{T_0}\|\tH_n^k-\tH_m^k-M_{m,n}\|^{\frac12}_{H^{-1}}(\|Dw_n\|_{2}+\|Dw_m\|_{2})^{\frac12}(\|\tH_n\|_{\infty}+\|\tH_m\|_{\infty})^{\frac{p-k+1}2}\,dt\nonumber\\
&\quad+\int_0^{T_0}|M_{m,n}|(\|\tH_n\|_{\infty}+\|\tH_m\|_{\infty})^{p-k}\,dt\,. \label{cz3}
\end{align}
Fix $n$, $m\in \N$. We now estimate the $H^{-1}$-norm of $\tH_n^k-\tH_m^k-M_{m,n}$. Recall that, in view of Remark \ref{rm:normah-1},
{
\begin{equation} 
\|\tH_n^k-\tH_m^k-M_{m,n}\|_{H^{-1}}=\|Du\|_2\,,
\label{900}
\end{equation} 
}
where $u$ is the unique $Q$-periodic solution  of 
{
\begin{equation} 
\begin{cases}
-\Delta u=\tH_n^k-\tH_m^k-M_{m,n} & \text{in $Q$,}\\
\int_Qu\, dx=0\,.
\end{cases}
\label{901}
\end{equation} 
}
Thus
\beq\label{cz4}
\int_Q|Du|^2\, dx=\int_Qu(\tH_n^k-\tH_m^k-M_{m,n})\, dx=\int_Qu(\tH_n-\tH_m)\sum_{i=0}^{k-1}\tH_n^{k-1-i}\tH_m^i\, dx\,,
\eeq
where we used also the fact that $\int_Qu\, dx=0$.
Fix $\de\in (0, 1)$ (to be chosen) and let $T^{\de}(t):=(t\lor-\de)\land\de$. Then
 {
\begin{equation} 
\tH_n=[(\tH_n-\de)^++\de]+T^\de(\tH_n)- [(-\tH_n-\de)^++\de]
\label{902}
\end{equation} 
}
and (see \eqref{wn})
$$
(\tH_n-\de)^++\de=\begin{cases}
w_n^{\frac{1}{p-1}} & \text{if $w_n\geq \de^{{p-1}}$,}\\
\de & \text{otherwise.}
\end{cases}
$$
Hence
\beq\label{cz5}
|D[(\tH_n-\de)^++\de]|\leq c\frac{|Dw_n|}{\de^{p-2}}\,,
\eeq
and a similar estimates holds for $D[(-\tH_n-\de)^++\de]$. We now set 
\begin{equation}
V_{n,\de}:=[(\tH_n-\de)^++\de]- [(-\tH_n-\de)^++\de]\,.
\label{Vndelta}
\end{equation}
From \eqref{cz4} we have
\begin{align}
&\int_Q|Du|^2\, dx\nonumber \\
&= \int_Q u(\tH_n-\tH_m)\sum_{i=0}^{k-1}\sum_{r=0}^{k-1-i}\sum_{s=0}^i\left(k-1-i\atop r\right)
\left(i\atop s\right)V_{n, \de}^{k-1-i-r}V_{m, \de}^{i-s}\bigl[T^\de(\tH_n)\bigr]^r\bigl[T^\de(\tH_m)\bigr]^s\, dx\nonumber\\
&= \int_Q u(\tH_n-\tH_m)\sum_{i=0}^{k-1}V_{n, \de}^{k-1-i}V_{m, \de}^{i}\, dx\nonumber\\
&+ \int_Q u(\tH_n-\tH_m)\sum_{i=0}^{k-1}\sum_{(r,s)\neq (0,0)}\left(k-1-i\atop r\right)
\left(i\atop s\right)V_{n, \de}^{k-1-i-r}V_{m, \de}^{i-s}\bigl[T^\de(\tH_n)\bigr]^r\bigl[T^\de(\tH_m)\bigr]^s\, dx\nonumber\\
&=: L+M \,.\label{cz6}
\end{align}
We start by estimating the last term in the previous chain of equalities:
\begin{align}
|M|&\leq c\int_Q |u||\tH_n-\tH_m|\sum_{i=0}^{k-1}\sum_{(r,s)\neq (0,0)}\de^{r+s}V_{n, \de}^{k-1-i-r}V_{m, \de}^{i-s}\, dx\nonumber\\
&\leq c\int_Q |u|(|\tH_n|+|\tH_m|)\sum_{\ell=1}^{k-1}\de^{\ell}\bigl[V_{n, \de}^{k-1-\ell}+V_{m, \de}^{k-1-\ell}\bigr]\, dx\nonumber \\
&\leq c\de \int_Q |u|(|\tH_n|+|\tH_m|)\bigl(1+V_{n, \de}^{k-2}+V_{m, \de}^{k-2}\bigr)\, dx\nonumber\\
&\leq c\de\biggl(\int_Qu^2\, dx\biggr)^{\frac12}(1+\|\tH_n\|_{\infty}+\|\tH_m\|_{\infty})^{k-1}\nonumber\\
&\leq \frac16\int_Q|Du|^2\, dx+c\de^2(1+\|\tH_n\|_{\infty}+\|\tH_m\|_{\infty})^{2(k-1)}\,,\label{cz7}
\end{align}
where  we used \eqref{Vndelta} and the Poincar\'e and Young inequalities. 
To deal with $L$, we integrate by parts and use \eqref{div eq} and the periodicity of $u$, $\tilde h_n$, and $\tilde h_m$ to get
{
\begin{align*}
L&=\int_Q\Bigl(\frac{D\tilde h_n}{\tilde J_n}-\frac{D\tilde h_m}{\tilde J_m}\Bigr)Du
\sum_{i=0}^{k-1}V_{n, \de}^{k-1-i}V_{m, \de}^{i}\, dx +\int_Q\Bigl(\frac{D\tilde h_n}{\tilde J_n}-\frac{D\tilde h_m}{\tilde J_m} \Bigr)u
\sum_{i=0}^{k-1}D\bigl(V_{n, \de}^{k-1-i}V_{m, \de}^{i}\bigr)\, dx\,,
\end{align*}
}
where 
{
\begin{equation} 
\tilde h_n(x,t):=h_{i,n}(x)\quad\text{if }t\in[(i-1)\tau_n, {i}\tau_n)\quad \text{and} \quad\tilde J_{n}(x.t):=\sqrt{1+|D\tilde h_n(x,t)|^2}\,.\label{tilde h n}
\end{equation}
}
From the equality above, recalling \eqref{wn}, \eqref{cz5}, and \eqref{Vndelta}, and setting
$$
\e_{n,m}:=\sup_{t\in [0, T_0]}\Bigl\|\frac{D\tilde h_n}{\tilde J_n}(\cdot, t)-\frac{D\tilde h_m}{\tilde J_m}(\cdot, t) \Bigr\|_{\infty}\,,
$$
we may estimate
\begin{align*}
|L| &\leq c\e_{n,m}\int_Q|Du|\bigl(1+|\tH_n|^{k-1}+|\tH_m|^{k-1}\bigr)\, dx\\
&\quad+ c\e_{n,m}\int_Q|u|\sum_{i=0}^{k-1}\bigl[|DV_{n, \de}^{k-1-i}|V_{m,\de}^i+|DV_{m, \de}^{i}|V_{n,\de}^{k-1-i}\bigr]\, dx\\
&\leq \frac16\int_Q|Du|^2\, dx+c\e_{n,m}^2(1+\|\tH_n\|_{\infty}+\|\tH_m\|_{\infty})^{2(k-1)}\\
&\quad +c\e_{n,m}\int_Q|u|\frac{|Dw_n|}{\de^{p-2}}\sum_{i=0}^{k-2}V_{n, \de}^{k-2-i}V_{m, \de}^i\, dx+ c\e_{n,m}\int_Q|u|\frac{|Dw_m|}{\de^{p-2}}\sum_{i=0}^{k-2}V_{m, \de}^{i-1}V_{n, \de}^{k-i-1}\, dx\\
&\leq  \frac16\int_Q|Du|^2\, dx+c\e_{n,m}^2(1+\|\tH_n\|_{\infty}+\|\tH_m\|_{\infty})^{2(k-1)}\\
&\quad +c\frac{\e_{n,m}}{\de^{p-2}}\int_Q|u|(|Dw_n|+|Dw_m|)(1+\|\tH_n\|_{\infty}+\|\tH_m\|_{\infty})^{k-2}\, dx\\
& \leq  \frac13\int_Q|Du|^2\, dx+c\e_{n,m}^2(1+\|\tH_n\|_{\infty}+\|\tH_m\|_{\infty})^{2(k-1)}\\
&\quad +c\frac{\e^2_{n,m}}{\de^{2(p-2)}}\int_Q(|Dw_n|+|Dw_m|)^2(1+\|\tH_n\|_{\infty}+\|\tH_m\|_{\infty})^{2(k-2)}\, dx\,.
\end{align*}
From this estimate, \eqref{900}, \eqref{901}, \eqref{cz6}, and \eqref{cz7}, choosing $\de^{2(p-2)}=\e_{n,m}$, 
with $n,m$ so large that $\e_{n,m}<1$ (see Theorem \ref{th:pippa}(i)), we obtain 
\begin{multline}\label{aggiunta1}
\|\tH_n^k-\tH_m^k-M_{m,n}\|_{H^{-1}}^2\leq c\e^{\alpha}_{n,m}\bigl[(1+\|\tH_n\|_{\infty}+\|\tH_m\|_{\infty})^{2(k-1)}\\
+(\|Dw_n\|_2+\|Dw_m\|_2)^2(1+\|\tH_n\|_{\infty}+\|\tH_m\|_{\infty})^{2(k-2)}\bigr]\,,
\end{multline}
where $\alpha:=\min\{1, \frac{1}{p-2}\}$.

We now estimate $M_{m,n}$. Since
$$
M_{m,n}=\int_Q(\tH_n^k-\tH_m^k)\, dx=\int_Q(\tH_n-\tH_m)\sum_{i=0}^{k-1}\tH_n^{k-1-i}\tH_m^i\, dx\,,
$$
using the same argument  with $u\equiv 1$ (see \eqref{aggiunta1}) gives
$$
|M_{m,n}|\leq c\e^{\frac{\alpha}2}_{n,m}\bigl[(1+\|\tH_n\|_{\infty}+\|\tH_m\|_{\infty})^{(k-1)}\\
+(\|Dw_n\|_2+\|Dw_m\|_2)(1+\|\tH_n\|_{\infty}+\|\tH_m\|_{\infty})^{(k-2)}\bigr] \,.
$$
 From this inequality, recalling \eqref{wn}, \eqref{cz3}, and \eqref{aggiunta1}, we deduce
{
\begin{align*}
\int_0^{T_0}\int_Q\bigl||\tH_n|^p&-|\tH_m|^p\bigr|\,dxdt\leq 
c(\e_{n,m})^{\frac\alpha4}\int_0^{T_0}(\|Dw_n\|_{2}+\|Dw_m\|_{2})^{\frac12}(1+\|w_n\|_{\infty}+\|w_m\|_{\infty})^{\frac{p}{2(p-1)}}\,dt\\
&\quad + c(\e_{n,m})^{\frac\alpha4}\int_0^{T_0}(\|Dw_n\|_{2}+\|Dw_m\|_{2})(1+\|w_n\|_{\infty}+\|w_m\|_{\infty})^{\frac12}\,dt\\
&\quad+ c
(\e_{n,m})^{\frac\alpha2}\int_0^{T_0}(1+\|w_n\|_{\infty}+\|w_m\|_{\infty})\,dt\\
&\quad + c(\e_{n,m})^{\frac\alpha2}\int_0^{T_0}(\|Dw_n\|_{2}+\|Dw_m\|_{2})(\|w_n\|_{\infty}+\|w_m\|_{\infty})^{\frac{p-2}{p-1}}\,dt\,.
\end{align*}
}
{
Observe now that by \eqref{ghnboundbis} and \eqref{H tilde} there exists $C>0$  such that 
$\int_Q|w_n|\, dx\leq \|\tH_n\|_{p-1}^{p-1}\leq C$ for all $n$ 
and thus, using the embedding of $H^2(Q)$ into $C(\overline{Q})$ and Poincar\'e's inequality, 
\begin{equation}  
\|Dw_n\|_2+\|w_n\|_{\infty}\leq C(1+\|D^2w_n\|_2)\,.
\label{Dwn}
\end{equation}
}

Therefore, from the above inequalities and using also the fact that $\frac12+\frac{p}{2(p-1)}<2$ and that $1+\frac{p-2}{p-1}<2$,  we conclude 
$$
\int_0^{T_0}\int_Q\bigl||\tH_n|^p-|\tH_m|^p\bigr|\,dxdt\leq c(\e_{n,m})^\frac\alpha4\int_0^{T_0}\bigl(1+\|D^2w_n\|_2+\|D^2w_m\|_2\bigr)^{2}\,dt\leq c(\e_{n,m})^\frac\alpha4\,,
$$
where the last inequality follows from \eqref{51}. This proves that the sequence $|\tH_n|^p$ is a Cauchy sequence
in $L^1(0,T_0; L^1(Q))$. Note also that using Lemma~\ref{facile} we have
\begin{align}
\int_0^{T_0}\int_Q\bigl||\tH_n|-|\tH_m|\bigr|^p\,dxdt & \leq c \int_0^{T_0}\int_Q\bigl||\tH_n|-|\tH_m|\bigr|
\bigl(|\tH_n|+|\tH_m|\bigr)^{p-1}\,dxdt\nonumber\\
&\leq c \int_0^{T_0}\int_Q\bigl||\tH_n|^p-|\tH_m|^p\bigr|\,dxdt\,. \label{cz8}
\end{align}

\noindent{\bf Step 2.} We now conclude the proof by showing that $w_n$ is a Cauchy sequence in $L^1(0,T_0;L^2(Q))$. To this purpose, we use Lemma~\ref{lm:H-1} to obtain 
\begin{align}
\int_0^{T_0}\|w_n-w_m\|_2\, dt&\leq \int_0^{T_0}\|w_n-w_m-N_{m,n}\|_2\, dt+\int_0^{T_0}|N_{m,n}|\, dt\nonumber\\
&\leq c \int_0^{T_0}\|w_n-w_m-N_{m,n}\|^{\frac23}_{H^{-1}}\|D^2w_n-D^2w_m\|^{\frac13}_{2}\, dt+\int_0^{T_0}|N_{m,n}|\, dt\,,
\label{cz8bis}
\end{align}
where $N_{m,n}:=\int_Q(w_n-w_m)\, dx$.
{
As observed in \eqref{900} and \eqref{901} , $\|w_n-w_m-N_{m,n}\|_{H^{-1}}=\|Dv\|_2$,
}
 where $v$ is the unique $Q$-periodic solution  of 
$$
\begin{cases}
-\Delta v=w_n-w_m-N_{m,n} & \text{in $Q$,}\\
\int_Qv\, dx=0\,.
\end{cases}
$$
As in \eqref{cz4}, using the fact that $\int_Qv\, dx=0$, we have
\begin{align}
\int_Q|Dv|^2\, dx & =\int_Q(w_n-w_m-N_{m,n})v=\int_Q\bigl(|\tH_n|^{p-2}\tH_n-|\tH_m|^{p-2}\tH_m\bigr)v\, dx \nonumber\\
&=\int_Q\bigl(|\tH_n|^{p-2}-|\tH_m|^{p-2}\bigr)\tH_nv\, dx+ 
\int_Q(\tH_n-\tH_m)|\tH_m|^{p-2}v\, dx\nonumber\\
&=:\tilde L+\tilde M\,. \label{cz9}
\end{align}
Now, by H\"older's inequality twice and the Sobolev embedding theorem,
\begin{align}
|\tilde L| & \leq \int_Q\bigl|(|\tH_n|^{p})^{\frac{p-2}{p}}-(|\tH_m|^{p})^{\frac{p-2}{p}}\bigr||\tH_n||v|\, dx\leq 
\int_Q\bigl||\tH_n|^{p}-|\tH_m|^{p}\bigr|^{\frac{p-2}{p}}|\tH_n||v|\, dx\nonumber\\
&\leq \|v\|_p\|\tH_n\|_{\infty}\biggl(\int_Q\bigl||\tH_n|^{p}-|\tH_m|^{p}\bigr|^{\frac{p-2}{p-1}}\, dx\biggr)^{\frac{p-1}{p}}
\leq  c\|Dv\|_2\|\tH_n\|_{\infty}\||\tH_n|^{p}-|\tH_m|^{p}\|_1^{\frac{p-2}{p}}\nonumber\\
&\leq \frac16\int_Q|Dv|^2\, dx+c\|\tH_n\|^2_{\infty}\||\tH_n|^{p}-|\tH_m|^{p}\|_1^{\frac{2(p-2)}{p}}\,.\label{cz10}
\end{align}

To estimate $\tilde M$,  arguing as in the previous step (see \eqref{902}) and observing that $(-|\tH_m|^{p-2}-\de)^+=0$), we write
{
\begin{align*}
\tilde M & = \int_Q(\tH_n-\tH_m)\bigl[(|\tH_m|^{p-2}-\de)^++\de\bigr]v\, dx+
\int_Q(\tH_n-\tH_m)\bigl[T^{\de}(|\tH_m|^{p-2})-\de\bigr]v\, dx\\
&= \int_Q\Bigl(\frac{D\tilde h_n}{\tilde J_n}-\frac{D\tilde h_m}{\tilde J_m}\Bigr)Dv \bigl[(|\tH_m|^{p-2}-\de)^++\de\bigr]\, dx\\
&\quad+\int_Q\Bigl(\frac{D\tilde h_n}{\tilde J_n}-\frac{D\tilde h_m}{\tilde J_m}\Bigr)v D\bigl[(|\tH_m|^{p-2}-\de)^++\de\bigr]\, dx+
\int_Q(\tH_n-\tH_m)\bigl[T^{\de}(|\tH_m|^{p-2})-\de\bigr]v\, dx\,.
\end{align*}
 Similarly to what we proved in \eqref{cz5}, we have 
$$
|D[(|\tH_m|^{p-2}-\de)^++\de]|\leq c\frac{|Dw_m|}{\de^{\frac1{p-2}}}\,.
$$
}
Therefore, arguing as in the previous step, we obtain 
\begin{align*}
|\tilde M|&\leq \frac16\int_Q|Dv|^2\, dx+c\e^2_{n,m}(1+\|\tH_m\|_\infty)^{2(p-2)}\\
& \quad+c\e_{n,m}\int_Q|v|\frac{|Dw_m|}{\delta^{\frac1{p-2}}}\, dx+c\de\int_Q|v|(\|\tH_n\|_\infty+\|\tH_m\|_\infty)\, dx\\
&\leq \frac13\int_Q|Dv|^2\, dx+c\e^2_{n,m}(1+\|\tH_m\|_\infty)^{2(p-2)}+c\frac{\e^2_{n,m}}{\delta^{\frac2{p-2}}}\|Dw_m\|_2^2+c\de^2(\|\tH_n\|_\infty+\|\tH_m\|_\infty)^2\,,
\end{align*}
where in the last line we used the Young and Poincar\'e inequalities. Choosing $\de$ so that $\de^{\frac2{p-2}}=\e_{n,m}$ and recalling \eqref{cz9} and \eqref{cz10}, we conclude that 
\begin{align}
\|w_n-w_m-N_{m,n}\|_{H^{-1}}&\leq c\|\tH_n\|_{\infty}\||\tH_n|^{p}-|\tH_m|^{p}\|_1^{\frac{p-2}{p}}\nonumber\\
&\quad+c(\e_{n,m})^{\frac{\beta}{2}}(1+\|\tH_n\|_\infty+\|\tH_m\|_\infty+\|\tH_m\|_\infty^{p-2}+\|Dw_m\|_2)\,,\label{aggiunta2}
\end{align}
where $\beta=\min\{1, p-2\}$.

Since by \eqref{wn}, 
$$
N_{m,n}=\int_Q(w_n-w_m)\, dx
=\int_Q\bigl(|\tH_n|^{p-2}-|\tH_m|^{p-2}\bigr)\tH_n\, dx+\int_Q(\tH_n-\tH_m)|\tH_m|^{p-2}\, dx\,,
$$
the same argument used to estimate the last two integrals in \eqref{cz9} (with $v\equiv 1$) gives 
\begin{align*}
|N_{m,n}|&\leq  c\|\tH_n\|_{\infty}\||\tH_n|^{p}-|\tH_m|^{p}\|_1^{\frac{p-2}{p}}\\
&\quad+c(\e_{n,m})^{\frac{\beta}{2}}(\|\tH_n\|_\infty+\|\tH_m\|_\infty+\|\tH_m\|_\infty^{p-2}+\|Dw_m\|_2)\,.
\end{align*}
{
From this estimate, recalling \eqref{wn}, \eqref{cz8bis} and \eqref{aggiunta2}, we have that
\begin{align*}
&\int_0^{T_0}\|w_n-w_m\|_2\, dt\leq
c\int_0^{T_0}\||\tH_n|^{p}-|\tH_m|^{p}\|_1^{\frac{2(p-2)}{3p}}\|w_n\|_\infty^{\frac{2}{3(p-1)}}(\|D^2w_n\|_2+\|D^2w_m\|_2)^{\frac13}\, dt\\
&\quad+ c (\e_{n,m})^{\frac{\beta}{3}}\int_0^{T_0}\bigl(1+\|w_n\|^{\frac{1}{p-1}}_\infty+\|w_m\|^{\frac{1}{p-1}}_\infty+\|w_m\|_\infty^{\frac{p-2}{p-1}}+\|Dw_m\|_2\bigr)^{\frac23}(\|D^2w_n\|_2+\|D^2w_m\|_2)^{\frac13}\, dt\\
&\quad+  c\int_0^{T_0}\|w_n\|^{\frac{1}{p-1}}_{\infty}\||\tH_n|^{p}-|\tH_m|^{p}\|_1^{\frac{p-2}{p}}\, dt\\
&\quad+c(\e_{n,m})^{\frac{\beta}{2}}\int_0^{T_0}(\|w_n\|^{\frac{1}{p-1}}_\infty+\|w_m\|^{\frac{1}{p-1}}_\infty+\|w_m\|_\infty^{\frac{p-2}{p-1}}+\|Dw_m\|_2)\, dt\,.
\end{align*}
}
Using \eqref{Dwn} and H\"older's inequality, we can bound the right-hand side of this inequality by
\begin{align*}
&\leq c\int_0^{T_0}\||\tH_n|^{p}-|\tH_m|^{p}\|_1^{\frac{2(p-2)}{3p}}(1+\|D^2w_n\|_2+\|D^2w_m\|_2)^{\frac13+\frac{2}{3(p-1)}}\, dt\\
&\quad+ c (\e_{n,m})^{\frac{\beta}{3}}\int_0^{T_0}\bigl(1+\|D^2w_n\|_2+\|D^2w_m\|_2\bigr)^{\frac23}(\|D^2w_n\|_2+\|D^2w_m\|_2)^{\frac13}\, dt\\
&\quad+  c\int_0^{T_0}\bigl(1+\|D^2w_n\|_2\bigr)^{\frac{1}{p-1}}\||\tH_n|^{p}-|\tH_m|^{p}\|_1^{\frac{p-2}{p}}\, dt
+c(\e_{n,m})^{\frac{\beta}{2}}\int_0^{T_0}(1+\|D^2w_n\|_2+\|D^2w_m\|_2)\, dt\\
&\leq c\biggl(\int_0^{T_0}\int_Q\bigl||\tH_n|^{p}-|\tH_m|^{p}\bigr|\,dxdt\biggr)^{\frac{2(p-2)}{3p}}
\biggl[\int_0^{T_0}\bigl(\|D^2w_n\|_2+\|D^2w_m\|_2\bigr)^{\frac{p(p+1)}{(p-1)(p+4)}}\biggr]^{\frac{p+4}{3p}}\\
&\quad+c\biggl(\int_0^{T_0}\int_Q\bigl||\tH_n|^{p}-|\tH_m|^{p}\bigr|\,dxdt\biggr)^{\frac{(p-2)}{2}}
\biggl[\int_0^{T_0}\bigl(1+\|D^2w_n\|_2\bigr)^{\frac{p}{2(p-1)}}\biggr]^{\frac{2}{p}}\\
&\quad+ c(\e_{n,m})^{\frac{\beta}{3}}\int_0^{T_0}(1+\|D^2w_n\|_2+\|D^2w_m\|_2)\,dt\,.
\end{align*}
Since $\frac{p(p+1)}{(p-1)(p+4)}<2$ and $\frac{p}{2(p-1)}<2$, recalling \eqref{51}, we finally have 
\begin{multline*}
\int_0^{T_0}\|w_n-w_m\|_2\, dt\leq c\biggl(\int_0^{T_0}\int_Q\bigl||\tH_n|^{p}-|\tH_m|^{p}\bigr|\,dxdt\biggr)^{\frac{2(p-2)}{3p}}\\+c\biggl(\int_0^{T_0}\int_Q\bigl||\tH_n|^{p}-|\tH_m|^{p}\bigr|\,dxdt\biggr)^{\frac{(p-2)}{2}}+c(\e_{n,m})^{\frac{\beta}{3}}\,.
\end{multline*}
The conclusion follows from Step 1.
\end{proof}

\begin{corollary}\label{cor:conv}
Let $\tH_n$  be the functions defined in \eqref{H tilde}, let $h$  be the limiting function provided by Theorem~\ref{th:pippa}, and set 
$$
H:=-\Div\Bigl(\frac{Dh}{1+|Dh|^2}\Bigr)\,.
$$
Then,
\beq\label{conv1}
|\tH_n|^p\to |H|^p \text{ in $L^1(0,T_0; L^1(Q))$}\quad\text{and}\quad
|\tH_n|^{p-2}\tH_n\to |H|^{p-2}H\text{ in $L^1(0,T_0; L^2(Q))$.}
\eeq
\end{corollary}
\begin{proof}
Let $\tilde h_n$ and $\tilde J_n$ be  as in the proof of Lemma~\ref{th:contazzi}. 
From Theorem~\ref{th:pippa}(i) we get that for all $t\in (0, T_0)$ and for all $\vphi\in C^{1}_\#(Q)$ we have
{
$$
\int_Q\tH_n\vphi\, dx=\int_Q\frac{D\tilde h_n}{\tilde J_n}\cdot D\vphi\, dx\to \int_Q\frac{Dh}{J}\cdot D\vphi\, dx=\int_Q\tH\vphi\, dx\,,$$
}
where $J=\sqrt{1+|Dh|^2}$. Since for every $t$, $\tH_n(\cdot, t)$ is bounded in $L^p(Q)$, we deduce that 
for all $t\in (0,T_0)$,
\beq\label{conv2}
\tH_n(\cdot, t)\wto H(\cdot, t)\qquad\text{weakly in $L^p(Q)$.}
\eeq
On the other hand, from Lemma~\ref{th:contazzi}  we know that there exist a subsequence $n_j$ and two functions $z$, $w$ such that for a.e. $t$,
\beq\label{conv3}
|\tH_{n_j}(\cdot, t)|^p\to z(\cdot, t)\quad\text{in $L^1(Q)$}\qquad\text{and}\qquad (|\tH_{n_j}|^{p-2}\tH_{n_j})(\cdot, t)\to w(\cdot, t)\quad\text{in $L^2(Q)$.}
\eeq
Moreover, for any such $t$ there exists a further subsequence, not relabelled,  (depending on $t$) such that $|\tH_{n_j}(x,t)|^p$, $|\tH_{n_j}(x,t)|^{p-2}\tH_{n_j}(x,t)$, and thus $\tH_{n_j}(x, t)$ converge for a.e. $x$. By \eqref{conv2} $\tH_{n_j}(x, t)\to H(\cdot, t)$  for a.e. $x$.  Thus, we conclude that 
$z=|H|^p$ and $w=|H|^{p-2}H$. 
\end{proof}
We now prove short time  existence for \eqref{geq}.
\begin{theorem}\label{th:existence}
Let $h_0\in W^{2,p}_\#(Q)$, let $h$ be the function given in Theorem~\ref{th:pre-pippa},  and let $T_0>0$ be as in Theorem~\ref{th:pippa}. Then $h$ is a solution of \eqref{geq} in $[0,T_0]$ in the sense of Definition~\ref{def:weaksol} with initial datum $h_0$. Moreover, there exists a non increasing $g$ such that
\beq\label{ex0}
F(h(\cdot, t), u_{h}(\cdot, t))=g(t) \quad\text{for  $t\in[0,T_0]\setminus Z_0$\,,} 
\eeq
where $Z_0$ is a set of zero measure, and 
\beq F(h(\cdot, t), u_{h}(\cdot, t))\leq g(t+)\quad\text{ for }t\in Z_0\,.\label{3000}\eeq 
\end{theorem}
This result motivates the following definition.
\begin{definition}\label{def:solvar}
We say that a solution to \eqref{geq} is  {\em variational} if it is the limit of  a subsequence of the minimizing movements scheme as in Theorem~\ref{th:pippa}(i). 
\end{definition}
\begin{proof}[Proof of Theorem~\ref{th:existence}]
Let $\tH_n$, $\tilde h_n$,  $\tilde J_n$ be the functions given in \eqref{H tilde}, and \eqref{tilde h n}. Set 
$\tilde W_n(x,t):=W(E(u_{i,n})(x, h_{i,n}(x)))$ and $\tilde v_n(x,t):=v_{h_{i,n}}(x)$  for   $t\in[(i-1)\tau_n, {i}\tau_n)$. 
Moreover, define $\hat v_n:=\frac{\tilde v_n}{\tau_n}$. Note that  for all $t$, $\hat v_n(\cdot, t)$ is the unique $Q$-periodic solution 
to
{
\beq\label{vaccabis}
\begin{cases}
\displaystyle\Delta_{\Gamma_{\tilde h_n(\cdot, t-\tau_n)}}w=\frac{1}{\tilde J_n(\cdot, t-\tau_n)}\frac{\partial   h_n(\cdot, t)}{\partial t}\vspace{7pt}\\
\displaystyle\int_{\Gamma_{\tilde h_n(\cdot, t-\tau_n)}} w\, d\H^{2}=0\,.
\end{cases}
\eeq
}
Fix $t\in (0,T_0)$ and a sequence $(i_k, n_k)$
such that $t_k:={i_k}\tau_{n_k}\to t$. Summing \eqref{mostro} from $i=1$ to $i=i_k$, we get
\begin{align}
&\int_0^{t_k}\int_Q\tilde W_{n_k}\vphi\, dxdt+\int_0^{t_k}\int_QD\psi(-D\tilde h_{n_k},1)\cdot(-D\vphi, 0)\, dxdt+\frac{\e}p
\int_0^{t_k}\int_Q|\tH_{n_k}|^p\frac{D\tilde h_{n_k}\cdot D\vphi}{\tilde J_{n_k}}\, dx dt\nonumber\\
&\qquad-\e\int_0^{t_k} \int_{Q}|\tilde H_{n_k}|^{p-2}\tilde H_{n_k}\biggl[\Delta\vphi-
\frac{D^2\vphi[D\tilde h_{n_k},D\tilde h_{n_k}]}{\tilde J^2_{n_k}}\nonumber\\
&\qquad\qquad-\frac{\Delta \tilde h_{n_k}D\tilde h_{n_k}\cdot D\vphi}{\tilde J^2_{n_k}}-2\frac{D^2\tilde h_{n_k}[D\tilde h_{n_k}, D\vphi]}{\tilde J^2_{n_k}}+3
\frac{D^2\tilde h_{n_k}[D\tilde h_{n_k}, D\tilde h_{n_k}]D\tilde h_{n_k}\cdot D\vphi}{\tilde J^4_{n_k}}\biggr]\, dxdt\nonumber\\
&\qquad -\int_0^{t_k}\int_Q\hat v_{n_k}\vphi\, dxdt=0\,.\label{ex1} 
\end{align}
We claim that we can pass to the limit in the above equation to get
{
\begin{align}
&\int_0^{t}\int_Q W(E(u(x, h(x, s), s)))\vphi\, dxds+\int_0^{t}\int_QD\psi(-D h,1)\cdot(-D\vphi, 0)\, dxds\nonumber\\
&\qquad+\frac{\e}p
\int_0^{t}\int_Q|H|^p\frac{D h\cdot D\vphi}{ J}\, dx ds\nonumber\\
&\qquad-\e\int_0^{t} \int_{Q}| H|^{p-2} H\biggl[\Delta\vphi-
\frac{D^2\vphi[D h,Dh]}{ J^2}\nonumber\\
&\qquad\qquad-\frac{\Delta hDh\cdot D\vphi}{ J^2}-2\frac{D^2 h[Dh, D\vphi]}{J^2}+3
\frac{D^2 h[Dh, Dh]Dh\cdot D\vphi}{ J^4}\biggr]\, dxds\nonumber\\
&\qquad -\int_0^{t}\int_Q\hat v\vphi\, dxds=0\,,\label{ex2} 
\end{align}
}
where $\hat v(\cdot, t)$ is the unique periodic solution in $H^{1}_\#(\Gamma(t))$ to
\beq\label{vaccafinal}
\begin{cases}
\displaystyle\Delta_{\Gamma_{h(\cdot, t)}}w=\frac{1}{J(\cdot, t)}\frac{\partial h(\cdot, t)}{\partial t}\,,\vspace{7pt}\\
\displaystyle\int_{\Gamma_{ h(\cdot, t)}} w\, d\H^{2}=0
\end{cases}
\eeq
for a.e. $t\in (0,T_0)$.
To prove the claim, observe that the convergence of the first two integrals in \eqref{ex1} immediately follows from (i) and (iii) of Theorem~\ref{th:pippa}. The convergence of the third integral \eqref{ex1} follows from  \eqref{conv1} and (i) of Theorem~\ref{th:pippa}. Similarly  \eqref{conv1} and (i) of Theorem~\ref{th:pippa} imply that  
{
$$
\int_0^{t_k} \int_{Q}|\tilde H_{n_k}|^{p-2}\tilde H_{n_k}\biggl[\Delta\vphi-
\frac{D^2\vphi[D\tilde h_{n_k},D\tilde h_{n_k}]}{\tilde J^2_{n_k}}\biggr]\, dxdt\to 
\int_0^{t} \int_{Q}| H|^{p-2} H\biggl[\Delta\vphi-
\frac{D^2\vphi[Dh,D h]}{ J^2}\biggr]\, dxds\,.
$$
}
Next we show the convergence of
\begin{multline*}
\int_0^{t_k} \int_{Q}|\tilde H_{n_k}|^{p-2}\tilde H_{n_k}\biggl[-\frac{\Delta \tilde h_{n_k}D\tilde h_{n_k}\cdot D\vphi}{\tilde J^2_{n_k}}\\-2\frac{D^2\tilde h_{n_k}[D\tilde h_{n_k}, D\vphi]}{\tilde J^2_{n_k}}+3
\frac{D^2\tilde h_{n_k}[D\tilde h_{n_k}, D\tilde h_{n_k}]D\tilde h_{n_k}\cdot D\vphi}{\tilde J^4_{n_k}}\biggr]\, dxdt
\end{multline*}
to the corresponding term in \eqref{ex2}. To this purpose, we only show that 
\beq\label{ex3}
\int_0^{t_k} \int_{Q}|\tilde H_{n_k}|^{p-2}\tilde H_{n_k}\frac{\Delta \tilde h_{n_k}D\tilde h_{n_k}\cdot D\vphi}{\tilde J^2_{n_k}}\, dxdt\to \int_0^{t} \int_{Q}|H|^{p-2} H\frac{\Delta hD h\cdot D\vphi}{J^2}\, dxds
\eeq
since the convergence of the other terms can be shown in a similar way. To prove \eqref{ex3}, we first observe that by  \eqref{ghnboundbis} and Theorem~\ref{th:pippa}(i)  we have
$ \Delta \tilde h_{n_k}(\cdot, t)\wto \Delta h(\cdot, t)$ in $L^p(Q)$  for all $t\in (0, T_0)$. On the other hand, \eqref{conv1} yields that for a.e. $t\in (0, T_0)$ we have
 $(\tilde H_{n_k}|^{p-2}\tilde H_{n_k})(\cdot, t)\to (|H|^{p-2}H)(\cdot, t)$ in $L^2(Q)$. Therefore, for a.e. $t\in (0, T_0)$
 $$
  \int_{Q}|\tilde H_{n_k}|^{p-2}\tilde H_{n_k}\frac{\Delta \tilde h_{n_k}D\tilde h_{n_k}\cdot D\vphi}{\tilde J^2_{n_k}}\, dx\to \int_{Q}|H|^{p-2} H\frac{\Delta hD h\cdot D\vphi}{J^2}\, dx\,.
 $$
The conclusion then follows by applying the Lebesgue dominated convergence theorem after observing that by \eqref{hn} and \eqref{ghnboundbis},
\begin{align*}
\biggl| \int_{Q}|\tilde H_{n_k}|^{p-2}\tilde H_{n_k}\frac{\Delta \tilde h_{n_k}D\tilde h_{n_k}\cdot D\vphi}{\tilde J^2_{n_k}}\, dx\biggr| & \leq
C\|\Delta \tilde h_{n_k}\|_{L^2(Q)}\||\tilde H_{n_k}|^{p-2}\tilde H_{n_k}\|_{L^2(Q)}\\
&\leq C\||\tilde H_{n_k}|^{p-2}\tilde H_{n_k}\|_{L^2(Q)}
\end{align*}
and that $\||\tilde H_{n_k}|^{p-2}\tilde H_{n_k}\|_{L^2(Q)}$ converges in $L^1(0,T_0)$ thanks to \eqref{conv1}.

Note \eqref{conv1} implies that for a.e. $t\in (0, T_0)$ we have $\|\tilde H_{n_k}(\cdot, t)\|_{L^p(Q)}\to \|H(\cdot, t)\|_{L^p(Q)}$. Since 
$\tilde H_{n_k}(\cdot, t)\wto H(\cdot, t)$ in $L^p(Q)$ (see \eqref{conv2}), we may conclude that $\tilde H_{n_k}(\cdot, t)\to H(\cdot, t)$ in $L^p(Q)$ for a.e. $t\in (0, T_0)$. Therefore, by \eqref{div eq} and \cite[Lemma 7.2]{AFM}, we also have $\tilde h_{n_k}(\cdot, t)\to h(\cdot, t)$ in 
$W^{2,p}_\#(Q)$   for a.e. $t\in (0, T_0)$. Thus, by \eqref{hn} and \eqref{ghnboundbis} and the Lebesgue dominated convergence theorem  we infer
that
\beq\label{ex5}
\int_0^{T_0}\int_Q|D^2\tilde h_{n_k}-D^2h|^p\, dxdt\to 0\,.
\eeq
This, together with the fact that  $h_n\wto h$ weakly in  $H^1(0,T_0; H^{-1}_{\#}(Q))$ (see \eqref{32}), implies that 
\beq\label{ex6}
\frac{1}{\tilde J_{n_k}(\cdot, \cdot-\tau_{n_k})}\frac{\partial h_{n_k}}{\partial t}\wto \frac{1}{J}\frac{\partial h}{\partial t}\quad\text{in $L^2(0,T_0; H^{-1}_\#(Q))$.}
\eeq
{
Indeed, for any $\vphi\in L^2(0,T_0; H^1_\#(Q))$,
}
\begin{align}
&\biggl|\int_0^{T_0}\int_Q\Bigl(\frac{1}{\tilde J_{n_k}(\cdot, \cdot-\tau_{n_k})}\frac{\partial h_{n_k}}{\partial t} - \frac{1}{J}\frac{\partial h}{\partial t}\Bigr)\vphi\,dxdt\biggr|\nonumber\\ 
&   \leq \biggl|\int_0^{T_0}\int_Q\Bigl(\frac{1}{\tilde J_{n_k}(\cdot, \cdot-\tau_{n_k})}- \frac{1}{J}\Bigr)\frac{\partial h_{n_k}}{\partial t}\vphi\,dxdt\biggr|+\biggl|\int_0^{T_0}\int_Q\Bigl(\frac{\partial h_{n_k}}{\partial t}-\frac{\partial h}{\partial t}\Bigr)\frac{\vphi}{J}\,dxdt\biggr|\nonumber\vspace{5pt}\\
&\leq \int_0^{T_0}\int_Q\Bigl\|\frac{\partial h_{n_k}}{\partial t}\Bigr\|_{H^{-1}}\Bigl\|\frac{\vphi}{\tilde J_{n_k}(\cdot, \cdot-\tau_{n_k})}- \frac{\vphi}{J}\Bigr\|_{H^{1}}\,dxdt+\biggl|\int_0^{T_0}\int_Q\Bigl(\frac{\partial h_{n_k}}{\partial t}-\frac{\partial h}{\partial t}\Bigr)\frac{\vphi}{J}\,dxdt\biggr|\,.
\label{ex7}
\end{align}
Since $H^1_\#(Q)$ is embedded in $L^q(Q)$ for all $q\geq 1$, we deduce from \eqref{ex5} that 
$\frac{\vphi}{\tilde J_{n_k}(\cdot, \cdot-\tau_{n_k})}\to \frac{\vphi}{J}$ in $L^2(0,T_0; H^1_\#(Q))$. This convergence together with \eqref{31} shows that the second last integral in \eqref{ex7} vanishes in the limit. On the other hand, also the last integral in \eqref{ex7} vanishes in the limit since $h_{n_k}\wto h$ weakly in  $H^1(0,T_0; H^{-1}_{\#}(Q))$. Thus, \eqref{ex6} follows.

Arguing as in the proof of Theorem~\ref{th:5} and integrating with respect to $t$, we have from \eqref{vaccabis},
{
\beq\label{ex8}
\int_0^{t}\int_QA_{n_k}D\hat v_{n}\cdot D\vphi\, dxds=\int_0^t\int_Q\frac{1}{\tilde J_{n_k}(\cdot, \cdot-\tau_{n_k})}\frac{\partial h_{n_k}}{\partial t}\vphi\, dxds
\eeq
}
for all $\vphi\in L^2(0,T_0; H^1_\#(Q))$, where
$$
A_{n_k}(x,t):=\biggl(I- \frac{D\tilde h_{n_k}(\cdot, \cdot-\tau_{n_k})\otimes D\tilde h_{n_k}(\cdot, \cdot-\tau_{n_k})}{\tilde J_{n_k}(\cdot, \cdot-\tau_{n_k})^2}\biggr)\tilde J_{n_k}(\cdot, \cdot-\tau_{n_k})\,.
$$ 
Note that \eqref{noconstraint} implies that $A_{n_k}(x,t)$ is an elliptic matrix with ellipticity constants depending only on $\Lambda_0$ for all $(x,t)$. Therefore, \eqref{ex8} immediately implies that
$$
\int_0^{T_0}\int_Q|D\hat v_{n_k}|^2\, dxdt\leq c\int_0^{T_0}\Bigl\|\frac{\partial h_{n_k}}{\partial t}\Bigr\|_{H^{-1}}^2\, dt\leq c
$$
thanks to \eqref{31}. Since $A_{n_k}\to A:=(I- \frac{Dh\otimes Dh}{ J^2})J$ in $L^\infty(0,T_0; L^{\infty}(Q))$ by Theorem \ref{th:pippa}(i),  from the estimate above and recalling \eqref{ex6} and \eqref{ex8}, we conclude that 
$$
\hat v_{n_k}\wto \hat v \quad\text{weakly in $L^2(0, T_0; H^1_\#(Q))$, }
$$
where $\hat v$ satisfies 
{
$$
\int_0^{t}\int_QAD\hat v\cdot D\vphi\, dxds=\int_0^t\int_Q\frac{1}{J}\frac{\partial h}{\partial t}\vphi\, dxds
$$
}
for all $\vphi\in L^2(0,T_0; H^1_\#(Q))$ and for all $t\in (0, T_0)$. In turn, 
letting $\vphi$ vary in a countable dense subset 
of $H^{1}_\#(Q)$ and 
differentiating the above equation with respect to $t$, we  conclude that for a.e. $t\in (0, T_0)$
 $\hat v(\cdot, t)$ is the unique solution in $H^{1}_\#(\Gamma_{h(\cdot, t)})$ to \eqref{vaccafinal} for a.e. $t\in (0, T_0)$. This  shows that the last integral in \eqref{ex1} converges  and thus  \eqref{ex2} holds. Again letting $\vphi$ vary in a countable dense subset 
of $H^{1}_\#(Q)$ and 
differentiating \eqref{ex2} with respect to $t$ we obtain
\begin{align}
&\int_Q W(E(u(x, h(x, t), t)))\vphi\, dx+\int_QD\psi(-D h,1)\cdot(-D\vphi, 0)\, dx+\frac{\e}p
\int_Q|H|^p\frac{D h\cdot D\vphi}{ J}\, dx \nonumber\\
&\qquad-\e \int_{Q}| H|^{p-2} H\biggl[\Delta\vphi-
\frac{D^2\vphi[D h,Dh]}{ J^2}\nonumber\\
&\qquad\qquad-\frac{\Delta hDh\cdot D\vphi}{ J^2}-2\frac{D^2 h[Dh, D\vphi]}{J^2}+3
\frac{D^2 h[Dh, Dh]Dh\cdot D\vphi}{ J^4}\biggr]\, dx\nonumber\\
&\qquad -\int_Q\hat v\vphi\, dx=0\label{ex2bis} 
\end{align}
for all $\vphi\in H^{1}_\#(Q)$.
Since, by \eqref{51}, $|H|^{p-2}H\in L^2(0,T_0; H^2_\#(Q))$, arguing as in Step 2 of the proof of Theorem~\ref{th:5}, we have that the above equation is equivalent to 
\begin{align*}
&\e\int_{\Gamma_{h}}D_{\Gamma_{h}}(|H|^{p-2}H)D_{\Gamma_{h}}\phi\, d\H^2
-\e\int_{\Gamma_{h}}|H|^{p-2}H\Bigl(|B|^2-\frac1pH^2\Bigr)\phi\, d\H^2\nonumber\\
&+\int_{\Gamma_{h}}\bigl[\Div_{\Gamma_{h}}(D\psi(\nu))
+W(E(u))\bigr]\phi\, d\H^2
-\int_{\Gamma_{h}}\hat v\phi\, d\H^2=0
\end{align*}
for a.e. $t\in (0, T_0)$, where $\phi:=\frac{\vphi}{J}$. This equation, together with \eqref{vaccafinal}, implies that $h$ is a  solution to \eqref{geq} in the sense of Definition~\ref{def:weaksol}.

Next, to show that the energy decreases during the evolution, we observe first that for every $n$ the map 
$t\mapsto F(\tilde h_n(\cdot, t), \tilde u_n(\cdot, t))$ is non increasing, as shown in \eqref{31.5}. 
Note also that thanks to \eqref{conv1} we may assume, up to extracting a further subsequence, that for a.e. $t$, 
$\tH_n\to H$ in $L^p(Q)$. This fact, together with (i) and (iii) of Theorem~\ref{th:pippa},  implies that for all such $t$, $F(\tilde h_n(\cdot, t), \tilde u_n(\cdot, t))\to F( h(\cdot, t),  u(\cdot, t))$. Thus also \eqref{ex0} follows. Let $t\in Z_0$ and choose $t_n\to t^+$, with $t_n\not\in Z_0$ for every $n$. Finally, since $h(\cdot, t_n)\wto h(\cdot, t)$ weakly in 
$W^{2,p}_\#(Q)$ by \eqref{ghnboundbis}, by lower semicontinuity we get that 
$$
F(h(\cdot, t), u(\cdot, t))\leq \liminf_n F(h(\cdot, t_n), u(\cdot, t_n))=\lim_n g(t_n)=g(t+)\,.
$$
\end{proof}
\section{Liapunov stability of the flat configuration}\label{sec:liapunov}
In this section we are going to study the Liapunov stability of an admissible flat configuration. 
Take  $h(x)\equiv d>0$ and let $u_d$ denote the corresponding elastic equilibrium.  Throughout this section we assume that the Dirichlet datum $w_0$ is affine, i.e., of the form $w_0(x,y)=(A[\, x\,], 0)$ for some $A\in \mathbb{M}^{2\times 2}$. As already mentioned, a tipical choice is given by $w_0(x,y):=(e^1_0 x_1, e_0^2 x_2,0)$, where the vector $e_0:=(e_0^1, e_0^2)$, with $e_0^1$, $e_0^2>0$, embodies the mismatch between the crystalline lattices of film and substrate.  

A detailed analysis of the so-called 
Asaro-Tiller-Grinfeld morphological stability/instability was undertaken in \cite{Bo, FM09}. It was shown that if $d$ is sufficiently small, then the flat configuration $(d, u_d)$ is a volume constrained local minimizer for the functional 
\beq\label{Gg}
G(h,u):=\int_{\Omega_h}W(E(u))\,dz+\int_{\Gamma_h}\psi(\nu)\, d\H^{2}\,.
\eeq
To be precise, it was proved that  if $d$ is small enough, then the second variation $\pa^2 G(d, u_d)$ is positive definite and that, in turn, this implies the local minimality property.
 In order to state the results of this section, we need to introduce some preliminary notation.
 In the following, given $h\in C^2_\#(Q)$, $h\geq 0$, $\nu$ will denote the unit vector field coinciding with the gradient of  the  signed distance from $\Om_h^\#$, which is well defined in a sufficiently small tubular neighborhood of $\Gamma_h^\#$.  Moreover,  for every $x\in\Gamma_h$ we set
\begin{equation}\label{sff}
\mathbb{B}(x):=D\nu(x).
\end{equation}
Note that the bilinear form associated with $\mathbb{B}(x)$ is symmetric and, when restricted to
$T_x \Gamma_h{\times}T_x \Gamma_h$, it coincides with the {\em second fundamental form
of $\Gamma_h$ at $x$}. Here $T_x \Gamma_h$ denotes the tangent space to $\Gamma_h$ at $x$. 
For $x\in \Gamma_h$ we also set $H(x):=\Div \nu(x)={\rm trace\,}\mathbb{B}(x)$, which is the {\em sum of the principal curvatures} of $\Gamma_h$ at $x$.
Given a (sufficiently) smooth and positively 1-homogeneous function $\omega:\R^N\setminus\{0\}\to\R$,
we consider the {\em anisotropic second fundamental form} defined as
$$
\mathbb{B}^{\omega}:=  D ( D \omega \circ \nu )
$$
and we  set
\begin{equation}\label{defHpsi}
H^{\omega}:= \text{trace} \, \mathbb{B}^{\omega} = \Div\, ( D \omega  \circ \nu ).
\end{equation}
 We also introduce the following space of periodic displacements
 \beq\label{adiomega}
A(\Om_h):=\{u\in LD_\#(\Om_h;\R^3):\, u(x,0)=0\}\,.
\eeq
 Given a regular configuration $(h, u_h)\in X$ with $h\in C^2_\#(Q)$
and $\vphi\in \widetilde H_\#^1(Q)$, where  
 \beq\label{h1tilde}
\widetilde H_\#^1(Q):=\biggl\{\vphi\in H^1_\#(Q):\, \int_Q\vphi\, dx=0\biggr\}\,,
\eeq
we recall that the second variation of $G$ at $(h, u_h)$ with respect to the direction $\vphi$ is 
$$
\frac{d^2}{dt^2}G(h+t\vphi, u_{h+t\vphi})|_{t=0}\,,
$$
where, as usual,  $u_{h+t\vphi}$ denotes the elastic equilibrium in  $\Om_{h+t\vphi}$. 
It turns out (see \cite[Theorem~4.1]{Bo}) that
\begin{multline}\label{secvar}
\frac{d^2}{dt^2}G(h+t\vphi, u_{h+t\vphi})|_{t=0}= \partial^2 G(h,  u_h)[\vphi] \\
- \int_{\Gamma_h}( W (E(u_h)) +H^{\psi}) \,
\Div_{\Gamma_h} \left[ \left( \frac{(D h, |D h|^2 )} {\sqrt{1 + |D h|^2}} \circ \pi \right) \phi^2\right] \, d \H^2\,,
\end{multline}
where $\partial^2 G(h,  u_h)[\vphi]$ is the (non local) quadratic form defined as
\begin{align}
\partial^2 G(h, u_h)[\vphi] := &-2\int_{\Om_h}W(E(v_{\phi}))\, dz+\int_{\Gamma_h}D^2\psi(\nu)[D_{\Gamma_h}\phi, D_{\Gamma_h}\phi]\, d\H^2\nonumber\\
&+ \int_{\Gamma_h} (\partial_\nu [ W(E(u_h) ] - \textnormal{trace} (\mathbb{B}^{\psi}\mathbb{B}) ) \, \phi^2\,  d\H^2 \label{ps2G}\,,
\end{align}
$$
\phi:=\frac{\vphi}{\sqrt{1+|D h|^2}}\circ \pi\,,
$$
and
$v_{\phi}$   the unique solution in $A(\Omega_h)$ to
\begin{equation} \label{vf}
\int_{\Omega_h}\C E(v_{\phi}) : E(w)\, dz = \int_{\Gamma_h}\Div_{\Gamma_h}(\phi\, \C E(u_h)) \cdot w \,d\H^2
\qquad \text {for all } w\in A(\Omega_h)\,.
\end{equation}
Note that if $(h, u_h)$ is a  critical pair of $G$ (see Definition \ref{def:critical} with $\e=0$), then the integral in \eqref{secvar} vanishes so that 
$$
\frac{d^2}{dt^2}G(h+t\vphi, u_{h+t\vphi})|_{t=0}= \partial^2 G(h,  u_h)[\vphi]\,.
$$
Throughout this section $\alpha$ will denote a fixed  number in the interval $(0, 1-\frac2p)$.
The next result is a simple consequence of \cite[Theorem 6.6]{Bo}. 

 \begin{theorem}\label{th:bonacini}
 Assume that the surface density $\psi$ is of class $C^3$ away from the origin, it satisfies \eqref{sotto}, and the following convexity condition holds: 
{
for every $\xi\in S^2$,
\beq\label{uc}
D^2\psi(\xi)[w, w]>0 \qquad\text{for all  $w\perp \xi$, $w\neq 0$.}
\eeq
}
If 
 \beq\label{bo1}
 \partial^2G(d, u_d)[\vphi]>0\qquad\text{for all $\vphi\in \widetilde H_\#^{1}(Q)\setminus\{0\}$}\,,
 \eeq
 then 
 there exists $\de>0$  such that 
 $$
 G(d, u_d)<G(k,v)
 $$
 for all $(k,v)\in X$, with  $|\Om_k|=|\Om_d|$,   $0<\|k-d\|_{C^{1,\alpha}_\#(Q)} \leq \de$.  
 \end{theorem}
 \begin{proof}
{
 By condition \eqref{bo1} and \cite[Theorem 6.6]{Bo} there exists $\de_0>0$ such that if  $0<\|k-d\|_{C^1_\#(Q)} \leq \de_0$ and $\|D\eta\|_\infty\leq 1+\|Du_d\|_{\infty}$, with $(k, \eta)\in X$, then
\beq\label{fon1}
 G(d, u_d)<G(k,\eta)\,.
\eeq
Note that we may choose $0<\de< \de_0$ such that if  $\|k-d\|_{C^{1,\alpha}_\#(Q)}\leq \de$ and $u_k$ is the elastic equilibrium corresponding to $k$, by elliptic regularity (see also Lemma~\ref{lm:chepallequadrate}) we have that
 $\|Du_k\|_{\infty}\leq 1+\|Du_d\|_{\infty}$. Therefore, using \eqref{fon1} with $\eta:=u_k$, we may conclude that 
$$
 G(d, u_d)< G(k,u_k)\leq  G(k,v)\,,
 $$
where in the last inequality we used the minimality of $u_k$, and  the result follows. 
}
 \end{proof}

\begin{remark}\label{remark-bonacini}
{
It can be shown that Theorem~\ref{th:bonacini} continues to hold if  \eqref{uc} is replaced by the weaker condition 
\beq\label{weaker}
D^2\psi(e_3)[w, w]>0 \qquad\text{for all $w\perp e_3$, $w\neq 0$.}
\eeq
Indeed,  \eqref{weaker} implies that \eqref{uc} holds for all $\xi$ belonging to a suitable neighborhood $U\subset S^2$  of $e_3$.
In turn, by choosing $\delta$ sufficiently small we can ensure that the outer unit normals to $\Gamma_k$ lie in $U$, provided 
$\|k-d\|_{C^{1, \alpha}_\#(Q)}<\de$. A careful inspection of the proof of \cite[Theorem 6.6]{Bo} shows that under these circumstances condition \eqref{uc} is only needed to hold at vectors $\xi\in U$.
}
\end{remark}

 \begin{remark}\label{uc2}
 Under assumption \eqref{uc}, it can be shown that \eqref{bo1} is equivalent to having (see \cite[Corollary 4.8]{Bo})
 \beq\label{emme0}
 \inf\{\partial^2G(d, u_d)[\vphi]:\, \vphi\in \widetilde H_\#^1(Q),\, \|\vphi\|_{H^1_\#(Q)}=1 \}=:m_0>0\,,
 \eeq
i.e., 
 $$
 \partial^2G(d, u_d)[\vphi]\geq m_0 \|\vphi\|_{H^1_\#(Q)}^2\qquad\text{for all $\vphi\in \widetilde H_\#^{1}(Q)$}\,.
 $$
 \end{remark}

{
 \begin{remark}\label{remark critical}
Note that  if  the profile $h\equiv d$ is flat, then the corresponding elastic equilibrium $u_d$ is affine. It immediately follows that $(d,u_d)$ is a critical pair in the sense of  Definition~\ref{def:critical}.
 \end{remark}
}
 We now consider the case of a non-convex surface energy density $\psi$, and introduce the 
 ``relaxed'' functional defined for all $(h,u)\in X$ as
 \beq\label{Ggr}
\overline G(h,u):=\int_{\Omega_h}W(E(u))\,dz+\int_{\Gamma_h}\psi^{**}(\nu)\, d\H^{2}\,,
 \eeq
 where $\psi^{**}$ is the convex envelope of $\psi$.
 It turns out that if the boundary of the Wulff shape $W_\psi$ associated with the nonconvex density $\psi$ contains a flat horizontal facet, then the flat configuration is always an isolated volume-constrained local minimizer, irrespectively of the value of $d$.  We recall that the Wulff shape $W_\psi$  is given by (see \cite[Definition 3.1]{Fo})
 $$
 W_{\psi}:=\{z\in \R^3:\, z\cdot\nu<\psi(\nu)\text{ for all }\nu\in S^{2}\}\,.
 $$
 The following result can be easily obtained from \cite[Theorem 7.5 and Remark 7.6]{Bo} arguing as in the last part of the proof of Theorem~\ref{th:bonacini}. 
 \begin{theorem}\label{th:bonacini2}
 Let $\psi:\R^3\to [0, +\infty)$ be a Lipschitz positively one-homogeneous function, satisfying \eqref{sotto},  and let $\{(x,y)\in \R^3:\, |x|\leq \alpha, y=\beta\}\subset\pa W_{\psi}$ for some $\alpha$, $\beta>0$. Then  there exists $\de>0$  such that  
 $$
 \overline G(d, u_d)<\overline G(k,v)
 $$
 for all $(k,v)\in X$, with  $|\Om_k|=|\Om_d|$,  $0<\|k-d\|_{C^{1, \alpha}_\#(Q)} \leq \de$.  
 \end{theorem}
 In the next two subsections we use the previous theorems to study the Liapunov stability of the flat configuration both in the convex and nonconvex case.
 \begin{definition}\label{def:liapunov}
We say that the flat configuration $(d, u_d)$ is \emph{Liapunov stable} if for every $\sigma>0$, there exists $\de(\sigma)>0$ such that if $(h_0, u_0)\in X$
{
 with  $|\Om_{h_0}|=|\Om_d|$ 
}
  and $\|h_0-d\|_{W^{2,p}_\#(Q)}\leq \de(\sigma)$, then every variational solution $h$ to \eqref{geq}  according to Definition~\ref{def:solvar}, with initial datum $h_0$, exists for all times and $\|h(\cdot, t)-d\|_{W^{2,p}_\#(Q)}\leq \sigma$ for all $t>0$.
\end{definition}

\subsection{The case of a non-convex surface density.}\label{sub:nc}
In this subsection will show that if the boundary of the Wulff shape $W_\psi$ associated with  $\psi$ contains a flat horizontal facet, then the flat configuration is always Liapunov stable. 
\begin{theorem}\label{th:bonacinievol}
Let $\psi:\R^3\to [0, +\infty)$ be a positively one-homogeneous function of class $C^2$ away from the origin, such that \eqref{sotto} holds, and let $\{(x,y)\in \R^3:\, |x|\leq \alpha, y=\beta\}\subset\pa W_{\psi}$ for some $\alpha$, $\beta>0$.
Then for every $d>0$ the flat configuration $(d, u_d)$ is Liapunov stable (according to Definition~\ref{def:liapunov}).
\end{theorem}
\begin{proof}
We start by observing that from the assumptions on $\psi$, $e_3$ is normal to boundary $\pa W_\psi$ of the Wulff shape $W_\psi$ associated with $\psi$. Thus, by \cite[Proposition 3.5-(iv)]{Fo} it follows that $\psi(e_3)=\psi^{**}(e_3)$.  In turn, by Theorem~\ref{th:bonacini2}, we may find $\de>0$ such that 
\beq\label{bonacini21}
 F(d, u_d)=\overline G(d, u_d)<\overline G(k,v)\leq F(k,v)
 \eeq
  for all $(k,v)\in X$, with  $|\Om_k|=|\Om_d|$ and  $0<\|k-d\|_{C^{1,\alpha}_\#(Q)}\leq \de$.  Fix $\sigma>0$ and choose $\de_0\in (0,\min\{\de, \sigma/2\})$ so small that 
 \beq\label{qq1}
\|h-d\|_{C^{1,\alpha}_\#(Q)}\leq \de_0\Longrightarrow \| Dh\|_\infty<\Lambda_0\,,
\eeq
where $\Lambda_0$ is as in \eqref{pin}.  
For every $\tau>0$ set
{
\begin{equation*}
\omega(\tau):=\sup\bigl\{\|k-d\|_{C^{1,\alpha}_\#(Q)}\bigr\}
\end{equation*}
where the supremum is taken over all $(k, v)\in X$ such that
$$
 |\Om_k|=|\Om_d|\,,\quad \|k-d\|_{C^{1,\alpha}_\#(Q)}\leq\de\,,
\quad\text{and } F(k, v)-F(d, u_d)\leq \tau\,.
$$
}
Clearly, $\omega(\tau)>0$ for $\tau>0$. We claim that $\omega(\tau)\to 0$ as $\tau\to 0^+$. Indeed, to see this we assume by contradiction that there exists a sequence $(k_n, v_n)\in X$, with $|\Om_{k_n}|=|\Om_d|$, such that
{
\begin{equation}
\liminf_nF(k_n, v_n)\le F(d, u_d)\quad\text{and}\quad  0<c_0\leq \|k_n-d\|_{C^{1,\alpha}_\#(Q)}\leq \de \label{3001}
\end{equation}
}
 for some $c_0>0$. By Lemma~\ref{lm:morini}, up to a subsequence, we may assume that 
$k_n\wto k$ in $W^{2,p}_\#(Q)$ and that $v_n\wto v$ in $H^1_{\loc}(\Om_k; \R^3)$ for some $(k,v)\in X$ satisfying
$\de\geq \|k-d\|_{C^{1,\alpha}_\#(Q)}\geq c_0$, since $W^{2,p}_\#(Q)$ is compactly embedded in $C^{1,\alpha}_\#(Q)$. By lower semicontinuity we also have that 
{
$$
F(k,v)\leq \liminf_nF(k_n, v_n)\le F(d, u_d)\,,
$$
}
which contradicts \eqref{bonacini21}.

Let $\delta(\sigma)$ so small that if $\|h_0-d\|_{W_\#^{2,p}(Q)}\leq \de(\sigma)$ then
$$
\|h_0-d\|_{C^{1,\alpha}_\#(Q)}<\de_0\quad\text{and } F(h_0,u_0 )-F(d,u_d)\leq \omega^{-1}(\de_0/2)\,,
$$
where $\omega^{-1}$ is the generalized inverse of $\omega$ defined as 
$\omega^{-1}(s):=\sup\{\tau>0:\, \omega(\tau)\leq s\}$ for all $s>0$. Note that since $\omega(\tau)>0$ for $\tau>0$ and $\omega(\tau)\to 0$ as $\tau\to 0+$ we have that $\omega^{-1}(s)\to 0$ as $s\to 0+$.
Let $h$ be a variational solution as in  Theorem~\ref{th:pre-pippa} (see Definition \ref{def:solvar}).
Let 
$$
T_1:=\sup\{t>0:\, \|h(\cdot, s)-d\|_{C^{1,\alpha}_\#(Q)}\leq \de_0\quad\text{for all } s\in (0,t)\}\,.
$$
Note that by Theorem~\ref{th:pippa}, $T_1>0$. We claim that $T_1=+\infty$.
Indeed, if $T_1$ were finite, then, recalling \eqref{32.5}, we would get for all $s\in [0, T_1]$
\beq\label{qq2}
F(h(\cdot, T_1), u_{h(\cdot,T_1)})-F(d, u_d)\leq F(h_0, u_0)-F(d, u_d)\leq \omega^{-1}(\de_0/2)\,,
\eeq
which implies $\| h(\cdot, T_1)-d\|_{C^{1,\alpha}_\#(Q)}\leq \de_0/2$ by the definition of $\omega$. Then, equation \eqref{qq1},  Remark~\ref{rm:elle0}, and Theorem~\ref{th:pippa} would imply that there exists $T>T_1$ such that
$\| h(\cdot, t)-d\|_{C^{1,\alpha}_\#(Q)}\leq \de_0$ for all $t\in (T_1, T)$, thus giving a contradiction.
We conclude that $T_1=+\infty$ and that $\| h(\cdot, t)-d\|_{C^{1,\alpha}_\#(Q)}\leq \de_0$ for all $t>0$.
Therefore, \eqref{qq1} implies that  $\| Dh(\cdot,t)\|_{\infty}<\Lambda_0$ for all times, which, together with  Remark~\ref{rm:elle0}, gives
that $h$ is a solution to \eqref{geq} for all times.
Moreover, by \eqref{qq2} we have also shown that $F(h(\cdot, t), u_{h(\cdot, t)})-F(d, u_d)\leq \omega^{-1}(\de_0/2)$ for all $t>0$, which by \eqref{bonacini21} implies that 
$$
\e\int_{\Gamma_{h(\cdot, t)}}|H|^p\, d\mathcal{H}^2\leq  \omega^{-1}(\de_0/2)\,.
$$
Using elliptic regularity (see  \eqref{div eq}), this inequality and  the fact that $\|h(\cdot, t)-d\|_{\infty}\leq \sigma/2$ for all $t>0$ imply that 
$\|h(\cdot, t)-d\|_{W^{2,p}_\#(Q)}\leq \sigma$ provided that $\de_0$ and in turn $\de(\sigma)$ are chosen sufficiently small.
\end{proof}

\subsection{The case of a convex surface density.} \label{sub:c} 
In this section we will show that, under the  convexity assumption \eqref{uc}, the condition $\pa^2 G(d, u_d)>0$ implies that $(d, u_d)$ is asymptotically stable for the regularized evolution equation \eqref{geq} (see Theorem~\ref{th:botto} below). We start by addressing the Liapunov stability (see Definition~\ref{def:liapunov}).
\begin{theorem}\label{th:boli}
Assume that the surface density $\psi$ satisfies the assumptions of Theorem~\ref{th:bonacini} and that the flat configuration  $(d, u_d)$ satisfies \eqref{bo1}. Then  $(d, u_d)$ is Liapunov stable.
\end{theorem}
\begin{proof}
Since \eqref{bonacini21} still holds with $\overline G$ replaced by $G$ in view of Theorem~\ref{th:bonacini}, we can conclude as in the proof of  Theorem~\ref{th:bonacinievol}.
\end{proof}
\begin{remark}[Stability of the flat configuration for small volumes]\label{rm:bonacini}
If the surface density $\psi$ satisfies the assumptions of Theorem~\ref{th:bonacini}, then there exists $d_0>0$ (depending only on Dirichlet boundary datum $w_0$) such that  \eqref{bo1} holds for all $d\in (0, d_0)$ (see \cite[Proposition 7.3]{Bo}).
\end{remark}
\begin{definition}\label{def:asymptotic}
We say that flat configuration $(d, u_d)$ is \emph{asymptotically stable} if there exists $\de>0$ such that if $(h_0, u_0)\in X$,
{
with  $|\Om_{h_0}|=|\Om_d|$ 
}
   and $\|h_0-d\|_{W^{2,p}_\#(Q)}\leq \de$, then every variational solution  $h$ to \eqref{geq}  according to Definition~\ref{def:solvar}, with initial datum $h_0$,  exists for all times and $\|h(\cdot, t)-d\|_{W^{2,p}_\#(Q)}\to 0$ as $t\to+\infty$.
\end{definition}

We start by showing that if a variational solution to \eqref{geq} exists for all times, then there exists a sequence $\{t_n\}\subset (0,+\infty)$, with $t_n\to\infty$, such that $h(\cdot,t_n)$ converges to a critical profile (see Definition~\ref{def:critical}).
\begin{proposition}\label{th:pci}
Assume that for a certain initial datum $h_0\in W^{2,p}_\#(Q)$ there exists a global in time  variational  solution $h$. Then there exist a sequence $\{t_n\}\subset (0,+\infty)\setminus Z_0$,  where $Z_0$ is the set  in \eqref{ex0},  and a critical profile  $\bar h$ for $F$ such that $t_n\to\infty$ and $h(\cdot,t_n) \to\bar h$ strongly in $W^{2,p}_\#(Q)$. 
\end{proposition}
\begin{proof}
From equation \eqref{31}, by lower semicontinuity we have that 
$$
\int_0^\infty\Bigl\|\frac{\partial h}{\partial t}\Bigr\|^2_{H^{-1}(Q)}dt\leq CF(h_0, u_0)\,.
$$
Since the set $Z_0$ has measure zero, we may find a sequence $\{t_n\}\subset (0,+\infty)\setminus Z_0$,  $t_n\to\infty$, such that  $\bigl\|\frac{\partial h(\cdot,t_n)}{\partial t}\bigr\|_{H^{-1}(Q)}\to0$. Since 
{
$h\in L^{\infty}(0,\infty; W^{2,p}_\#(Q))\cap H^1(0,\infty; H^{-1}_\#(Q))$, 
}
setting $h^n=h(\cdot,t_n)$, we may also assume that there exists $\bar h\in W^{2,p}_\#(Q)$ such that $h^n\wto\bar h$ weakly in $W^{2,p}_\#(Q)$. In turn, denoting by $u_{h^n}$ the corresponding elastic equilibria, by elliptic regularity (see also Lemma~\ref{lm:chepallequadrate} ) we have that $u_{h^n}(\cdot, h^n(\cdot))\to u_{\bar h}(\cdot, \bar h(\cdot))$ in $C^{1,\alpha}_\#(Q;\R^3)$. Let $\hat v^n$ be the unique $Q$-periodic solution  to \eqref{vaccafinal} with $t=t_n$ and note that $\hat v^n\to0$ in $H^{1}_\#(Q)$ since $\bigl\|\frac{\partial h(\cdot,t_n)}{\partial t}\bigr\|_{H^{-1}(Q)}\to0$. Writing the equation satisfied by $h^n$ as in \eqref{mostro}, we have for all $\varphi\in C^2_\#(Q)$, with $\int_Q\varphi\,dx=0$,
\begin{align}
&\int_QW(E(u_{h^n}(x, h^n(x))))\vphi\, dx+\int_QD\psi(-Dh^n,1)\cdot(-D\vphi, 0)\, dx+\frac\e{p}
\int_Q|H^n|^p\frac{Dh^{n}\cdot D\vphi}{J^{n}}\nonumber\\
&\qquad-\e \int_{Q}|H^n|^{p-2}H^n\biggl[\Delta\vphi-
\frac{D^2\vphi[Dh^{n},Dh^{n}]}{(J^n)^2}\nonumber\\
&\qquad-\frac{\Delta h^{n}Dh^{n}\cdot D\vphi}{(J^n)^{2}}-2\frac{D^2h^{n}[Dh^{n}, D\vphi]}{(J^n)^{2}}+3
\frac{D^2h^{n}[Dh^{n}, Dh^{n}]Dh^{n}\cdot D\vphi}{(J^n)^{4}}\biggr]\, dx\nonumber\\
&\qquad -\int_Q\hat v^n\vphi\, dx=0\,,\label{mostrobis}
\end{align}
where $H^n$ stands for the sum of the principal curvatures of $h^n$ and $J^n=\sqrt{1+|Dh^n|^2}$. Arguing exactly as in the proof of Theorem~\ref{th:5}(see \eqref{1000}) we deduce  that 
\beq\label{pc1}
\int_Q|D^2(|H^n|^{p-2}H^n)|^2\,dx\leq C\int_Q(1+(\hat v^n)^2)\,dx\,
\eeq
for some constant $C$ independent of $n$. Thus, passing to a subsequence, if necessary, we may also assume that there exists 
{
$w\in H^2_\#(Q)$ 
}
such that $|H^n|^{p-2}H^n\wto w$ weakly in $H^2_\#(Q)$ and $|H^n|^{p-2}H^n\to w$ strongly in $H^1_\#(Q)$. Since $H^1_\#(Q)$ is continuously embedded in $L^q(Q)$ for every $1\le q<\infty$ by the Sobolev embedding theorem,  there exists  $z\in L^1(Q)$ such that $|H^n|^p\to z$ in $L^1(Q)$. The same argument used at the end of the proof of Corollary~\ref{cor:conv} shows that $z=|\bar H|^p$ and $w=|\bar H|^{p-2}\bar H$, where $\bar H$ is the sum of the principal curvatures of $\bar h$. 
\par
Using all the convergences proved above, and arguing as in the proof of Theorem~\ref{th:existence} we may pass to the limit in equation \eqref{mostrobis}, thus getting  that $\bar h$ is a critical profile by Remark~\ref{rm:5}.
\end{proof}
\begin{lemma}\label{lm:c2alfa}
Assume that \eqref{uc} and \eqref{bo1} hold. Then there exist $\sigma>0$ and $c_0>0$ such that
$$
\partial^2 G(h, u_h)[\vphi]\geq c_0 \|\vphi\|_{H^1_\#(Q)}^2\qquad\text{for all $\vphi\in \widetilde H_\#^{1}(Q)$}\,,
$$
provided
  $\|h-d\|_{C^{2, \alpha}_\#(Q)}\leq \sigma$, where $\widetilde H_\#^{1}(Q)$ is defined in \eqref{h1tilde}.
\end{lemma}
\begin{proof}
Let $m_0$ be the positive constant  defined in \eqref{emme0}. We claim that there exists $\sigma>0$ such that
$$
\inf\{\partial^2 G(h, u_h)[\vphi]:\, \vphi\in \widetilde H_\#^1(Q),\, \|\vphi\|_{H^1_\#(Q)}=1 \}\geq \frac{m_0}2\,,
$$
whenever 
$\|h-d\|_{C^{2, \alpha}_\#(Q)}\leq \sigma$. Indeed, if not, then  
there exist two  sequences $\{h_n\}\subset C^{2,\alpha}_\#(Q)$, with 
 $h_n\to d$ in $C^{2,\alpha}_\#(Q)$,  and $\{\vphi_n\}\subset \widetilde H^1_\#(Q)$, with $\|\vphi_n\|_{H^1_\#(Q)}=1$, such that
\beq\label{absurd}
\partial^2G(h_n, u_{h_n})[\vphi_n]<\frac{m_0}2\,.
\eeq
Set 
\beq\label{phizero}
\phi_n:=\frac{\vphi_n}{\sqrt{1+|D h_n|^2}}\circ \pi\,,
\eeq
where we recall that $\pi(x,y)=x$. Let $v_{\phi_n}$ be the unique solution in $A(\Omega_{h_n})$, see \eqref{adiomega}, to
\begin{equation}\label{vufienne} 
\int_{\Omega_{h_n}}\C E(v_{\phi_n}) : E(w)\, dz = \int_{\Gamma_{h_n}}\Div_{\Gamma_{h_n}}(\phi_n\, \C E(u_{h_n})) \cdot w \,d\H^2
\qquad \text{for all }w\in A(\Omega_{h_n})\
\end{equation}
and let
 $v_{\vphi_n}$ be the unique solution in $A(\Omega_d)$ to
\begin{equation} \label{1001}
\int_{\Omega_d}\C E(v_{\vphi_n}) : E(w)\, dz = \int_{\Gamma_{d}}\Div_{\Gamma_{d}}(\vphi_n\, \C E(u_d)) \cdot w \,d\H^2
\qquad \text{for all } w\in A(\Omega_d)\,.
\end{equation}
 Observe that (see, e.g., Lemma~\ref{lm:chepallequadrate})
{
$$
\|\Div_{\Gamma_{h_n}}(\phi_n\, \C E(u_{h_n}))\|_{L^2(\Gamma_{h_n})}\leq C\|\vphi_n\|_{H^1_\#(Q)}
$$ 
}
for some constant $C>0$ depending only on 
$$
\sup_n(\|\C E(u_{h_n})\|_{C^1(\Gamma_{h_n})}+\|h_n\|_{C^2_\#(Q)})
$$
and thus independent of $n$. Therefore, 
choosing $w=v_{\phi_n}$ in \eqref{vufienne},
and using  Korn's inequality, we  deduce that 
\beq\label{vufienneb}
\sup_n\|v_{\phi_n}\|_{H^1(\Om_{h_n})}<+\infty\,.
\eeq
The same bound holds for the sequence $\{v_{\vphi_n}\}$.

Next we show that 
\beq\label{start}
\int_{\Om_{h_n}}W(E(v_{\phi_n}))\, dz-\int_{\Om_d}W(E(v_{\vphi_n}))\, dz\to 0
\eeq
as $n\to\infty$. Consider a sequence $\{\Phi_n\}$ of diffeomorphisms $\Phi_n: \Om_d\to \Om_{h_n}$ such that   $\Phi_n-Id$ is
$Q$-periodic with respect to $x$, $\Phi_n(x, y)=(x, y+ d-h_n(x))$ in a neighborhood of $\Gamma_d$, and  $\|\Phi_n- Id\|_{C^{2,\alpha}(\overline\Om_d; \R^3)}\leq C\|h_n-d\|_{C^{2, \alpha}_\#(Q)}\to 0$.
Set $w_n:=v_{\phi_n}\circ\Phi_n$. Changing variables, we get that $w_n\in A(\Om_d)$ satisfies
\begin{equation}\label{1002}
\int_{\Omega_d}A_nD w_n : D w \,dz
= \int_{\Gamma_d} \bigl(\Div_{\Gamma_{h_n}}(\phi_n\, \C E(u_{h_n})\bigr) \circ \Phi_{n} \bigr) \cdot w \; J_{\Phi_n} \,d\H^2
\end{equation}
for every $w\in A(\Om_d)$, where $J_{\Phi_n}$ stands for the $(N-1)$-Jacobian of $\Phi_n$ and the fourth order tensor valued functions $A_n$ satisfy $A_n\to \C$  in $C^{1,\alpha}(\overline \Om_d)$.  We claim that
\beq\label{claim0.0}
\int_{\Om_d} W(E(w_n-v_{\vphi_n}))\, dz\to 0
\eeq
as $n\to\infty$.
Note that this would immediately  imply
 $\int_{\Om_d} W(E(w_n))\, dz-\int_{\Om_d} W(E(v_{\vphi_n}))\, dz\to0$ and in turn, taking also into account that 
 $A_n\to \C$  uniformly and that $\frac12\int_{\Omega_d}A_nD w_n : D w_n \,dz=\int_{\Om_{h_n}}W(E(v_{\phi_n}))\, dz$, claim \eqref{start} would follow. In order to prove \eqref{claim0.0}, we write
 \begin{align*}
\int_{\Omega_d} \C &D(v_{{\vphi}_n}-w_n) : D(v_{{\vphi}_n}-w_n) \,dz \\
&= \int_{\Omega_d} \C D v_{{\vphi}_n} : D(v_{{\vphi}_n}-w_n) \,dz
- \int_{\Omega_d} (\C-A_n) D w_n : D(v_{{\vphi}_n}-w_n) \,dz \\
&\hspace{1cm}- \int_{\Omega_d} A_n D w_n : D(v_{{\vphi}_n}-w_n) \,dz \\
&= \int_{\Gamma_d} \Div_{\Gamma_d} ({\vphi}_n \C E(u_d)) \cdot (v_{{\vphi}_n}-w_n) \,d\H^2
- \int_{\Omega_d} (\C-A_n) D w_n : D(v_{{\vphi}_n}-w_n) \,dz \\
&\hspace{1cm}- \int_{\Gamma_d} \bigl( \Div_{\Gamma_{h_n}}(\phi_n \C E(u_{h_n}))\circ\Phi_{n} \bigr) \cdot (v_{{\vphi}_n}-w_n) J_{\Phi_n}\,d\H^2 \\
&=: I_1-I_2-I_3\,,
\end{align*}
where we used \eqref{1001} and \eqref{1002}. From  \eqref{vufienneb},  the analogous bound for the sequence $\{v_{\vphi_n}\}$, and the uniform convergence of $A_n$ to $\C$ we deduce that $I_2$ tends to $0$.

Fix $\,\eta=(\eta_1, \eta_2, \eta_3)\in C^1_\#(\Gamma_d;\R^3) \simeq C^1_\#(Q;\R^3)$. Using the fact that 
$\Phi^{-1}_n(x,y)=(x, y-h_n(x)+d)$ in a neighborhood of $\Gamma_{h_n}$ we have 
$$
D_{\Gamma_{h_n}}(\eta_j\circ\Phi^{-1}_{n})=  \bigl(I-\nu_{h_n}\otimes\nu_{h_n}\bigr)D_{\Gamma_d}\eta_j\circ\Phi_{n}^{-1}\,,
$$ 
where we set $\nu_{h_n}:=\frac{(-Dh_n, 1)}{\sqrt{1+|Dh_n|^2}}$. Using this fact, we then have by repeated integrations by parts and changes of variables,
\begin{align*}
\int_{\Gamma_d} \bigl( \Div_{\Gamma_{h_n}}&(\phi_n \C E(u_{h_n}))\circ\Phi_{n} \bigr) \cdot \eta \, J_{\Phi_n}d\H^2 \\
&= \int_{\Gamma_{h_n}} \Div_{\Gamma_{h_n}}(\phi_n\C E(u_{h_n})) \cdot \eta\circ\Phi^{-1}_{n} \,d\H^2 \\
&= - \int_{\Gamma_{h_n}} \phi_n \C E(u_{h_n}) : D_{\Gamma_{h_n}}(\eta\circ\Phi^{-1}_{n}) \,d\H^2 \\
&= - \int_{\Gamma_{h_n}} \bigl(I-\nu_{h_n}\otimes\nu_{h_n}\bigr)\phi_n \C E(u_{h_n}) : D_{\Gamma_d}\eta\circ\Phi_{n}^{-1}\ \,d\H^2 \\
&= - \int_{\Gamma_d}\bigl[\bigl(I-\nu_{h_n}\otimes\nu_{h_n}\bigr)\phi_n \C E(u_{h_n})\bigr]\circ\Phi_n : D_{\Gamma_d}\eta\, J_{\Phi_n}\,d\H^2 \\
&= \int_{\Gamma_d} \Div_{\Gamma_d} \Bigl[\bigl[\bigl(I-\nu_{h_n}\otimes\nu_{h_n}\bigr)\phi_n \C E(u_{h_n})\bigr]\circ\Phi_n J_{\Phi_n}\Bigr] \cdot \eta \,d\H^2.
\end{align*}
Hence, we may rewrite
\beq\label{itre}
I_1-I_3 = \int_{\Gamma_d} \Div_{\Gamma_d}g_n \cdot (v_{{\vphi}_n}-w_n) \,d\H^2,
\eeq
where by \eqref{phizero},
\begin{align*}
g_n & :=
 {\vphi}_n \C E(u_d)-\bigl[\bigl(I-\nu_{h_n}\otimes\nu_{h_n}\bigr)\phi_n \C E(u_{h_n})\bigr]\circ\Phi_n J_{\Phi_n}\\
 & ={\vphi}_n\left[\C E(u_d)- \bigl[\bigl(I-\nu_{h_n}\otimes\nu_{h_n}\bigr) \C E(u_{h_n})\bigr]\circ\Phi_n 
 \frac{J_{\Phi_n}}{\sqrt{1+|D h_n|^2}}\right]\,.
\end{align*}
Since $h_n\to d$ in $C^{2,\alpha}_\#(Q)$, by standard Schauder's estimates for
the elastic displacements $u_{h_n}$, we get
$$
\C E(u_d)- \bigl[\bigl(I-\nu_{h_n}\otimes\nu_{h_n}\bigr) \C E(u_{h_n})\bigr]\circ\Phi_n 
 \frac{J_{\Phi_n}}{\sqrt{1+|D h_n|^2}}\to 0 \qquad\text{in $C^{1,\alpha}(\Gamma_d)$}\,.
$$ 
Therefore, by  \eqref{itre} and the equiboundedness of $\{v_{\phi_n}\}$ and $\{w_n\}$ we have that 
 $I_1-I_3\to 0$.  
 This concludes the proof of \eqref{claim0.0} and, in turn, of \eqref{start}.

Finally, again from the $C^{2,\alpha}$-convergence of $\{h_n\}$ to $d$ and the fact that 
$$
\partial_\nu [ W(E(u_{h_n}) ]\circ\Phi_n\to  \partial_\nu [ W(E(u_d)) ]\quad\text{in $C^{0,\alpha}_\#(\Gamma_d)$}
$$
by standard Schauder's  elliptic estimates, recalling \eqref{ps2G} we easily infer  that 
\begin{multline}\label{infer}
\biggl(\partial^2G(h_n, u_{h_n})[\vphi_n]+2\int_{\Om_{h_n}}W(E(v_{\phi_n}))\, dz\biggr)\\-
\left(\partial^2G(d, u_d)[\vphi_n]+2\int_{\Om_d}W(E(v_{\vphi_n}))\, dz\right)\to 0
\end{multline}
as $n\to\infty$. Thus, recalling \eqref{start}, we also have
$$
\partial^2G(h_n, u_{h_n})[\vphi_n]-\partial^2G(d, u_d)[\vphi_n]\to 0
$$
and, in turn, by \eqref{absurd}
$$
\limsup \partial^2G(d, u_d)[\vphi_n]\leq\frac{m_0}{2}\,,
$$
which is a contradiction to \eqref{emme0}. This concludes the proof of the lemma.
\end{proof}

Next we prove that $(d,u_d)$ is an isolated critical pair.
\begin{proposition}\label{prop:nocritici}
Assume that \eqref{uc} and \eqref{bo1} hold. Then there exists $\sigma>0$ such that if $(h, u_h)\in X$ 
{
with  $|\Om_{h}|=|\Om_d|$ and $0<\|h-d\|_{W^{2,p}_\#(Q)}\leq \sigma$, 
}
then $(h, u_h)$ is not a critical pair.
\end{proposition}
\begin{proof}
 Assume by contradiction that there exists a sequence 
{
$h_n\to d$ in $W^{2,p}_\#(Q)$,  with $h_n\ne d$ and  $|\Om_{h_n}|=|\Om_d|$, 
}
  such that $(h_n, u_{h_n})$ is a critical pair.  Using the Euler-Lagrange equation and arguing as in the proof  of Theorem~\ref{th:5}, one can
show that 
$$
 \int_Q|D^2(|H_n|^{p-2}H_n)|^2\,dx\leq C\int_Q\Bigl(|D^2h_{n}|^{2}|D(|H_n|^{p-2}H_n)|^2+|H_n|^{2(p+1)}
+1\Bigr)\, dx\,.
$$
Indeed, this can obtained as  \eqref{53}, taking into account that  there is no contribution from the time derivative.
From this inequality, arguing exactly as in the final part of the proof of Theorem~\ref{th:5} we deduce that
$$
 \int_Q|D^2(|H_n|^{p-2}H_n)|^2\,dx\leq C
$$
for some $C$ independent of $n$. In particular, by the Sobolev embedding theorem, $\{|H_n|^{p-2}H_n\}$ is bounded in 
$C^{0,\beta}_\#(Q)$ for every $\beta\in (0,1)$. Hence, $\{H_n\}$ is bounded in 
$C^{0,\beta}_\#(Q)$ for all $\beta\in (0,1/(p-1))$. In turn, by \eqref{div eq} and standard elliptic regularity this implies that 
$\{h_n\}$ is  bounded in  $C^{2,\beta}_\#(Q)$ for all $\beta\in (0,1/(p-1))$ and thus  $h_n\to d$ in $C^{2,\beta}(Q)$ for all
such  $\beta$.
Since $(d,u_d)$ is a critical pair (see Remark~\ref{remark critical}), 
{
 $\frac{d}{ds}F(d+s(h_n-d), u_{d+s(h_n-d)})|_{s=0}=0$, and so
}
by \eqref{secvar} to reach a contradiction it is enough to show that for $n$ large
\begin{multline*}
\frac{d^2}{ds^2}F(d+s(h_n-d), u_{d+s(h_n-d)})|_{s=t}=\partial^2G(h_{n,t}, u_{h_{n,t}})[h_n-d]\\
- \int_{\Gamma_{h_{n,t}}}( W (E(u_{h_{n,t}})) +H_{h_{n,t}}^{\psi}) \,
\Div_{\Gamma_{h_{n,t}}}  \left( \frac{(D h_{n,t}, |D h_{n,t}|^2 )(h_{n,t}-d)^2} {(1 + |D h_{n,t}|)^{\frac32}} \circ \pi \right) \, d \H^2\\
+\e\frac{d^2}{ds^2}\mathcal{W}_p(d+s(h_n-d))|_{s=t}>0
\end{multline*}
for all $t\in (0,1)$, where $h_{n, t}:=d+t(h_n-d)$, $H_{h_{n,t}}^{\psi}$ is defined as in \eqref{defHpsi} with $h$ replaced by 
$h_{n,t}$, and
$$
 \mathcal{W}_p(h):=\int_{\Gamma_h}|H|^p\, d\H^2\,.
 $$
 To this purpose, note that  since $h_n\to d$ in $C^{2,\beta}$, by Lemma~\ref{lm:chepallequadrate}  we have
 $$
 \sup_{t\in (0,1)}\|W (E(u_{h_{n,t}})) +H^{\psi} -W_d\|_{L^\infty(\Gamma_{h_{n,t}})}\to 0
 $$
 as $n\to\infty$, where $W_d$ is the constant value of $W(E(u_d))$ on $\Gamma_d$ (see Remark~\ref{remark critical}). Therefore, also
 by Lemma~\ref{lm:c2alfa}, we deduce  that
 \begin{align*}
& \partial^2G(h_{n,t}, u_{h_{n,t}})[h_n-d]\\
& \quad- \int_{\Gamma_{h_{n,t}}}( W (E(u_{h_{n,t}})) +H^{\psi}_{h_{n,t}}) \,
\Div_{\Gamma_{h_{n,t}}}  \left( \frac{(D h_{n,t}, |D h_{n,t}|^2 )(h_{n,t}-d)^2} {(1 + |D h_{n,t}|)^{\frac32}} \circ \pi \right)  \, d \H^2\\
&=\partial^2G(h_{n,t}, u_{h_{n,t}})[h_n-d]\\
&\quad - \int_{\Gamma_{h_{n,t}}}( W (E(u_{h_{n,t}})) +H^{\psi}_{h_{n,t}}-W_d) \,
\Div_{\Gamma_{h_{n,t}}} \left( \frac{(D h_{n,t}, |D h_{n,t}|^2 )(h_{n,t}-d)^2} {(1 + |D h_{n,t}|)^{\frac32}} \circ \pi \right)  \, d \H^2\\
&\geq c_0\|h_n-d\|^2_{H^1_\#(Q)}-
C\|W (E(u_{h_{n,t}})) +H^{\psi}_{h_{n,t}} -W_d\|_{L^\infty(\Gamma_{h_{n,t}})}\|h_n-d\|^2_{H^1_\#(Q)}\geq
 \frac{c_0}{2}\|h_n-d\|^2_{H^1_\#(Q)}
\end{align*}
for $n$ large and for some constant $c_0>0$ independent of $n$,
{
 where we used the facts that 
\begin{align*}
\int_{\Gamma_{h_{n,t}}}\left\|\Div_{\Gamma_{h_{n,t}}}  \left( \frac{(D h_{n,t}, |D h_{n,t}|^2 )(h_{n,t}-d)^2} {(1 + |D h_{n,t}|)^{\frac32}} \circ \pi \right) \right\|  \, d \H^2 \le C \|h_n\|_{C^2_\#(Q)} \|h_n-d\|_{H^1_\#(Q)}^2
\end{align*}
and that $h_n\to d$ in $C^{2,\beta}(Q)$.
}

Since 
$$
\mathcal{W}_p(d+t(h_n-d))= t^p \int_Q\biggl|\Div\frac{Dh_n}{\sqrt{1+t^2|Dh_n|^2}}\biggr|^p\, dx=: f_n(t)\,,
$$
in order to conclude it is enough to show that $f''_n(t)\geq 0 $ for all $t\in (0,1)$.
Set
$$
g_n(x,t):=\biggl|\Div\frac{Dh_n(x)}{\sqrt{1+t^2|Dh_n(x)|^2}}\biggr|^2
$$
so that 
\begin{equation}\label{f2}
f''_n=\int_Q\Bigl[p(p-1)t^{p-2}g_n^{\frac p2}+p^2t^{p-1}g_n^{\tfrac{p-2}{2}}\partial_t g_n
+\tfrac{p}2t^p\bigl(\bigl(\tfrac{p}2-1\bigr)g_n^{\frac{p-4}{2}}(\partial_t g_n)^2+ 
g_n^{\frac{p-2}{2}}\partial_{tt} g_n\bigr)\Bigr ]\,dx\,.
\end{equation}
On the other hand, observe that
$$
g_n=\frac{|\Delta h_n|^2}{1+t^2|Dh_n|^2}+t^2\frac{|D^2h_n[Dh_n, Dh_n]|^2}{(1+t^2|Dh_n|^2)^3}
-2t\frac{D^2h_n[Dh_n, Dh_n]\Delta h_n}{(1+t^2|Dh_n|^2)^2}\,
$$
so that for  $n$ large 
$$
g_n\geq  \frac12|\Delta h_n|^2-C|D^2h_n||Dh_n|^2 \quad\text{and}\quad 
|\partial_t g_n|+|\partial_{tt}g_n|\leq C|D^2h_n||Dh_n|\,.
$$
  We then deduce from \eqref{f2} that there exist $C_0$, $C_1>0$ independent of $n$ and $t\in (0,1)$ such that 
 $$
 f''_n(t)\geq C_0\int_Q|\Delta h_n|^p\, dx-C_1\|Dh_n\|^p_\infty\int_Q|D^2h_n|^p\, dx\,.
 $$
 Since $\|Dh_n\|_\infty\to 0$, by Lemma~\ref{lm:morini} we conclude that the right-hand side in the above inequality
 is non-negative for $n$ large, thus concluding the proof of the proposition.
\end{proof}
Finally, we prove the main result of this section, namely, the asymptotic stability of the flat configuration (see Definition~\ref{def:asymptotic}).
\begin{theorem}\label{th:botto}
Under the assumptions of Theorem~\ref{th:boli}, $(d, u_d)$ is asymptotically stable.
\end{theorem}
\begin{proof} By  Proposition~\ref{prop:nocritici} there exists $\sigma>0$ such that
{
 if $h$ is a critical profile, with $|\Omega_h|=|\Omega_d|$ and $\|h-d\|_{W^{2,p}_\#(Q)}\leq \sigma$, then $h=d$. In view of Theorem~\ref{th:bonacini} we may take $\sigma$ so small that
\beq\label{eq:botto0}
F(d, u_d)<F(k,u_k)\qquad\text{for all $(k,u_k)\in X$ with $0<\|k-d\|_{W^{2,p}_\#(Q)}\leq \sigma$.}
\eeq 
 Since $(d, u_d)$ is Liapunov stable by Theorem~\ref{th:boli},  for every fixed $(h_0, u_{0})\in X$ with  $|\Omega_{h_0}|=|\Omega_d|$ and $\|h_0-d\|_{W^{2,p}_\#(Q)}\leq \de(\sigma)$,  we have  
 \beq\label{forallt}
 \|h(\cdot, t)-d\|_{W^{2,p}_\#(Q)}\leq \sigma \qquad\text{for all $t>0$.}
 \eeq
Here  $\de(\sigma)$ is the number given in Definition~\ref{def:liapunov}. 
We claim that
\beq\label{eq:botto}
F(h(\cdot, t), u_{h}(\cdot, t))\to F(d, u_d) \qquad\text{as $t\to +\infty$.}
\eeq
By Proposition~\ref{th:pci} there exists a sequence $\{t_n\}\subset(0, +\infty)\setminus Z_0$  such that $t_n\to+\infty$  and  $\{h(\cdot, t_n)\}$ converges to a critical profile in $W^{2,p}_\#(Q)$,  where $Z_0$ is the set in \eqref{ex0}. In view of the choice of $\sigma$ and by \eqref{forallt}, we conclude that $h(\cdot, t_n)\to d$ in $W^{2,p}_\#(Q)$. 
}

 In particular $F(h(\cdot, t_n), u_{h(\cdot, t_n)})\to F(d, u_d)$. In turn, by \eqref{ex0}, this implies that $F(h(\cdot, t), u_{h}(\cdot, t))\to F(d, u_d)$ as $t\to +\infty$, $t\not\in Z_0$. On the other hand, by \eqref{3000}  for $t\in Z_0$ we have that  $F(h(\cdot, t), u_{h}(\cdot, t))\leq F(h(\cdot, \tau), u_{h}(\cdot, \tau))$ for all $\tau<t$, $\tau\not\in Z_0$. Therefore 
 $$
 \limsup_{t\to+\infty, t\in Z_0}F(h(\cdot, t), u_{h}(\cdot, t))\leq F(d, u_d). 
 $$
 Recalling \eqref{eq:botto0}, we finally obtain \eqref{eq:botto}. 
{
In turn, reasoning as in the proof of Theorem~\ref{th:bonacinievol} (see \eqref{3001}), it follows from \eqref{eq:botto0} and \eqref{forallt} that for every sequence $\{s_n\}\subset(0,+\infty)$, with $s_n\to+\infty$, there exists a subsequence such that $\{h(\cdot, s_n)\}$ converges to $d$ in $W^{2,p}_\#(Q)$. This implies that $h(\cdot, t)\to d$ 
in $W^{2,p}_\#(Q)$ as $t\to+\infty$ and concludes the proof.
}
\end{proof}

\subsection{The two-dimensional case.}\label{subsec:2d} As remarked in the introduction, the arguments presented in the previous subsections 
apply to the two-dimensional version of \eqref{geq}, with $p=2$, studied in \cite{FFLM2}, with 
\beq\label{geq2d}
V=\Bigl((g_{\theta\theta}+g)k+W(E(u))-\e\bigl(k_{\sigma\sigma}+\frac12k^3\bigr)\Bigr)_{\sigma\sigma}\,.
\eeq
Here $V$ denotes  the outer normal velocity of $\Gamma_{h(\cdot, t)}$, $k$ is its curvature, $W(E(u))$ is the trace of
 $W(E(u(\cdot, t)))$ on $\Gamma_{h(\cdot, t)}$, with $u(\cdot, t)$ the elastic equilibrium in $\Om_{h(\cdot, t)}$,  under the conditions
that $Du(\cdot, y)$  is $b$-periodic and  $u(x,0)=e_0(x,0)$, for some $e_0>0$; and 
 $(\cdot)_\sigma$ stands for tangential differentiation along $\Gamma_{h(\cdot, t)}$. The constant $e_0>0$ measures the lattice mismatch between the elastic film and the (rigid) substrate. Moreover,  
 $g:[0,2\pi]\to (0,+\infty)$ is defined as
\beq\label{g}
g(\theta)=\psi(\cos\theta, \sin\theta)
\eeq
and is evaluated at $\arg(\nu(\cdot, t))$, where $\nu(\cdot, t)$ is the outer normal to $\Gamma_{h(\cdot, t)}$. The underlying energy functional is then given by
$$
F(h,u):=\int_{\Omega_h}W(E(u))\,dz+\int_{\Gamma_h}\Bigl(\psi(\nu)+\frac\e2k^2\Bigr)\, d\H^1\,.
$$
In the two-dimensional framework, given $b>0$, we search for for $b$-periodic solutions to \eqref{geq2d}. 

A local-in-time {\em b-periodic weak solution } to \eqref{geq2d} is  a function  $h\in H^1\bigl(0,T_0; H^{-1}_\#(0,b)\bigr)\cap L^\infty(0,T_0; H^2_\#(0,b))$ such that:
\begin{itemize}
\item[(i)] $\displaystyle (g_{\theta\theta}+g)k+W(E(u))-\e\bigl(k_{\sigma\sigma}+\frac12k^3\bigr)\in L^2(0,T_0; H^1_\#(0,b))$,
\item[(ii)] for almost every $t\in [0,T_0]$, 
$$
\frac{\partial h}{\partial t}=J\Bigl((g_{\theta\theta}+g)k+Q(E(u))-\e\bigl(k_{\sigma\sigma}+\frac12k^3\bigr)\Bigr)_{\sigma\sigma} \qquad 
\text{in $H^{-1}_\#(0,b)$.}
$$
\end{itemize}

Given $(h_0, u_0)$, with $h_0\in H^2_\#(0,b)$,  $h_0>0$, and $u_0$ the corresponding elastic equilibrium,  local-in-time existence of a {\em unique} weak solution with initial datum $(h_0, u_0)$ has been established in \cite{FFLM2}.
The Liapunov and asymptotic stability analysis of the flat configuration established in Subsections~\ref{sub:nc} and \ref{sub:c} extends to the two-dimensional case, where, in addition, the range of $d$'s under which \eqref{bo1} holds can be analytically determined  for isotropic elastic energies of the form
$$
W(\xi):=\mu|\xi|^2+\frac{\lambda}2(\operatorname*{trace} \xi)^2\,.
$$
In the above formula  the {\em Lam\'e coefficients} $\mu$ and $\lambda$ are chosen to satisfy the ellipticity conditions $\mu>0$ and $\mu+\lambda>0$, see \cite{FM09, Bo0}. The stability range of the flat configuration depends on $\mu$, $\lambda$, and the mismatch constant $e_0$ appearing in the Dirichlet condition $u(x,0)=e_0(x,0)$. For the reader's convenience, we recall the results. Consider the  
 {\em Grinfeld function} $K$ defined by
\begin{equation}\label{grinfeld-k}
K(y):=\max_{n\in\N}\frac{1}{n}J(ny)\,, \quad y\geq0\,,
\end{equation}
where
$$
J(y):=\frac{y+(3-4\nu_p)\sinh y\cosh y}{4(1-\nu_p)^2+y^2+(3-4\nu_p)\sinh^2y}\,,
$$
and $\nu_p$ is the {\em Poisson modulus} of the elastic material, i.e.,  
\begin{equation}\label{poisson}
\nu_p: =\frac{\lambda}{2(\lambda+\mu)}\,.
\end{equation}
It turns out that $K$ is strictly increasing and continuous,  $K(y)\leq Cy$, and  $\displaystyle\lim_{y\to +\infty}K(y)=1$, for some positive constant $C$.
We also set, as in the previous subsections,  
$$
G(h,u):=\int_{\Omega_h}W(E(u))\,dz+\int_{\Gamma_h}\psi(\nu)\, d\H^1\,.
$$
 Combining \cite[Theorem~2.9]{FM09} and \cite[Theorem~2.8]{Bo0} with the results of the previous subsection, we obtain the 2D asymptotic stability of the flat configuration.
\begin{theorem}\label{th:2dliapunov}
Assume $\partial^2_{11}\psi(0,1)>0$.   Let $\dloc:(0,+\infty)\to (0,+\infty]$ be defined as $\dloc (b):=+\infty$,  if $0<b\leq \frac{\pi}{4}\frac{(2\mu+\lambda)\partial^2_{11}\psi(0,1)}{e_0^2\mu(\mu+\lambda)}$, and  as the solution to 
\begin{equation}\label{mainminloc1}
K\Bigl(\frac{2\pi \dloc(b)}{b}\Bigr)=\frac{\pi}{4}\frac{(2\mu+\lambda)\partial^2_{11}\psi(0,1)}{e_0^2\mu(\mu+\lambda)}\frac1b\,,
\end{equation}
otherwise.
Then the second variation of $G$ at $(d, u_d)$ is positive definite, i.e.,
$$
 \partial^2G(d, u_d)[\vphi]>0\qquad\text{for all $\vphi\in  H_\#^{1}(0,b)\setminus\{0\}$, with }\int_0^b\vphi\, dx=0\,,
$$
 if and only if $0<d<\dloc(b)$.
In particular, for all  $d\in (0,\dloc(b))$  the flat configuration   $(d, u_d)$ is asymptotically stable.
\end{theorem}

\section{Appendix}
\subsection{Regularity results}
In this subsection we collect a few regularity results that have been used in the previous sections. We start with the following elliptic estimate, whose proof is essentially contained in \cite[Lemma 6.10]{FFLM2}.
\begin{lemma}\label{lm:chepallequadrate}
Let $M>0$, $c_0>0$ . Let $h_1$, $h_2\in C^{1,\alpha}_{\#}(Q)$ for some $\alpha\in (0,1)$, with $\|h_i\|_{C^{1,\alpha}_\#(Q)}\leq M$ and $h_i\geq c_0$, $=i=1,2$, and let $u_1$ and $u_2$ be the corresponding
elastic equilibria  in $\Om_{h_1}$ and $\Om_{h_2}$, respectively. 
Then, 
\beq\label{eq:chepalle}
\|E(u_1(\cdot, h_1(\cdot))-E(u_2(\cdot, h_2(\cdot))\|_{C^{1,\alpha}_\#(Q)}\leq C\|h_1-h_2\|_{C^{1,\alpha}_\#(Q)}
\eeq
for some constant $C>0$ depending only on $M$, $c_0$, and $\alpha$.
\end{lemma}
The following lemma is probably well-known to the experts, however for the reader's convenience we provide a proof.
\begin{lemma}\label{lm:weyl-fusco}
Let $p>2$, $u\in L^{\frac{p}{p-1}}(Q)$ such that 
$$
\int_Q u\,AD^2\vphi\, dx+\int_Qb\cdot D\vphi+\int_Q c\vphi\, dx=0\qquad\text{for all $\vphi\in C^\infty_{\#}(Q)$ with $\int_Q\vphi\, dx=0$,}
$$
where $A\in W^{1,p}_{\#}(Q; \mathbb{M}^{2\times 2}_{\rm sym})$ 
{
satisfies standard uniform ellipticity conditions (see \eqref{mor2} below), 
}
$b\in L^1(Q; \R^2)$, and $c\in L^1(Q)$. Then 
$u\in L^q(Q)$ for all $q\in (1,2)$. Moreover, if $b$, $u\,\Div A\in L^r(Q; \R^2)$ and $c\in L^r(Q)$ for some $r>1$, then $u\in W^{1,r}_\#(Q)$.
\end{lemma}
\begin{proof}
We only prove the first assertion, since the other one can be proven using similar arguments.
Denote by $A_\e$, $u_\e$, $b_\e$, and $c_\e$ the standard mollifications of
$A$, $u$, $b$, and $c$, and let $v_\e\in C^{\infty}_\#(Q)$ be the unique solution to the following problem
$$
\begin{cases}
\displaystyle \int_{Q}\bigl(A_\e Dv_\e+u_\e\Div A_\e-b_\e\bigr)\cdot D\vphi\, dx- \int_Qc_\e\vphi\, dx=0 & \text{for all $\vphi\in C^1_\#(Q)$\,, $\displaystyle\int_Q\vphi\, dx=0$,}\vspace{5pt}\\
\displaystyle \int_{Q} v_\e\, dx= \displaystyle \int_{Q} u\, dx\,.
\end{cases}
$$ 
Denoting by $G_\e$ the Green's function associated with the elliptic operator
$$
-\mathrm{div}\, (A_\e Du)
$$
it is known, \cite[equation (3.66)]{DK} and \cite[equation (1.6)]{GW}, that for all $q\in [1,2)$ and for all $x\in Q$ we have
$$
\|D_yG_\e(x, \cdot)\|_{L^q(Q)}\leq C\,,
$$
with $C$ depending only on the ellipticity constants and $q$ and not on $\e$. Since 
{
\begin{align*}
v_{\epsilon}(x)&=\int_{Q} G_{\epsilon}
(x,y)\bigl[-\Div(u_\e\Div A_\e-b_\e)+c_\epsilon\bigr]\,dy\\
&=\int_{Q}\bigl[\bigl(u_\e\Div A_\e-b_\e\bigr)\cdot D_yG_\e(x,y)+
 G_{\epsilon}
(x,y)c_\epsilon\bigr]\,dy\,,
\end{align*}
}
 it follows by standard properties of convolution that for all $q>1$ there exists $C>0$ depending only on $q$ and the $L^1$-norms  of $u_\e\Div A_\e$, $b_\e$, $c_\e$, hence
on the $L^1$-norms  of $b$, $c$, the $L^{\frac{p}{p-1}}$ norm of $u$, and the $W^{1,p}$ norm of $A$, such that $\|v_\e\|_{L^q(Q)}\leq C$ for $\e$ sufficiently small. Thus, we may assume (up to subsequences) that $v_\e\wto v$ weakly in  $L^q(Q)$, where $v$ solves 
\beq\label{weyl-fusco1}
  \int_{Q}vAD^2\vphi\, dx+ \int_{Q}\bigl(v\,\Div A-u\,\Div A+b\bigr)\cdot D\vphi\, dx+\int_Qc\vphi\, dx=0 
\eeq
  for all  $\vphi\in C^2_\#(Q)$, with $\int_Q\vphi\, dx=0$, and satisfies
\beq\label{weyl-fusco2}  
 \int_{Q} v\, dx= \displaystyle \int_{Q} u\, dx\,.
\eeq
Since by assumption $u$  solves the problem \eqref{weyl-fusco1}-\eqref{weyl-fusco2}, it is enough to show that the problem admits a unique solution. Let $v_1$ and $v_2$ be two solutions and set $w:=v_2-v_1$. Then, we have
\beq\label{weyl-fusco3}  
 \int_{Q}wAD^2\vphi\, dx+ \int_{Q} w\,\Div A\cdot D\vphi\, dx=0
\eeq
 for all  $\vphi\in C^2_\#(Q)$, with $\int_Q\vphi\, dx=0$.
 Let $g\in C^1_\#(Q)$, with $\int_Qg\, dx=0$ and denote by $\vphi_g$ the unique solution
 to the equation $\Div(A[D\vphi_g])=g$ such that  $\int_Q\vphi_g\, dx=0$. Hence, from \eqref{weyl-fusco3} we deduce that $\int_{Q}wg\, dx=0$ for all $g\in C^1_\#(Q)$, with $\int_Qg\, dx=0$.
 This implies that $w$ is constant and, in turn, $w\equiv 0$ since  $\int_Qw\, dx=0$.
\end{proof}
In the next lemma we denote by $Lu$ an elliptic operator of the form
\beq\label{mor1}
Lu:=\sum_{ij}a_{ij}(x)D_{ij}u+\sum_{i}b_i(x)D_iu\,,
\eeq
where all the coefficients are $Q$-periodic functions, the $a_{ij}$'s are continuous, and the $b_i$ are bounded. Moreover, there exist $\lambda$, $\Lambda>0$ such that
\beq\label{mor2}
\Lambda |\xi|^2\geq \sum_{ij}a_{ij}(x)\xi_i\xi_j\geq \lambda |\xi|^2\quad\text{for all $\xi\in \R^2$,} \qquad \sum_i|b_i|\leq \Lambda\,.
\eeq
\begin{lemma}\label{lm:morini}
Let $p\geq 2$. Then, there exists $C>0$ such that  for all $u\in W^{2,p}_\#(Q)$ we have
$$
\|D^2u\|_{L^p(Q)}\leq C\|L u\|_{L^p(Q)}\,,
$$
where $L$ is the differential operator defined in \eqref{mor1}. The constant $C$ depends only on $p$, $\lambda$, $\Lambda$ and the moduli of continuity of the coefficients $a_{ij}$.
\end{lemma}
\begin{proof}
We argue by contradiction assuming that  there exists a sequence  $\{u_h\}\subset W^{2,p}_\#(Q)$, a modulus of continuity $\omega$, and a sequence of operators $\{L_h\}$ as in \eqref{mor1}, with periodic coefficients $a^h_{ij}$, $b_i^h$ satisfying \eqref{mor2} and  
$$
|a^h_{ij}(x_1)-a^h_{ij}(x_2)|\leq \omega(|x_1-x_2|)
$$ 
for all $x_1$, $x_2\in Q$, such that
$$
\|D^2u_h\|_{L^p(Q)}\geq  h\|L_h u_h\|_{L^p(Q)}\,.
$$
By homogeneity we may assume that 
\beq\label{=1}
\|D^2u_h\|_{L^p(Q)}=1 \qquad \text{for all $h\in \N$.}
\eeq
Recall that by periodicity
$$
\int_Q Du_h\, dx=0\,.
$$
Moreover, by adding a constant if needed, we may also assume that $\int_Qu_h\, dx=0$.
Therefore, by Poincar\'e inequality and up to a subsequence, $u_h\wto u$ weakly in $W^{2,p}_\#(Q)$. Moreover, we may also assume that there exist $a_{ij}$ and $b_i$ satisfying \eqref{mor2}, such that
$$
a^h_{ij}\to a_{ij}\quad\text{uniformly in $Q$}\qquad\text{and}\qquad b_i^h\stackrel{*}{\wto}b_i
\quad\text{weakly* in $L^{\infty}(Q)$.}
$$
Since $\|L_h u_h\|_{L^p(Q)}\to 0$, we have that $u$ is a periodic function satisfying $Lu=0$, where $L$ is the operator associated with the coefficients $a_{ij}$ and $b_i$. Thus, by the Maximum Principle (\cite[Theorem 9.6]{GT}) $u$ is constant, and thus $u=0$. On the other hand, by elliptic regularity (see \cite[Theorem 9.11]{GT}) there exists a constant $C>0$ depending on $p$, $\lambda$, $\Lambda$, and $\omega$ such that
$$
\|D^2u_h\|_{L^p(Q)}\leq C( \|u_h\|_{W^{1,p}(Q)}+ \|L_h u_h\|_{L^p(Q)})\,.
$$
Since the right-hand side vanishes, we reach a contradiction to \eqref{=1}.
\end{proof}

\subsection{Interpolation results}
\begin{theorem}\label{th:A}
Let $\Om\subset\R^n$ be a bounded open set satisfying the cone condition. Let $1\leq p\leq\infty$ and $j$, $m$ be two integers such that $0\leq j\leq m$ and $m\geq 1$. Then there exists $C>0$ such that 
\begin{equation}\label{A}
\|D^jf\|_{L^p(\Om)}\leq C\bigl(\|D^{m}f\|_{L^p(\Om)}^{\frac{j}{m}}\|f\|_{L^p(\Om)}^{\frac{m-j}{m}}+\|f\|_{L^p(\Om)}\bigr)
\end{equation}
for all $f\in W^{m,p}(\Om)$. Moreover, if $\Om$ is a cube,   $f\in W^{m,p}_{\#}(\Om)$, and if either $f$ vanishes at the boundary or
$\int_\Om f\, dx =0$, 
then \eqref{A} holds in the stronger form
 \begin{equation}\label{Aper}
\|D^jf\|_{L^p(\Om)}\leq C\|D^{m}f\|_{L^p(\Om)}^{\frac{j}{m}}\|f\|_{L^p(\Om)}^{\frac{m-j}{m}}\,.
\end{equation}
\end{theorem}
\begin{proof}
Inequality \eqref{A} follows by combining inequalities (1) and (3) in \cite[Theorem 5.2]{AdamsFournier}. If $\Omega$ is a cube, $f$ is periodic and if either $f$ vanishes at the boundary or
$\int_\Om f\, dx =0$, then inequality \eqref{Aper} follows by observing that 
$$
\|f\|_{W^{m,p}(\Om)}\leq C\|D^mf\|_{L^p(\Om)}\,,
$$
as a straightforward application of the Poincar\'e inequality.
\end{proof}

The next interpolation result  is obtained by combining  \cite[Theorem 5.8]{AdamsFournier} with \eqref{A}.

\begin{theorem}\label{th:C}
Let $\Om\subset \R^n$ be a bounded open set satisfying the cone condition. If $mp>n$,  let $1\leq p\leq q\leq \infty$; if $mp=n$ let $1\leq p\leq q<\infty$; if $mp< n$ let $1\leq p\leq q\leq np/(n-mp)$.  Then there exists $C>0$ such that 
\begin{equation}\label{C}
\|f\|_{L^q(\Om)}\leq C\bigl(\|D^mf\|_{L^{p}(\Om)}^{\theta}\|f\|_{L^p(\Om)}^{1-\theta}+\|f\|_{L^p(\Om)}\bigr)
\end{equation}
for all $f\in W^{m,p}(\Om)$, where $\theta:=\frac{n}{mp}-\frac{n}{mq}$. Moreover, if  $\Om$ is a cube,
$f\in W^{m,p}_{\#}(\Om)$, and if
either $f$ vanishes at the boundary or $\int_{\Om} f\, dx=0$, then \eqref{C} holds in the stronger form
\begin{equation}\label{Cper}
\|f\|_{L^q(\Om)}\leq C\|D^mf\|_{L^{p}(\Om)}^{\theta}\|f\|_{L^p(\Om)}^{1-\theta}\,.
\end{equation}
\end{theorem}

Combining Theorems~\ref{th:A} and \ref{th:C}, and arguing as in the proof of 
\cite[Theorem 6.4]{FFLM2}, we have the following theorem.
\begin{theorem}\label{th:D}
Let $\Om\subset \R^n$ be a bounded open set satisfying the cone condition.  Let $s$, $j$, and $m$ be integers such that $0\leq s\leq j\leq m$. Let $1\leq p\leq q<\infty$ if  $(m-j)p\geq n$, and let  $1\leq p\leq q\leq\infty$ if  $(m-j)p> n$ . Then, there exists $C>0$ such that
\begin{equation}\label{eq:D}
\|D^jf\|_{L^q(\Om)}\leq C\bigl(\|D^mf\|_{L^{p}(\Om)}^{\theta}\|D^sf\|_{L^p(\Om)}^{1-\theta}+\|D^sf\|_{L^p(\Om)}\bigr)
\end{equation}
for all $f\in W^{m,p}(\Om)$, where
$$
\theta:=\frac{1}{m-s}\left(\frac{n}p-\frac{n}q+j-s\right)\,.
$$
Moreover, if $\Om$ is a cube,  $f\in W^{m,p}_{\#}(\Om)$, and if 
either $f$ vanishes at the boundary or $\int_{\Om} f\, dx=0$, then \eqref{eq:D} holds in the stronger form
\begin{equation}\label{Dper}
\|D^jf\|_{L^q(\Om)}\leq C\|D^mf\|_{L^{p}(\Om)}^{\theta}\|D^sf\|_{L^p(\Om)}^{1-\theta}\,.
\end{equation}
\end{theorem}
Finally, we conclude with an interpolation estimate involving the $H^{-1}$-norm, see Remark~\ref{rm:normah-1}.
\begin{lemma}\label{lm:H-1}
There exists $C>0$ such that for all  $f\in H^1_\#(Q)$, with $\int_Q f\, dx=0$, we have
$$
\|f\|_{L^2(Q)}\leq C\|Df\|_{L^2(Q)}^{\frac12}\|f\|^{\frac12}_{H^{-1}_\#(Q)}\,.
$$
Similarly, there exists $C>0$ such that for all $f\in H^2_\#(Q)$, with $\int_Q f\, dx=0$, we have
$$
\|f\|_{L^2(Q)}\leq C\|D^2f\|_{L^2(Q)}^{\frac13}\|f\|^{\frac23}_{H^{-1}_\#(Q)}\,.
$$
\end{lemma}
\begin{proof}
Let $w$ be the unique $Q$-periodic solution to 
$$
\begin{cases}
-\Delta w=f & \text{in }Q\,,\\
\int_Qw\, dx=0\,.
\end{cases}
$$
Combining Lemma~\ref{lm:morini} with \eqref{Aper} we obtain
\begin{align*}
\|f\|_{L^2(Q)}&=\|\Delta w\|_{L^2(Q)}\leq C\|D^2 w\|_{L^2(Q)}\leq C\|D^3w\|^{\frac12}_{L^2(Q)}\|Dw\|^{\frac12}_{L^2(Q)}\\
&\leq
C\|\Delta(Dw)\|^{\frac12}_{L^2(Q)}\|Dw\|^{\frac12}_{L^2(Q)}=C\|Df\|^{\frac12}_{L^2(Q)}\|f\|^{\frac12}_{H^{-1}_\#(Q)}\,.
\end{align*}
The second inequality of the statement is proven similarly.
\end{proof}

\section*{Acknowledgements}
The authors warmly thank the Center for Nonlinear Analysis (NSF Grants No.
DMS-0405343 and DMS-0635983), where part of this research was carried out. The
research of I.~Fonseca was partially funded by the National Science Foundation
under Grant No. DMS-0905778 and that of G. Leoni under Grant No. DMS-1007989.
I. Fonseca and G. Leoni also acknowledge support of the National Science
Foundation under the PIRE Grant No. OISE-0967140. 
The work of N. Fusco was supported by ERC under FP7 Advanced Grant n. 226234 and was partially carried on at the University of
Jyv\"askyl\"a under the FiDiPro Program.

\end{document}